\documentclass[11pt]{article}
\usepackage{amsmath,mathtools}
\usepackage{amssymb,amsbsy,amsthm,dsfont,bbm}
\usepackage[T1]{fontenc}
\usepackage{graphicx}
\usepackage[dvipsnames]{xcolor}
\usepackage{caption}
\usepackage{subcaption}
\usepackage{enumerate}
\usepackage{cases}
\usepackage[margin=1in]{geometry}
\usepackage{hyperref}
\usepackage{comment}
\hypersetup{
	colorlinks,
	citecolor=blue,
	filecolor=red,
	linkcolor=blue,
	urlcolor=blue
}

\allowdisplaybreaks[1]

\numberwithin{equation}{section}

\newcommand{\p}{\mathbb{P}} 
\newcommand{\E}{\mathbb{E}} 

\DeclareMathOperator{\Bin}{Bin} 

\newcommand{\maj}{\mathsf{maj}}

\newcommand\stleq{\stackrel{\mathclap{\normalfont\mbox{\small{st}}}}{\preceq}} 

\newcommand{\wh}{\widehat}  

\newcommand{\SBM}{\mathrm{SBM}} 
\newcommand{\CSBM}{\mathrm{CSBM}} 

\newcommand{\MAP}{\mathrm{MAP}} 

\newcommand{\overlap}{\mathsf{ov}}

\newcommand{\cut}[1]{}
\newcommand{\Luczak}{{\L}uczak~}

\newcommand{\dch}{\mathrm{D_{+}}}
\newcommand{\dchab}{\dch(\alpha,\beta)} 
\newcommand{\dconn}{\mathrm{T_{c}}}
\newcommand{\dconnab}{\dconn(\alpha,\beta)} 

\def\cA{{\mathcal A}}
\def\cB{{\mathcal B}}
\def\cC{{\mathcal C}}

\def\cE{{\mathcal E}}
\def\cF{{\mathcal F}}
\def\cG{{\mathcal G}}
\def\cH{{\mathcal H}}
\def\cI{{\mathcal I}}

\def\cM{{\mathcal M}}
\def\cN{{\mathcal N}}

\def\cS{{\mathcal S}}

\newtheorem{theorem}{Theorem}[section]
\newtheorem{lemma}[theorem]{Lemma}

\newtheorem{corollary}[theorem]{Corollary}

\newtheorem{definition}[theorem]{Definition}

\newtheorem{remark}[theorem]{Remark}

\usepackage{algorithm}
\usepackage{algpseudocode}

\makeatletter
\newenvironment{breakablealgorithm}
  {
   \begin{center}
     \refstepcounter{algorithm}
     \hrule height.8pt depth0pt \kern2pt
     \renewcommand{\caption}[2][\relax]{
       {\raggedright\textbf{\ALG@name~\thealgorithm} ##2\par}%
       \ifx\relax##1\relax 
         \addcontentsline{loa}{algorithm}{\protect\numberline{\thealgorithm}##2}%
       \else 
         \addcontentsline{loa}{algorithm}{\protect\numberline{\thealgorithm}##1}%
       \fi
       \kern2pt\hrule\kern2pt
     }
  }{
     \kern2pt\hrule\relax
   \end{center}
  }
\makeatother

\begin{document}

\title{Exact Community Recovery\\ in Correlated Stochastic Block Models}
\author{
    Julia Gaudio
    \thanks{Northwestern University; \url{julia.gaudio@northwestern.edu}.}
    \and
	Mikl\'os Z.\ R\'acz
	\thanks{Princeton University; \url{mracz@princeton.edu}. Research supported in part by NSF grant DMS 1811724.} 
	\and
	Anirudh Sridhar
	\thanks{Princeton University; \url{anirudhs@princeton.edu}. Research supported in part by NSF grant DMS 1811724.} 
}
\date{\today}

\maketitle


\begin{abstract}
We consider the problem of learning latent community structure from multiple correlated networks. We study edge-correlated stochastic block models with two balanced communities, focusing on the regime where the average degree is logarithmic in the number of vertices. Our main result derives the precise information-theoretic threshold for exact community recovery using multiple correlated graphs. This threshold captures the interplay between the community recovery and graph matching tasks. In particular, we uncover and characterize a region of the parameter space where exact community recovery is possible using multiple correlated graphs, even though (1) this is information-theoretically impossible using a single graph and (2) exact graph matching is also information-theoretically impossible. In this regime, we develop a novel algorithm that carefully synthesizes algorithms from the community recovery and graph matching literatures. 
\end{abstract}


\section{Introduction} \label{sec:intro} 

Recovering communities in networks is a fundamental learning task that has myriad applications in sociology, biology, and beyond. 
Increasingly, network data is supplemented with further data that is correlated with the underlying communities, such as latent feature vectors (e.g., the interests of individuals in a social network) or further correlated networks (e.g., personal and professional social networks overlap, yet contain complementary information). 
Synthesizing information from these different data sources presents an opportunity to obtain improved community recovery algorithms and guarantees, 
yet this comes with algorithmic and statistical challenges. 
In particular, integrating information from correlated networks 
is often hindered because the graphs are not aligned, 
due to node labels that are missing, erroneous, anonymized, or otherwise unknown. 
This highlights the importance of graph matching, 
which is an important learning task in its own right.

Recently, R\'acz and Sridhar~\cite{RS21} determined the information-theoretic limits for exact graph matching in edge-correlated stochastic block models, 
and as an application they showed how to exactly recover communities from two correlated graphs in a regime where it is impossible to do so using just a single graph. 
The main contribution of our work is to go \emph{beyond} exact graph matching, and we determine the precise information-theoretic threshold for exact community recovery from two correlated block models. 
In particular, we uncover and characterize a region of the parameter space where exact community recovery is possible despite exact graph matching being impossible (and exact community recovery from a single graph also being impossible), 
positively resolving a conjecture 
of~\cite{RS21}. 
To do so, we develop a novel algorithm that carefully synthesizes 
community recovery and graph matching algorithms. 
Overall, our work highlights the subtle interplay between community recovery and graph matching, two canonical and widely-studied learning problems.


\subsection{Community recovery in correlated stochastic block models}\label{sec:CSBM}

\textbf{The stochastic block model (SBM).} 
The SBM is the canonical probabilistic generative model for networks with community structure. Introduced by Holland, Laskey, and Leinhardt~\cite{HLL83}, it has received enormous attention over the past decades; in particular, it serves as a natural theoretical testbed for evaluating and comparing clustering algorithms on average-case networks (see, e.g.,~\cite{dyer1989solution, bui1984graph, bopanna1987eigenvalues}). 
The SBM allows a precise understanding of when community information can be extracted from network data, due to the fact that it exhibits sharp information-theoretic phase transitions for various inference tasks. 
Such phase transitions were first conjectured by Decelle~et~al.~\cite{DKMZ11} and were subsequently proven rigorously by several authors~\cite{mossel2014reconstruction,massoulie2014community,mossel2018proof, abbe2016exact,mossel2016consistency,abbe2015community,bordenave2015nonbacktracking, Abbe_survey}. 

Here we focus on the SBM with two symmetric communities, arguably the simplest setting.  
For a positive integer $n$ and $p,q \in [0,1]$, we construct the graph $G \sim \mathrm{SBM}(n,p,q)$ as follows. 
The graph $G$ has $n$ vertices, 
labeled by the elements of $[n] : = \{1, \ldots, n \}$. 
Each vertex $i \in [n]$ has a community label $\sigma_{*}(i) \in \{ + 1, - 1 \}$; 
these are drawn i.i.d.\ uniformly at random across all $i \in [n]$. 
The vector of community labels is denoted by 
$\boldsymbol{\sigma_{*}} : = \{ \sigma_{*}(i) \}_{i = 1}^n$, 
with the two communities given by the sets 
$V^{+} : = \{ i \in [n]: \sigma_{*}(i) = + 1 \}$ and 
$V^{-} : = \{ i \in [n] : \sigma_{*}(i) = -1 \}$. 
Given the community labels $\boldsymbol{\sigma_{*}}$, 
the edges of $G$ are drawn independently across vertex pairs as follows. 
For distinct $i,j \in [n]$, if $\sigma_{*}(i) \sigma_{*}(j) = 1$ (i.e., $i$ and $j$ are in the same community), then the edge $(i,j)$ is in $G$ with probability~$p$; else, $(i,j)$ is in $G$ with probability~$q$.

\textbf{Community recovery.} 
In this setting, a community recovery algorithm takes as input the graph $G$, without knowledge of the community labels $\boldsymbol{\sigma_{*}}$, and outputs a community labeling $\wh{\boldsymbol{\sigma}}$. 
The success of an algorithm is measured by the \emph{overlap} between the estimated labeling and the ground truth, defined as 
\[
\overlap(\wh{\boldsymbol{\sigma}}, \boldsymbol{\sigma_{*}}) : = \frac{1}{n} \left|  \sum\limits_{i = 1}^n \wh{\sigma}(i) \sigma_{*}(i)  \right|.
\]
We take an absolute value in this formula since the labelings $\boldsymbol{\sigma_{*}}$ and $-\boldsymbol{\sigma_{*}}$ specify the same partition of communities, and it is only possible to recover $\boldsymbol{\sigma_{*}}$ up to its sign. Observe that 
$\overlap(\wh{\boldsymbol{\sigma}}, \boldsymbol{\sigma_{*}}) \in [0,1]$, 
with a larger value corresponding to a better estimate. 
In particular, the algorithm succeeds in exactly recovering the communities 
(i.e., $\wh{\boldsymbol{\sigma}} = \boldsymbol{\sigma_{*}}$ or $\wh{\boldsymbol{\sigma}} = - \boldsymbol{\sigma_{*}}$) 
if and only if 
$\overlap(\wh{\boldsymbol{\sigma}}, \boldsymbol{\sigma_{*}}) = 1$.

In the logarithmic degree regime---that is, when $p = \alpha \log (n) / n$ and $q = \beta \log (n) / n$ for some fixed constants $\alpha , \beta \ge 0$---it is well-known that there is a sharp information-theoretic threshold for exactly recovering communities in the SBM~\cite{abbe2016exact, mossel2016consistency,abbe2015community,Abbe_survey}. 
This is governed by the quantity 
\begin{equation}\label{eq:DCH_def}
\dchab := \left( \frac{\sqrt{\alpha} - \sqrt{\beta}}{\sqrt{2}} \right)^{2} = \frac{\alpha + \beta}{2} - \sqrt{\alpha \beta}. 
\end{equation}
In the general setting, this quantity is known as the \emph{Chernoff-Hellinger divergence}~\cite{abbe2015community,Abbe_survey}; in the specific setting of two balanced communities, it simplifies to the Hellinger divergence of the vectors $(\alpha/2,\beta/2)$ and $(\beta/2,\alpha/2)$, giving~\eqref{eq:DCH_def}. 
The information-theoretic threshold for exact community recovery is then given by 
\begin{equation}\label{eq:DCH_threshold}
\dchab = 1.
\end{equation}
If $\dchab > 1$, then exact community recovery is possible: 
there is a polynomial-time algorithm which outputs an  
estimator~$\wh{\boldsymbol{\sigma}}$ satisfying 
$\lim_{n \to \infty} \p( \overlap ( \wh{\boldsymbol{\sigma}}, \boldsymbol{\sigma_{*}}) = 1) = 1$. 
Moreover, 
if $\dchab < 1$, then 
this 
is impossible: 
for \emph{any} estimator $\widetilde{\boldsymbol{\sigma}}$, 
we have that 
$\lim_{n \to \infty} \p (\overlap(\widetilde{\boldsymbol{\sigma}}, \boldsymbol{\sigma_{*}}) = 1) = 0$.

\textbf{Correlated SBMs.} 
The goal of our work is to understand how the exact community recovery threshold given by~\eqref{eq:DCH_threshold} \emph{changes} when the input data consists of multiple \emph{correlated} SBMs. 
To this end, we study a natural model of correlated SBMs, which we describe next. 

We construct $(G_{1}, G_{2}) \sim \CSBM(n,p,q,s)$ as follows, 
where the additional parameter $s \in [0,1]$ controls the level of correlation between the two graphs. 
First, generate a parent graph $G \sim \SBM(n,p,q)$, 
and let $\boldsymbol{\sigma_{*}}$ denote the community labels. 
Next, given $G$, we construct $G_{1}$ by independent subsampling: 
each edge of $G$ is included in $G_{1}$ with probability $s$, 
independently of everything else, 
and non-edges of $G$ remain non-edges in $G_{1}$. 
We obtain a second graph, $G_{2}'$, independently in the same way. 
The graphs $G_{1}$ and $G_{2}'$ inherit both the vertex labels and the community labels $\boldsymbol{\sigma_{*}}$ from the parent graph $G$.  
Finally, we let $\pi_{*}$ be a uniformly random permutation of $[n]$, independently of everything else, and generate $G_{2}$ by relabeling the vertices of $G_{2}'$ according to $\pi_{*}$ (e.g., vertex $i$ in $G_{2}'$ is relabeled to $\pi_{*}(i)$ in $G_{2}$). 
This last step in the construction of $G_{2}$ reflects the fact that in applications, node labels are often obscured. 
To emphasize the effect of the vertex relabeling on the community labels, 
we define 
$\boldsymbol{\sigma_{*}^{1}} := \boldsymbol{\sigma_{*}}$ 
and 
$\boldsymbol{\sigma_{*}^{2}} := \boldsymbol{\sigma_{*}} \circ \pi_{*}^{-1}$, 
which are the community labels in~$G_{1}$ and $G_{2}$, respectively. 
This construction is visualized in Figure~\ref{fig:correlated_sbm}.

\begin{figure}[t]
    \centering
    \includegraphics[width=0.95\textwidth]{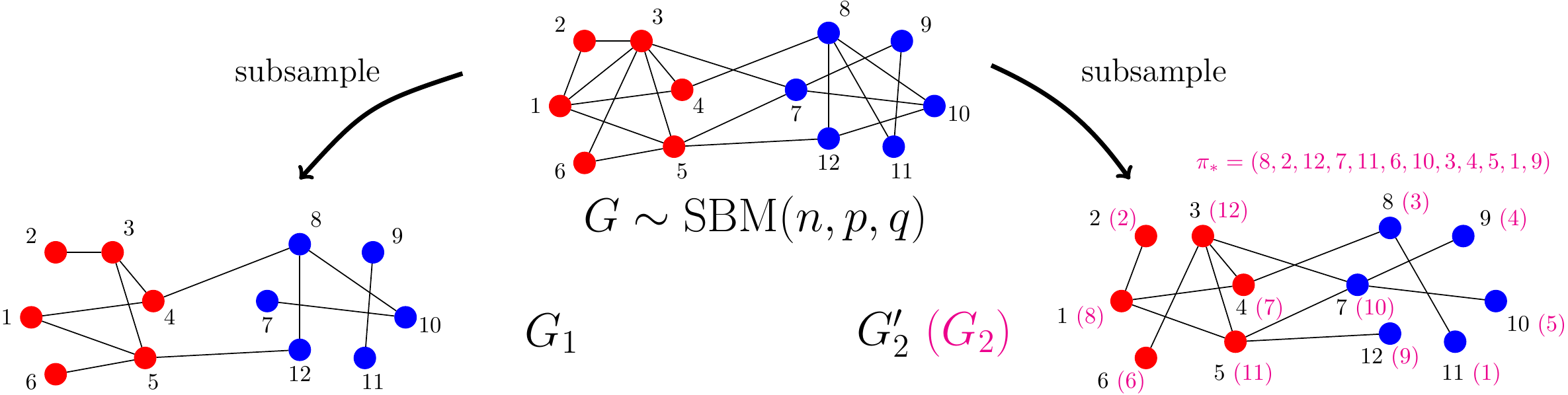}
    \caption{Schematic showing the construction of correlated SBMs (see text for details).}
    \label{fig:correlated_sbm}
\end{figure}

First studied by Onaran, Garg, and Erkip~\cite{onaran2016optimal}, this model of correlated SBMs is the natural generalization of correlated Erd\H{o}s-R\'enyi random graphs, which were introduced by Pedarsani and Grossglauser~\cite{pedarsani2011privacy} (see Section~\ref{sec:related} for discussion of further related work). 
In particular, marginally $G_{1}$ and $G_{2}$ are both SBMs. 
Specifically, since the subsampling probability is $s$, we have that $G_{1} \sim \SBM(n,ps,qs)$. 
Therefore, from~\eqref{eq:DCH_threshold} it follows that, 
in the logarithmic degree regime where 
$p = \alpha \log (n) / n$ and $q = \beta \log (n) / n$, 
the communities can be exactly recovered from $G_{1}$ \emph{alone} 
if $\dch(\alpha s, \beta s) > 1$. 
Since $\dch(\alpha s, \beta s) = s \dchab$, 
this condition is equivalent to $s \dchab > 1$, which we can also write as $\dchab > 1/s$. 

The central question of our work is how to go \emph{beyond} this single-graph threshold by incorporating the information in $G_{2}$. 
This question was initiated in recent work of R\'acz and Sridhar~\cite{RS21}, whose starting observation was the following. 
If $\pi_{*}$ were \emph{known}, 
then we can reconstruct $G_{2}'$ from $G_{2}$, 
and then ``overlay'' $G_{1}$ and $G_{2}'$ to obtain a new graph $H_{*}$ that combines the information in the two graphs. In particular, 
$(i,j)$ is an edge in $H_{*}$ if and only if $(i,j)$ is an edge in the parent graph $G$ and it is included in either $G_{1}$ or $G_{2}'$ in the subsampling process. 
It thus follows that $H_{*}$ is also an SBM, specifically, 
$H_* \sim \mathrm{SBM} \left( n, \alpha (1 - (1 - s)^2) \log(n)/n, \beta (1 - (1 - s)^2 ) \log(n)/n \right)$. 
In particular, this argument implies that if $\pi_{*}$ were \emph{known} 
and 
\begin{equation*}\label{eq:interesting_regime}
\frac{1}{1-(1-s)^{2}} < \dchab < \frac{1}{s}, 
\end{equation*}
then it is information-theoretically impossible to exactly recover $\boldsymbol{\sigma_{*}}$ from $G_{1}$ alone, 
but one can recover $\boldsymbol{\sigma_{*}}$ exactly by combining information from $G_{1}$ \emph{and} $G_{2}$.

\textbf{Graph matching.} 
Since $\pi_{*}$ is not known, the argument above raises the question of when can $\pi_{*}$ be exactly recovered from $(G_{1}, G_{2})$, 
a task known as \emph{graph matching}. 
The main result of R\'acz and Sridhar~\cite{RS21} answers this question (see also Section~\ref{sec:related} for discussion of related work). 
Specifically, they show that the information-theoretic threshold for exactly recovering $\pi_{*}$ is given by $s^{2} \left( \alpha + \beta \right) / 2 = 1$. 
Note that this is precisely the \emph{connectivity threshold} 
for the \emph{intersection graph} of $G_{1}$ and $G_{2}'$ (the edges of this intersection graph are the edges present in the parent graph $G$ that survived both subsampling processes). 
Letting 
\begin{equation}\label{eq:connectivity_threshold_def}
\dconnab := \frac{\alpha+\beta}{2}
\end{equation}
denote the connectivity threshold in $\SBM(n, \alpha \log(n)/n, \beta \log(n)/n)$, 
we can write the threshold for exactly recovering $\pi_{*}$ as 
$\dconn ( \alpha s^{2}, \beta s^{2} ) = 1$, 
or equivalently, 
$s^{2} \dconnab = 1$. 
Thus, if $s^{2} \dconnab > 1$, then $\pi_{*}$ can be exactly recovered from $(G_{1}, G_{2})$, 
while if $s^{2} \dconnab < 1$, then this is impossible. 

To summarize the two previous paragraphs, R\'acz and Sridhar~\cite{RS21} showed that if 
\begin{equation}\label{eq:prior_sufficient_condition}
s^{2} \dconnab > 1 
\qquad \qquad 
\text{ and }
\qquad \qquad 
\left( 1 - (1-s)^{2} \right) \dchab > 1, 
\end{equation}
then exact community recovery is possible, that is, 
it is possible to exactly recover $\boldsymbol{\sigma_{*}}$ using $(G_{1}, G_{2})$.

\textbf{The interplay between community recovery and graph matching.} 
The work of R\'acz and Sridhar~\cite{RS21} leaves open the question of what happens when exact graph matching is impossible. 
In particular, is there a parameter regime where exact community recovery is possible from $(G_{1}, G_{2})$, 
even though 
(1) this is information-theoretically impossible using a single graph and 
(2) exact graph matching is also impossible? 

We answer this question affirmatively, developing an algorithm that 
carefully combines community recovery and graph matching steps. 
Moreover, we determine the precise information-theoretic threshold for when exact community recovery is possible. 
If $\left( 1 - (1-s)^{2} \right) \dchab > 1$, 
then this threshold is given by 
\begin{equation}\label{eq:combined_threshold}
\dconn \left(\alpha s^{2}, \beta s^{2} \right) + \dch \left( \alpha s(1-s), \beta s(1-s) \right) = 1.
\end{equation}
The threshold in~\eqref{eq:combined_threshold} cleanly showcases the interplay between community recovery and graph matching: 
the first term in~\eqref{eq:combined_threshold} comes from graph matching, while the second term comes from community recovery.  
We now turn to describing our results formally. 

\subsection{Results}\label{sec:results}

We determine the information-theoretic threshold for exact community recovery 
from two correlated stochastic block models 
$(G_{1}, G_{2}) \sim \CSBM(n, \alpha \log(n) / n, \beta \log(n) / n, s)$. 
This result has two parts 
and we start with the positive one.

\begin{theorem} 
\label{thm:comm_recovery}
Fix constants $\alpha, \beta > 0$ and $s \in [0,1]$. 
Let $(G_1, G_2) \sim \mathrm{CSBM}\left( n,\frac{\alpha \log n}{n} ,\frac{\beta \log n}{n}, s \right)$.  
Suppose that 
\begin{equation}
\label{eq:community_achievability}
\left(1 - (1-s)^{2} \right) \dchab > 1 
\end{equation}
and that 
\begin{equation}
\label{eq:tradeoff_achievability}
s^{2} \dconnab + s(1-s) \dchab > 1. 
\end{equation}
Then there is an estimator $\wh{\boldsymbol{\sigma}} = \wh{\boldsymbol{\sigma}}(G_1, G_2)$ such that 
$
\lim\limits_{n \to \infty} \p \left( \overlap \left(\wh{\boldsymbol{\sigma}}, \boldsymbol{\sigma_{*}} \right) = 1 \right) = 1.
$
\end{theorem}
In the prior work~\cite{RS21} it was shown that~\eqref{eq:community_achievability} is necessary for exact community recovery, 
and that the conditions in~\eqref{eq:prior_sufficient_condition} suffice. 
As described in Section~\ref{sec:CSBM}, \cite{RS21} focused on determining the exact graph matching threshold and then using community recovery algorithms as a black box. 
The main contribution of Theorem~\ref{thm:comm_recovery} is to go \emph{beyond} 
exact graph matching, to showcase how exact community recovery is possible from $(G_{1},G_{2})$ even in regimes where (1) this is impossible from $G_{1}$ alone and (2) exact graph matching is impossible. 
This necessitates developing algorithms that combine information from $G_{1}$ and $G_{2}$ in more delicate ways, 
integrating ideas from community recovery and graph matching algorithms. 
Indeed, at a high level, the algorithm we develop to prove Theorem~\ref{thm:comm_recovery} has four main steps: 
\begin{enumerate}[(1)]
\item\label{step:matching} Obtain a partial almost exact graph matching $\wh{\mu}$ between $G_{1}$ and~$G_{2}$; 
\item\label{step:labels1} Obtain an almost exact community labeling of vertices in $G_{1}$; 
\item\label{step:correcting} For vertices in $G_{1}$ that are part of the matching $\wh{\mu}$: refine the almost exact labeling obtained in Step~\eqref{step:labels1} via a majority vote in the (denser) graph consisting of edges that are either in $G_1$ or in $G_2$ (determined using $\widehat{\mu}$).
\item\label{step:final_maj} For vertices in $G_{1}$ that are not part of the matching $\wh{\mu}$: classify them according to a majority vote of the labels of their neighbors, where we use only the edges in $G_{1}$. 
\end{enumerate}
In order to make such an algorithm work, the devil is in the details, with careful choices in each step; we refer to Section~\ref{sec:overview} for a more detailed overview of the algorithm. 

The threshold in~\eqref{eq:tradeoff_achievability} highlights the interplay between the community recovery and graph matching tasks. 
Indeed, the first term in~\eqref{eq:tradeoff_achievability} comes from graph matching: $s^{2} \dconnab = 1$ is the threshold for exact graph matching; moreover, when $s^{2} \dconnab < 1$, the best possible almost exact graph matching makes $n^{1 - s^{2} \dconnab + o(1)}$ errors, which is relevant for Step~\eqref{step:matching} of the algorithm. 
On the other hand, the second term in~\eqref{eq:tradeoff_achievability} comes from community recovery;  
in particular, this term arises from the majority vote in Step~\eqref{step:final_maj}. 
Note that while we use all edges in $G_{1}$ for this step, 
the unmatched nodes are isolated in the intersection graph, 
and hence the relevant edges are not present in $G_{2}$, 
leading to the ``effective'' factor of $s(1-s)$. 
Since exact community recovery in 
$\SBM(n, \alpha s (1-s) \log(n) / n, \beta s(1-s) \log(n) / n)$ 
is governed by the quantity 
$\dch ( \alpha s (1-s), \beta s (1-s) ) = s(1-s) \dchab$, 
this leads to the second term in~\eqref{eq:tradeoff_achievability}.

As the following impossibility result shows, Theorem~\ref{thm:comm_recovery} is tight. 

\begin{theorem}
\label{thm:comm_recovery_impossibility} 
Fix constants $\alpha, \beta > 0$ and $s \in [0,1]$. 
Let $(G_1, G_2) \sim \mathrm{CSBM}\left( n,\frac{\alpha \log n}{n} ,\frac{\beta \log n}{n}, s \right)$.  
Suppose that 
\begin{equation}
\label{eq:community_impossibility}
\left(1 - (1-s)^{2} \right) \dchab < 1 
\end{equation}
or that 
\begin{equation}
\label{eq:tradeoff_impossibility}
s^{2} \dconnab + s(1-s) \dchab < 1. 
\end{equation} 
Then for any estimator $\widetilde{\boldsymbol{\sigma}} = \widetilde{\boldsymbol{\sigma}}(G_1, G_2)$, we have that
$
\lim\limits_{n \to \infty} \p ( \overlap( \widetilde{\boldsymbol{\sigma}} ,\boldsymbol{\sigma_{*}}) = 1) = 0.
$
\end{theorem}

Impossibility of exact community recovery under the condition~\eqref{eq:community_impossibility} was shown in~\cite{RS21}, 
so the contribution of Theorem~\ref{thm:comm_recovery_impossibility} is to show impossibility under the condition~\eqref{eq:tradeoff_impossibility}. 

As discussed above, the condition~\eqref{eq:tradeoff_impossibility} highlights the interplay between the community recovery and graph matching tasks. In particular, Theorem~\ref{thm:comm_recovery_impossibility} uncovers and characterizes a region of the parameter space where exact community recovery from $(G_{1}, G_{2})$ is impossible, despite the fact that if $\pi_{*}$ were known, then exact community recovery would be possible from the correctly matched union graph $G_{1} \vee_{\pi_{*}} G_{2}$.

Putting together Theorems~\ref{thm:comm_recovery} and~\ref{thm:comm_recovery_impossibility}, 
we obtain the information-theoretic threshold for exact community recovery in correlated SBMs, see~\eqref{eq:combined_threshold}. 
These results are illustrated in the phase diagrams 
of Figures~\ref{fig:phase_new_diagram_fixed_s} and~\ref{fig:phase_new_diagram_fixed_beta}.

\begin{figure}
    \centering
    \begin{subfigure}{0.3\textwidth}
        \centering
        \includegraphics[width=\textwidth]{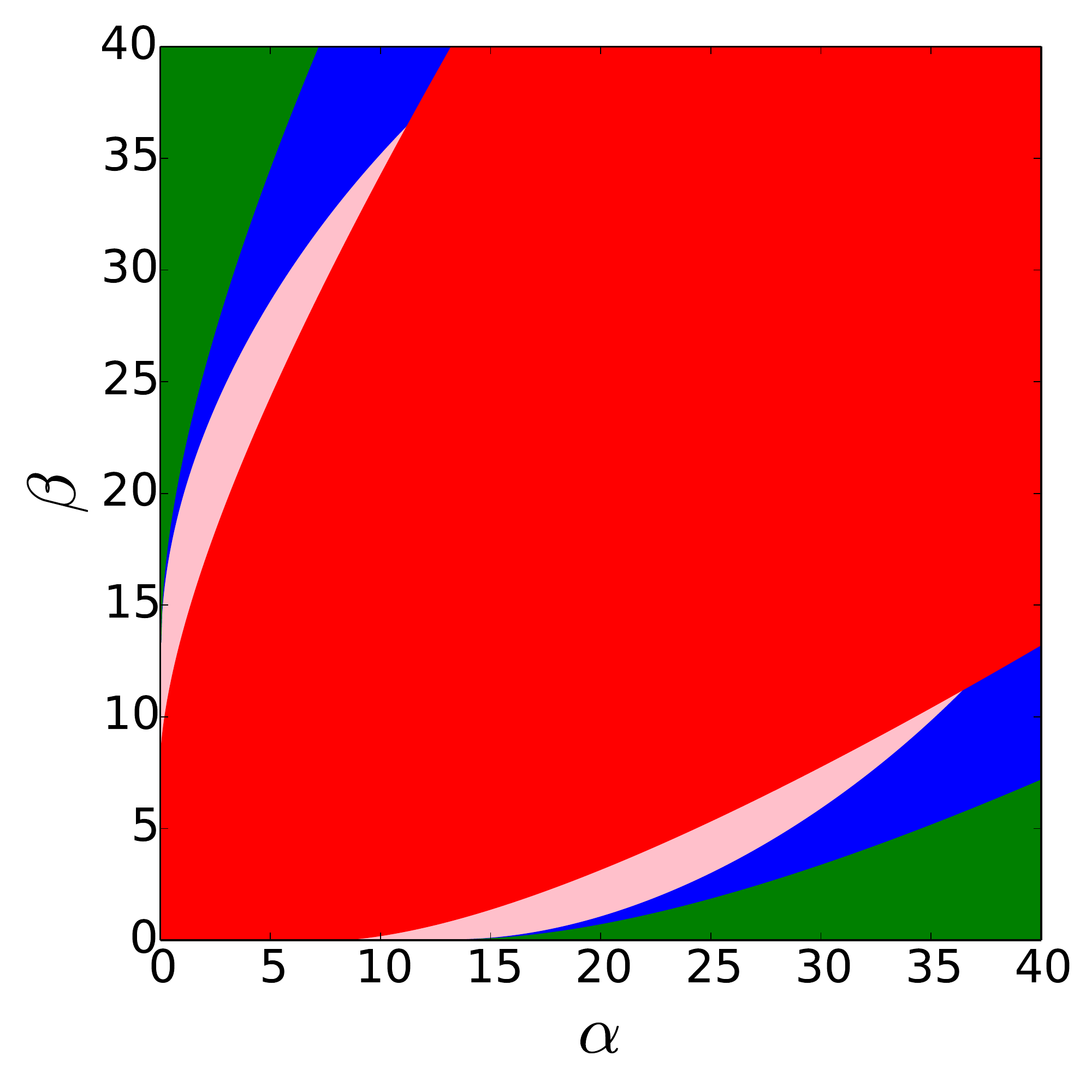}
        \caption{Fixed $s=0.15$.}
    \end{subfigure}
    \quad 
    \begin{subfigure}{0.3\textwidth}
        \centering
        \includegraphics[width=\textwidth]{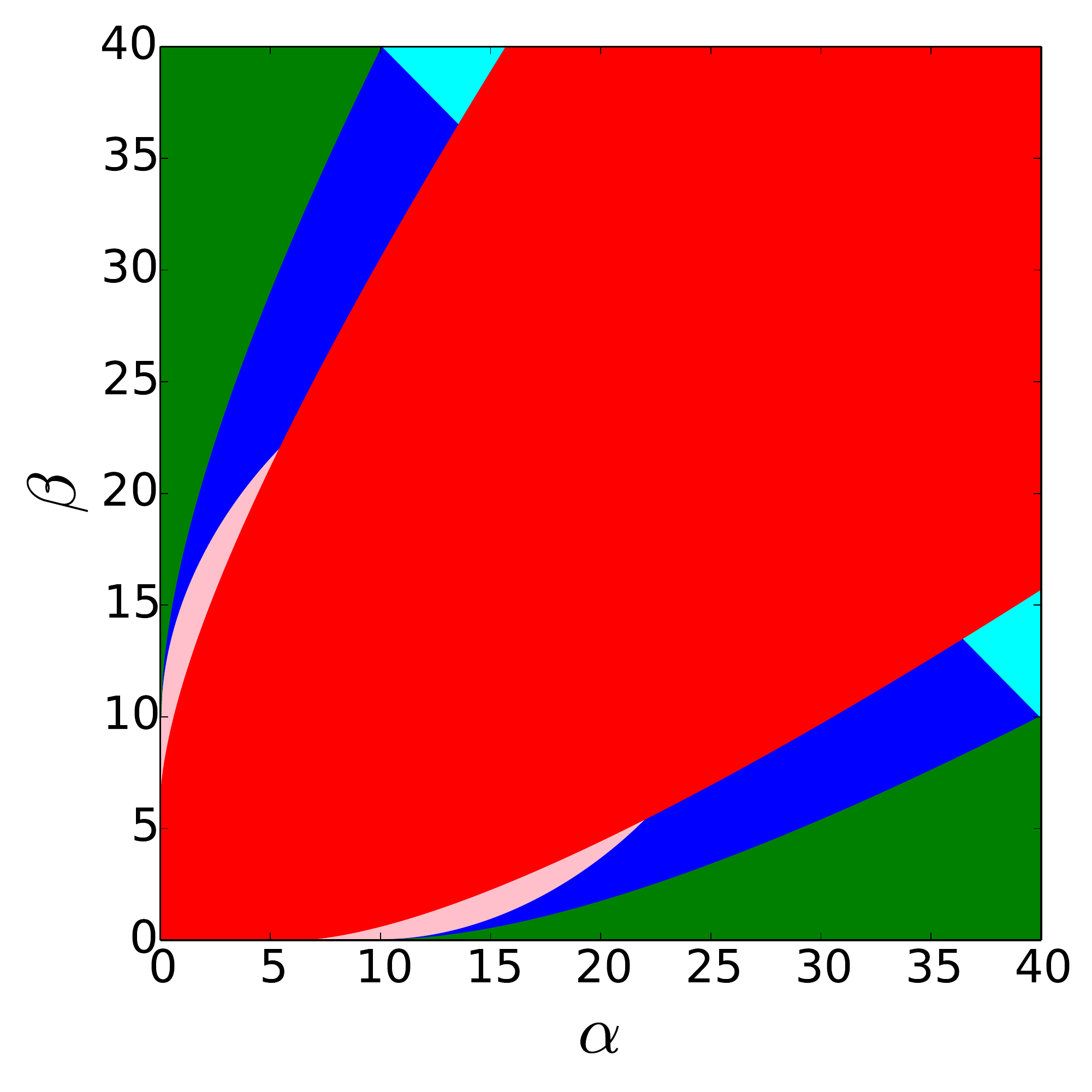}
        \caption{Fixed $s=0.2$.}
    \end{subfigure}
    \quad 
    \begin{subfigure}{0.3\textwidth}
        \centering
        \includegraphics[width=\textwidth]{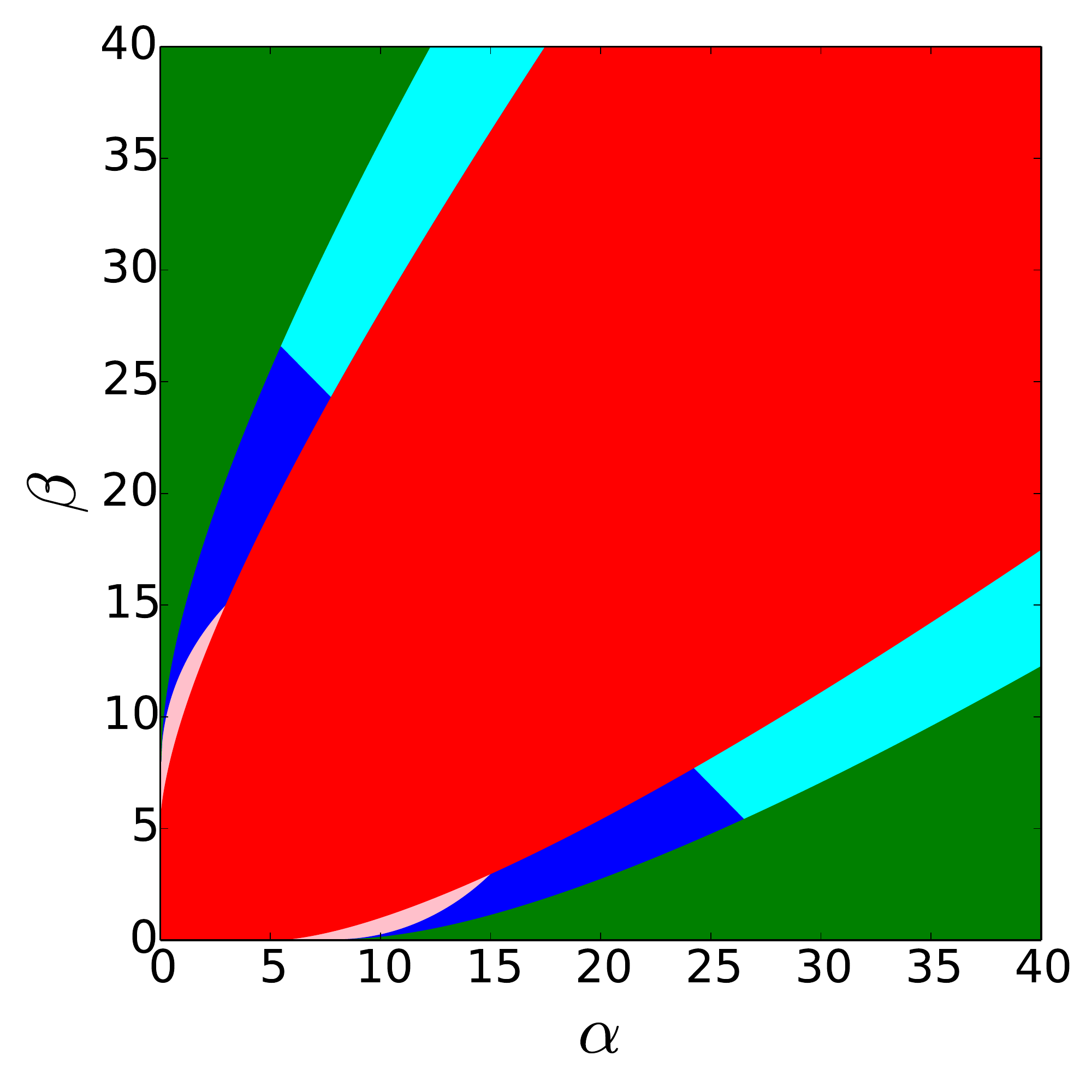}
        \caption{Fixed $s=0.25$.}
    \end{subfigure}
    \caption{Phase diagram for exact community recovery for fixed $s$, with $\alpha \in [0,40]$ and $\beta \in [0,40]$ on the axes. 
    \emph{Green region:} exact community recovery is possible from $G_{1}$ alone; 
    \emph{Cyan~region:} exact community recovery is impossible from $G_{1}$ alone, but exact graph matching is possible, and subsequently exact community recovery is possible from $(G_{1},G_{2})$; 
    \emph{Dark Blue region:} exact community recovery is impossible from $G_{1}$ alone, exact graph matching is also impossible, yet exact community recovery is nonetheless possible from $(G_{1},G_{2})$; 
    \emph{Pink region:} exact community recovery is impossible from $(G_{1},G_{2})$ (even though it would be possible if $\pi_{*}$ were known).
    \emph{Red region:} exact community recovery is impossible from $(G_{1},G_{2})$ (even if $\pi_{*}$ is known). 
    Characterizing the Dark Blue and Pink regions is the main result of this paper.}
    \label{fig:phase_new_diagram_fixed_s}
\end{figure}

\begin{figure}[t]
    \centering
    \begin{subfigure}{0.3\textwidth}
        \centering
        \includegraphics[width=\textwidth]{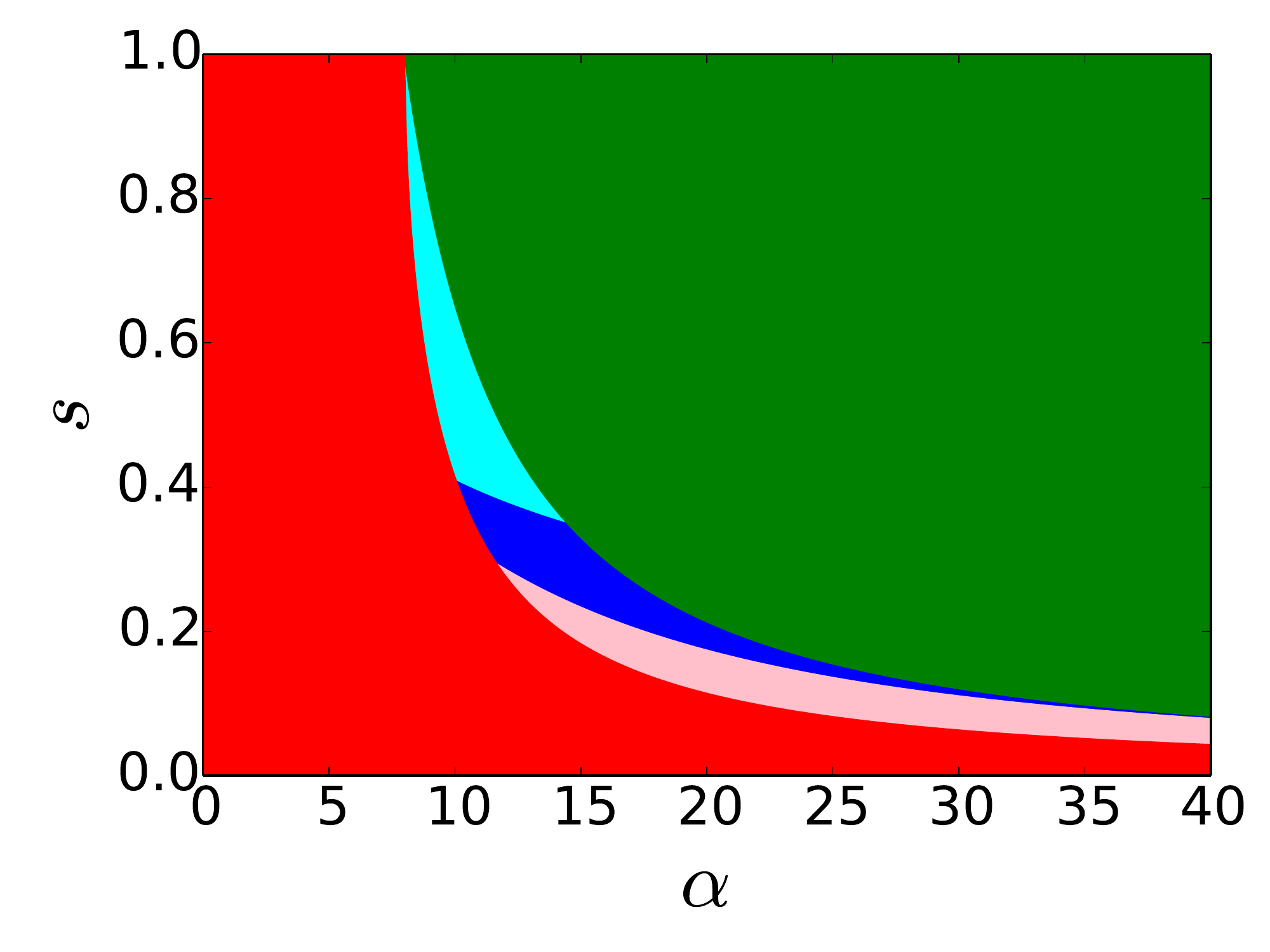}
        \caption{Fixed $\beta=2$.}
    \end{subfigure}
    \quad 
    \begin{subfigure}{0.3\textwidth}
        \centering
        \includegraphics[width=\textwidth]{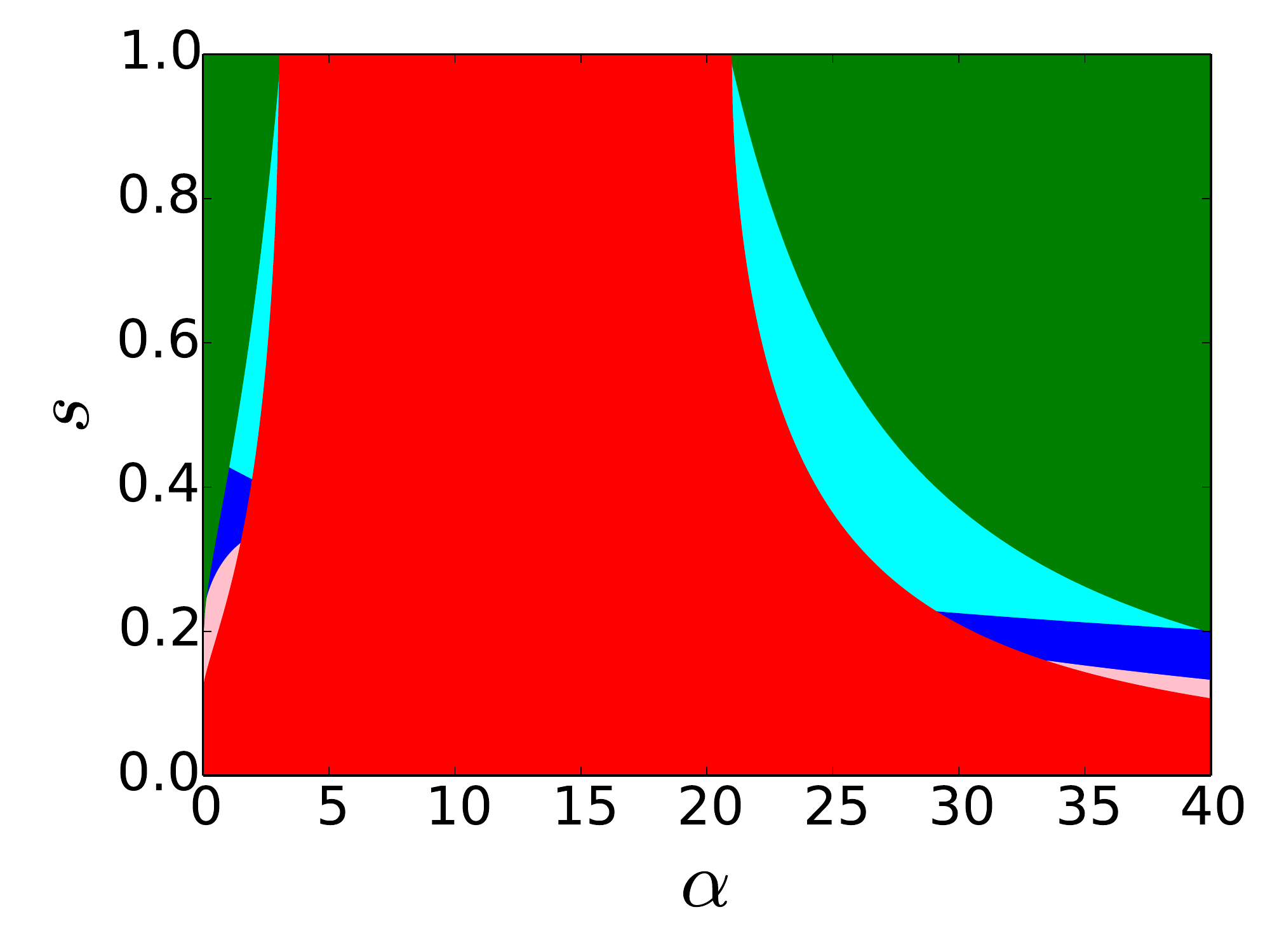}
        \caption{Fixed $\beta = 10$.}
    \end{subfigure}
    \quad 
    \begin{subfigure}{0.3\textwidth}
        \centering
        \includegraphics[width=\textwidth]{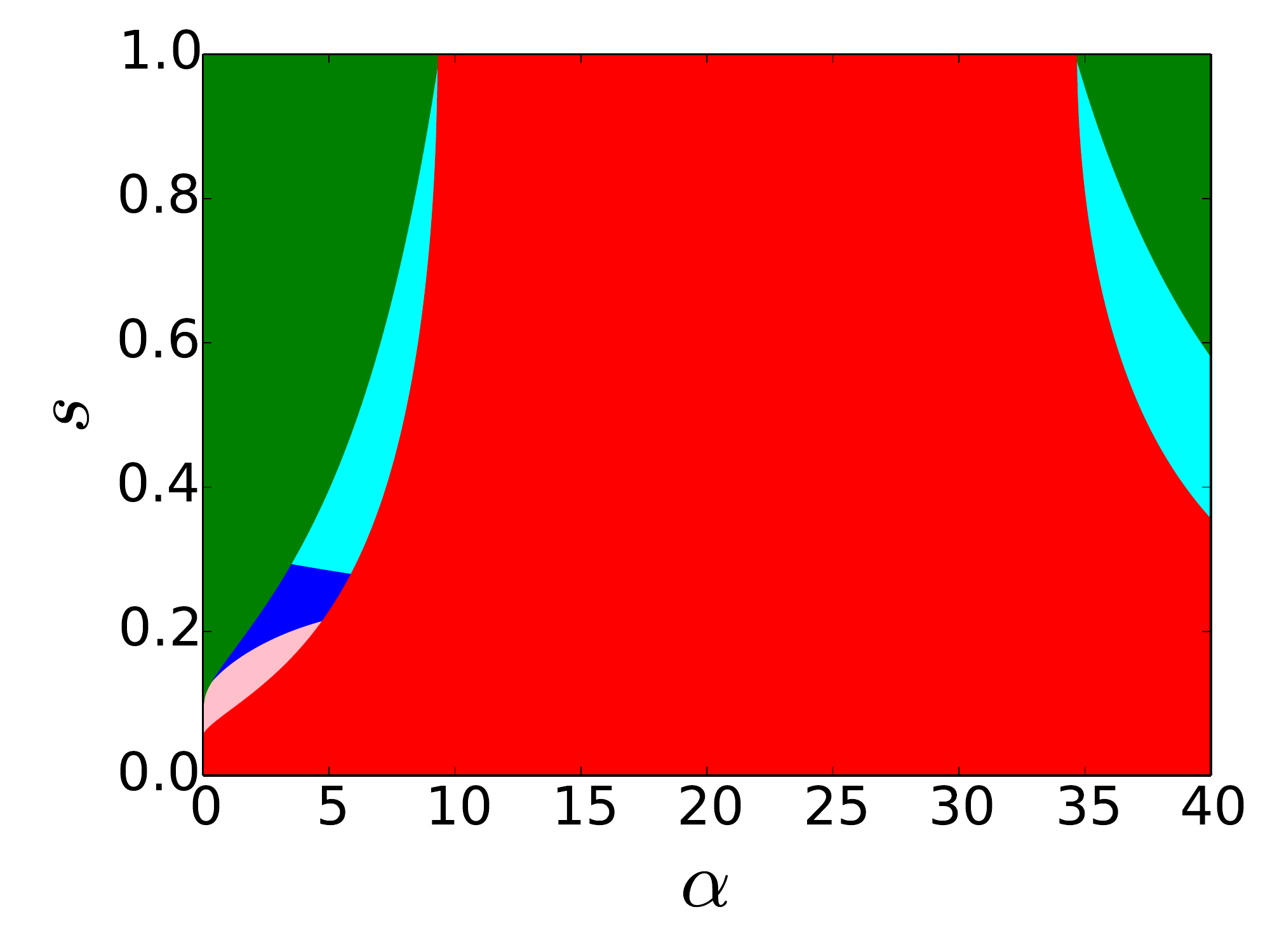}
        \caption{Fixed $\beta = 20$.}
    \end{subfigure}
    \caption{Phase diagrams for exact community recovery for fixed $\beta$, with $\alpha \in [0,40]$ and $s \in [0,1]$ on the axes. (Colors as in  Fig.~\ref{fig:phase_new_diagram_fixed_s}.)}
    \label{fig:phase_new_diagram_fixed_beta}
\end{figure}

\subsection{Overview of algorithms and proofs}\label{sec:overview}

We next expand upon the very high-level steps of the recovery algorithm presented in Section~\ref{sec:results}, 
detailing choices made in each step, and highlighting technical challenges that arise in the analysis. 
We also give an overview of the impossibility proof.  

\paragraph{Almost exact graph matching via the $k$-core estimator.} 
For a permutation $\pi : [n] \to [n]$, let $G_1 \land_{\pi} G_2$ be the corresponding intersection graph of $G_1$ and $G_2$, where $(i,j)$ is an edge in $G_1 \land_{\pi} G_2$ if and only if $(i,j)$ is an edge in $G_1$ and $(\pi(i), \pi(j))$ is an edge in $G_2$. The graph matching algorithm we study---called the $k$-core estimator---iterates over all permutations of $[n]$ and finds a permutation $\widehat{\pi}$ that induces the largest $k$-core\footnote{The $k$-core of a graph is the largest induced subgraph for which all vertices have degree at least $k$.} in the corresponding intersection graph. The output of the algorithm is a potentially incomplete vertex correspondence $\wh{\mu}$, which is the restriction of $\wh{\pi}$ to the vertex set of the $k$-core in $G_1 \land_{\wh{\pi}} G_2$. 
Figure~\ref{fig:k_core} provides an illustration.

While it is (information-theoretically) impossible that $\wh{\mu} = \pi_*$ with probability larger than $o(1)$ in the parameter regime that we consider, 
$\wh{\mu}$ has two important properties that we highlight. 
First, $\wh{\mu}$ matches almost all vertices; 
that is, if we denote by $F$ the set of unmatched vertices, then $|F| = o(n)$. 
Second, 
the vertices matched by $\wh{\mu}$ are all correct with probability $1 - o(1)$. This property is extremely important for downstream community recovery tasks, as it allows us to clearly leverage information from $G_1$ and $G_2$ in classifying a given vertex in the correspondence $\wh{\mu}$. 

Intuitively, 
graph matching algorithms tend to fail in aligning vertices that are not well-connected in $G_1 \land_{\pi_*} G_2$ (in particular, singletons), since such vertices have little common information across $G_1$ and $G_2$. 
On the other hand, the $k$-core is, by definition, a well-connected subgraph, 
explaining why the $k$-core estimator 
succeeds for an appropriate choice of $k$ (we choose $k = 13$ in this paper). 
Previously, 
Cullina, Kiyavash, Mittal, and Poor~\cite{cullina2020partial} 
proved the almost exact correctness of the $k$-core estimator for sparse correlated Erd\H{o}s-R\'{e}nyi graphs. Here, we extend their results to correlated SBMs, which requires a novel analysis, since the prior work of~\cite{cullina2020partial} utilized certain probability generating functions that can only be tractably computed for Erd\H{o}s-R\'{e}nyi graphs. Our approach circumvents this issue 
while also providing a tight analysis for the logarithmic degree regime.

\begin{figure}
    \centering
    \includegraphics[scale=0.25]{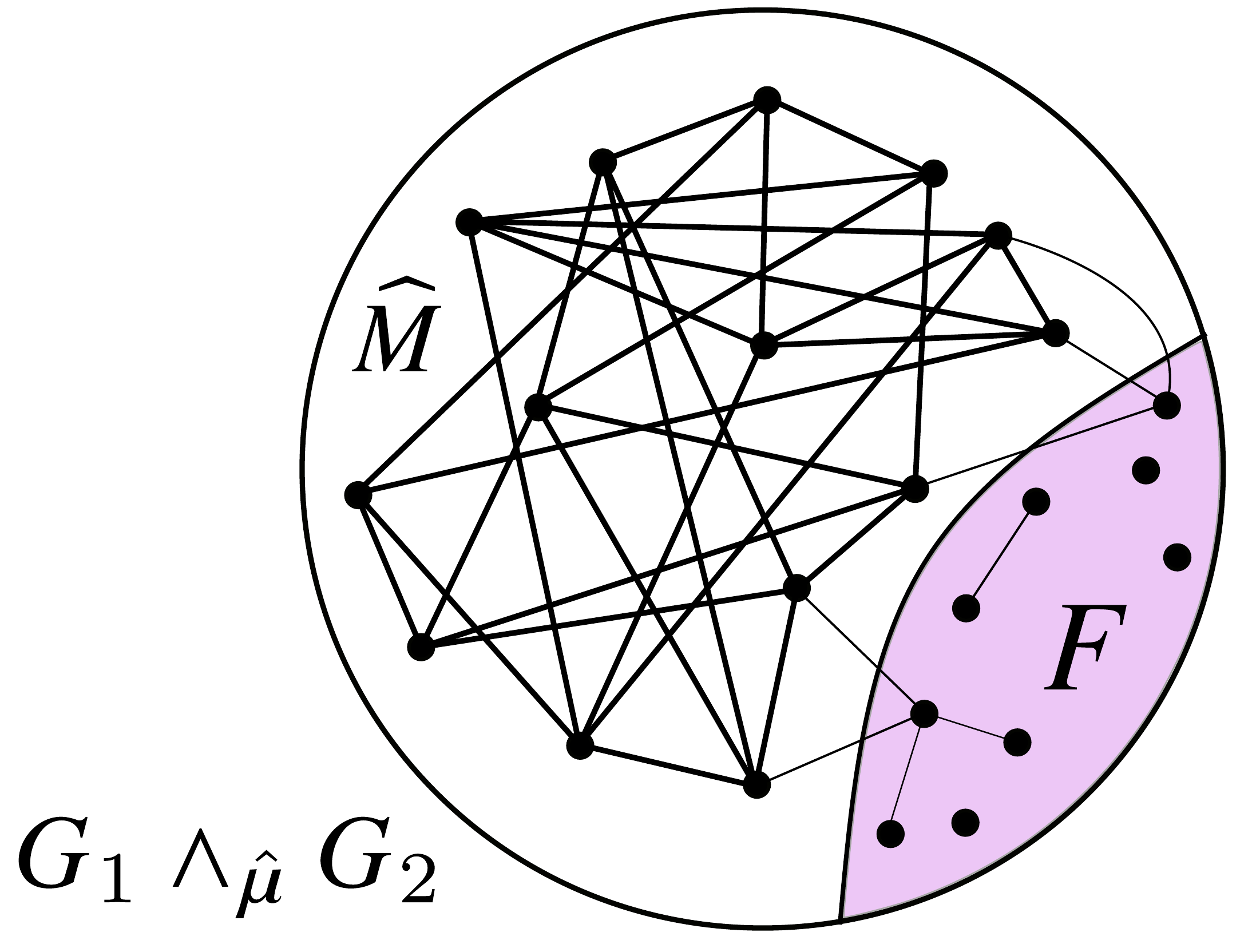}
    \caption{Illustration of the $k$-core matching (here, $k = 4$). Pictured here is the corresponding $k$-core of the intersection graph $G_1 \land_{\wh{\mu}} G_2$.}
    \label{fig:k_core}
\end{figure}

An important and novel consequence of our analysis is that for correlated SBMs in the logarithmic degree regime, the $k$-core estimator achieves \emph{optimal performance}. Specifically, we show that the $k$-core estimator fails to match at most $|F| \leq n^{1 - s^2 \dconnab + o(1)}$ vertex pairs, which is orderwise equal to the number of singletons of $G_1 \land_{\pi_*} G_2$, for which it is known that \emph{any} graph matching algorithm will fail \cite{RS21, cullina2016simultaneous}. This optimality of the $k$-core estimator is a fundamental reason why we utilize the $k$-core estimator to prove that the information-theoretic threshold can be achieved. 

\paragraph{Exact community recovery in the correctly matched region.} 
Since exact community recovery is impossible in $G_1$ or in $G_2$ alone, we will utilize the almost-exact vertex alignment $\wh{\mu}$ to combine information from both graphs to recover communities. The algorithm we design does so by recovering communities through multiple subroutines, each of which is executed carefully in order to de-couple the complex dependencies between $G_1$, $G_2$, and $\wh{\mu}$. 
Initially, we focus on recovering the community labels of vertices that are part of the (partial) matching $\wh{\mu}$.

First, we run a community recovery algorithm on $G_1$ to generate an almost-exact community labeling (i.e., with $o(n)$ errors). Using $\widehat{\mu}$, we then identify the graph $G_1 \lor_{\widehat{\mu}} G_2$, which consists of edges $(i,j)$ such that $(i,j)$ is an edge in $G_1$ or $(\widehat{\mu}(i), \widehat{\mu}(j))$ is an edge in $G_2$. Using $G_1 \lor_{\widehat{\mu}} G_2$, we then refine the almost-exact community labeling by re-classifying vertices in $\widehat{\mu}$ according to a majority vote among the labels of neighbors in $G_1 \lor_{\widehat{\mu}} G_2$. See Figure~\ref{fig:classifying_k_core} for an illustration.

\begin{figure}[t]
    \centering
    \includegraphics[scale=0.3]{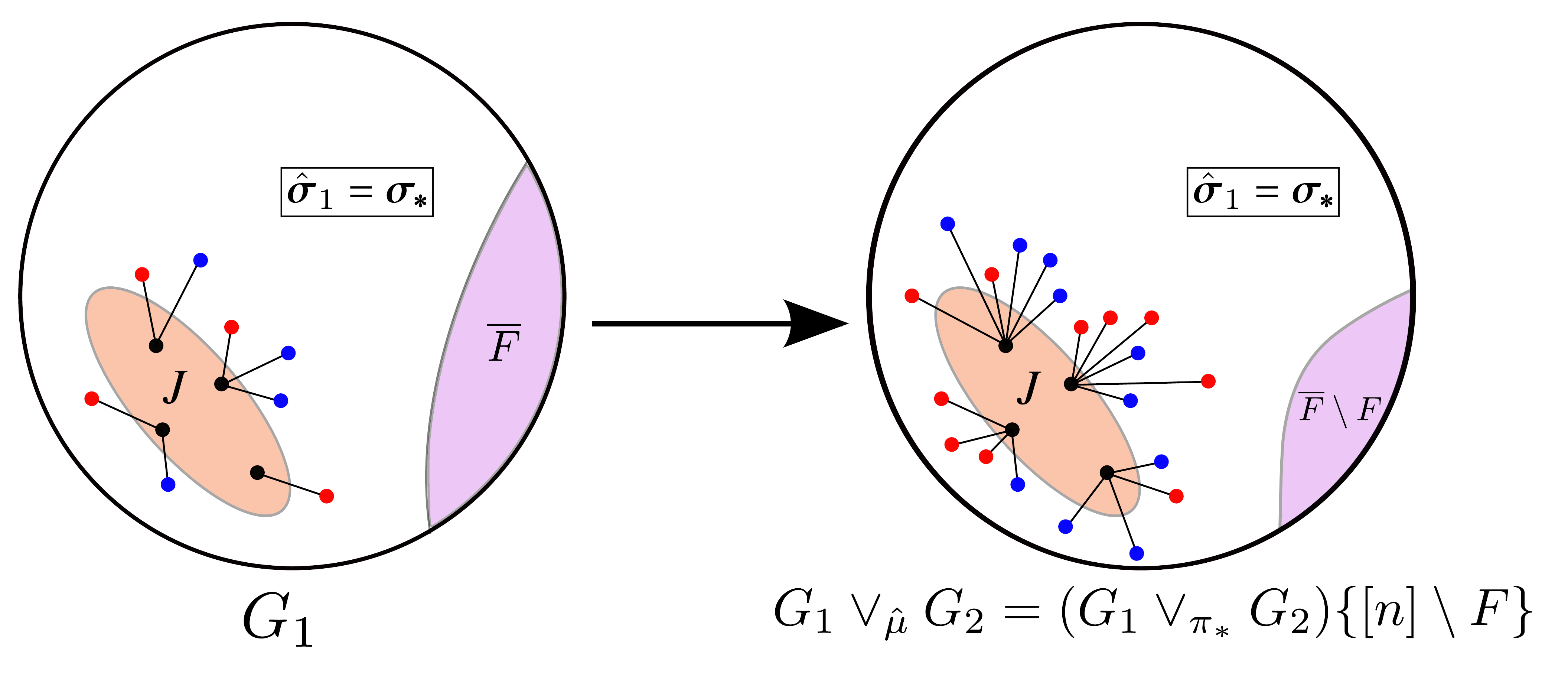}
    \caption{Illustration of how the recovery algorithm  classifies vertices in $[n] \setminus \overline{F}$ (see Algorithm~\ref{alg:labeling-k-core}). In each graph, the orange region denotes the set of incorrectly classified vertices resulting from applying Algorithm~\ref{alg:MNS} (almost exact community recovery) to $G_1$, the purple region denotes $\overline{F}$ or $\overline{F} \setminus F$ (which are classified later), and the white region denotes the set of correctly-classified vertices. When the two graphs are overlayed to form $G_1 \lor_{\wh{\mu}} G_2$, all vertices in $J$ have a neighborhood majority corresponding to the correct community labeling.}
    \label{fig:classifying_k_core}
\end{figure}

To make this algorithm work, there are several technical roadblocks which require novel 
ideas
to overcome. 
First, since the almost-exact community labels inferred from $G_1$ are subsequently analyzed in the context of $G_1 \lor_{\wh{\mu}} G_2$, it is critical that the incorrectly-classified vertices are not well-connected. If, for instance, a vertex had only a small number of correctly-classified neighbors in $G_1 \lor_{\wh{\mu}} G_2$, the final majority vote step would not be guaranteed to succeed. We therefore employ an algorithm previously developed by 
Mossel, Neeman and Sly~\cite{mossel2016consistency}, 
for which we can control the geometry of the misclassified vertices and show that the incorrectly-classified vertices are indeed only weakly connected. 
We remark that, as a consequence of our analysis, we show that the algorithm of 
\cite{mossel2016consistency} 
is \emph{optimal}, in the sense that it outputs a labeling which makes the smallest possible number of errors. We repeatedly leverage this property to show that our full algorithm works down to the information-theoretic threshold.  

The other key technical hurdle concerns the structure of $G_1 \lor_{\wh{\mu}} G_2$. Notice that, in the regime $\left(1 - (1 - s)^2 \right) \dchab > 1$, 
with high probability 
it is guaranteed that all vertices in $G_1 \lor_{\pi_*} G_2$ have the property that the majority of each community among their neighbors is the same as the community label of the vertex itself. However, it is unclear whether this is the case for $G_1 \lor_{\wh{\mu}} G_2$, since the removal of the vertices in $F$ that are not part of the (incomplete) vertex correspondence $\wh{\mu}$ may skew the neighborhood majority of nodes that have many neighbors in~$F$. To remedy this issue, we employ a method of 
\Luczak\cite{Luczak1991} 
to find a set $\overline{F} \supset F$ that is guaranteed to be only \emph{weakly connected} to vertices outside of $\overline{F}$ in $G_1 \lor_{\wh{\mu}} G_2$ (and not much greater in size than $F$). As a consequence, if $v \in [n] \setminus \overline{F}$ is incorrectly classified in $G_1$, the majority of its correctly labeled neighbors in $G_1 \lor_{\wh{\mu}} G_2$ will not be significantly affected by $\overline{F}$, resulting in the correct labeling. 

\paragraph{Classifying the vertices outside of the correctly matched region.} 
It now remains to classify vertices in $\overline{F}$. To classify $\overline{F} \setminus F$, it is useful to consider the graph $G_1 \setminus_{\wh{\mu}} G_2$, which consists of edges $(i,j)$ that are in $G_1$ such that $(\wh{\mu}(i), \wh{\mu}(j))$ is not in $G_2$. In our construction of $\overline{F}$, we are careful to only use the structure of $G_2$, not $G_1$, so that the neighbors of $\overline{F}$ in $G_1 \setminus_{\wh{\mu}} G_2$ do not depend strongly on the structure of $\overline{F}$. We then classify vertices in $\overline{F} \setminus F$ according to their majority among correctly-classified neighbors in $G_1 \setminus_{\wh{\mu}} G_2$, which are given by the previous step of the algorithm. Due to the approximate independence between $\overline{F}$ and its neighbors in $G_1 \setminus_{\wh{\mu}} G_2$, the probability that the majority vote fails can be computed in a straightforward manner and is shown to be $n^{- s(1 - s) \dch(\alpha, \beta) + o(1) }$. The factor of $s(1 - s)$ in the exponent reflects the fact that $G_1 \setminus_{\wh{\mu}} G_2$ is essentially constructed from edges in the parent graph $G$ that are sampled by $G_1$ but not by $G_2$. We show that $|\overline{F} | = n^{1 - s^2 \dconnab + o(1) }$ with probability $1 - o(1)$, which, along with the probability that majority fails, implies that we can informally bound the probability that the algorithm fails as 
\begin{multline*}
\p (\text{exists a vertex in $\overline{F} \setminus F$ such that the majority vote fails} ) \le | \overline{F} \setminus F | \cdot \p ( \text{majority vote fails} ) \\
\le n^{1 - s^2 \dconnab + o(1)} \cdot n^{- s(1 - s) \dchab + o(1)} 
 = n^{ 1 - ( s^2 \dconnab + s(1 - s) \dchab ) + o(1)}.
\end{multline*}
In particular, the final expression is $o(1)$ if \eqref{eq:tradeoff_achievability} holds. While the display above is quite informal, it captures the reason why our algorithm works in the regime \eqref{eq:tradeoff_achievability}. 

Finally, we classify vertices in $F$, which we recall is the set of vertices outside of the $k$-core of $G_1 \land_{\pi_*} G_2$, with high probability. We classify these vertices according to a majority vote with respect to their correctly-labeled neighbors in $G_1$ outside of $F$. By definition, vertices in $F$ can have at most $k - 1$ edges outside of $F$ in $G_1 \land_{\pi_*} G_2$, implying that the bulk of the neighbors will be in the graph $G_1 \setminus_{\pi_*} G_2$. Since $|F| \le n^{1 - s^2 \dconnab + o(1)}$ with high probability, we may repeat similar arguments as for the classification of $\overline{F} \setminus F$ above, to conclude that the majority vote will succeed in classifying $F$ provided that \eqref{eq:tradeoff_achievability} holds. 

\paragraph{Impossibility results.}
Since impossibility under the regime in \eqref{eq:community_impossibility} was proved by \cite{RS21}, we focus on the regime in \eqref{eq:tradeoff_impossibility}. 
To this end, notice that when $s^2 \dconnab < 1$, there are $n^{1 - s^2 \dconnab  + o(1)}$ singletons in $G_1 \land_{\pi_*} G_2$; these can be thought of as the vertices with non-overlapping information across $G_1$ and~$G_2$. This property makes such vertices 
impossible to match correctly. 
As a result, the maximum a posteriori (MAP) estimator for community labels of the singletons of $G_1 \land_{\pi_*} G_2$ almost completely disregards information from $G_2$ and classifies the singletons according to their neighborhood majority in $G_1$ alone. 
To make this rigorous, 
we give the MAP estimator additional information 
in the form of $\boldsymbol{\sigma_{*}^{2}}$ (the community labels in $G_{2}$) 
and the correct matching $\pi_{*}(i)$ for all nodes $i$ that are not singletons in $G_1 \land_{\pi_*} G_2$, 
and show that even with this additional information the MAP estimator fails. 
Specifically, we use the second moment method 
to show that under the condition \eqref{eq:tradeoff_impossibility}, with high probability, 
at least one of the majority votes will lead to the wrong classification. Since the MAP estimator fails in this regime, so too does any other estimator.

\subsection{Related work}\label{sec:related}

Since our work focuses on the interplay between community recovery and graph matching, 
it naturally connects with and builds upon the extensive literatures on these two topics. 
We highlight here the most relevant related work. 

\textbf{Community recovery in SBMs.} 
There is a vast literature on learning latent community structure in networks, 
and this question is by now understood well in SBMs~\cite{HLL83,dyer1989solution, bui1984graph, bopanna1987eigenvalues,DKMZ11,mossel2014reconstruction,massoulie2014community,mossel2018proof, abbe2016exact,mossel2016consistency,abbe2015community,bordenave2015nonbacktracking}; 
we refer the reader to Abb\'{e}'s survey~\cite{Abbe_survey} for an overview. 
We highlight in particular the works of Abb\'e, Bandeira, and Hall~\cite{abbe2016exact} and Mossel, Neeman, and Sly~\cite{mossel2016consistency}, which characterized the threshold for exact community recovery in the balanced two-community SBM. 
We build and expand upon their algorithms and analyses, 
in particular dealing with the uncertainties and dependencies arising from the partial, inexact matching between the correlated graphs.

\textbf{Beyond SBMs.} 
Roughly speaking, there are two main strands of literature that go beyond SBMs, 
incorporating various types of additional information to aid in recovering communities: 
contextual SBMs and multi-layer networks. 
In contextual SBMs, the idea is to leverage \emph{node-level} information (e.g., latent high-dimensional vectors) that is correlated with the community labels~\cite{bothorel2015clustering,kanade2016global,mossel2016local,zhang2016community,binkiewicz2017covariate,deshpande2018contextual,abbe2020ell_p,lu2020contextual,saad2020sideinfo,yan2021covariate,ma2021community}. 
In particular, the information-theoretic limits have recently been characterized for both community detection and exact community recovery~\cite{deshpande2018contextual,abbe2020ell_p,lu2020contextual}, 
and in both cases these limits shift due to the high-dimensional node covariates.

Multi-layer SBMs were 
introduced by Holland, Laskey, and Leinhardt, in 
the same
work that introduced SBMs~\cite{HLL83}. 
Here, given the underlying community structure, a collection of SBMs is generated on the same vertex set with the same latent community labels. Several variants 
have been explored~\cite{han2015consistent,arroyo2020inference, paul2020spectral, paul2021null, lei2019consistent, ali2019latent, bhattacharya2020consistent,chen2020global}, but typically the layers are \emph{conditionally independent} given the community labels. This is a major difference compared to the setting we consider, where the graphs are correlated through the formation of edges. Moreover, the node labels are assumed to be known in the multi-layer setting, which completely removes the need for graph matching.

The recent works~\cite{mayya2019mutual, ma2021community} jointly consider multi-layer networks and node-level information that is correlated with the latent community memberships, thus synthesizing these two strands of literature. 

\textbf{Graph matching: correlated Erd\H{o}s-R\'enyi model.} 
Arguably the simplest probabilistic generative model of correlated graphs 
is to consider two correlated Erd\H{o}s-R\'enyi random graphs. 
Consequently, this model, introduced by Pedarsani and Grossglauser~\cite{pedarsani2011privacy}, 
has been the focus of the theoretical literature on graph matching. 
The information-theoretic limits for recovering the latent vertex correspondence $\pi_{*}$ 
have been determined for exact recovery~\cite{cullina2016improved,cullina2018exact,wu2021settling} 
and almost exact recovery~\cite{cullina2020partial}, 
and significant progress has been made for weak recovery as well~\cite{ganassali2020tree,hall2020partial,ganassali2021impossibility,wu2021settling}. 

In particular, we highlight the work of Cullina~et~al.~\cite{cullina2020partial}, 
which is central to this paper and which we extend to correlated SBMs. 
They showed that the so-called \emph{$k$-core matching} achieves almost exact recovery when the average degree of the intersection graph diverges, 
and moreover, with high probability, all nodes in this partial matching are known to be correctly matched. 
This latter property is very useful, especially for downstream tasks such as combining a partial matching with community recovery steps. 
This directly motivates our choice of using a $k$-core matching in the algorithm that proves Theorem~\ref{thm:comm_recovery}. 

The quest for efficient algorithms for graph matching 
has led to numerous algorithmic advances~\cite{mossel2019seeded,barak2019,ding2021efficient,fan2020spectral,mao2021random}, 
culminating in the recent work of Mao, Rudelson, and Tikhomirov~\cite{mao2021exact}, 
who demonstrated an efficient algorithm for exact recovery in the constant noise regime.

\textbf{Graph matching: beyond Erd\H{o}s-R\'enyi.} 
A growing literature studies graph matching in models going beyond Erd\H{o}s-R\'enyi, including correlated SBMs~\cite{onaran2016optimal,cullina2016simultaneous,lyzinski2018information,RS21,shirani2021concentration} and more~\cite{korula2014efficient,RS22,yu2021power}. 
Closest to our work is that of R\'acz and Sridhar~\cite{RS21}, 
who determined the information-theoretic limits for exact graph matching in correlated SBMs, and subsequently leveraged this for exact community recovery. 
Our main contribution, discussed in detail in Sections~\ref{sec:CSBM} and~\ref{sec:results}, is to go beyond exact graph matching and to understand when exact community recovery is possible in the regime where exact graph matching is impossible.

\subsection{Discussion and future work}\label{sec:discussion}

Our work leaves open several important avenues for future work, which we now outline. 
\begin{itemize}
    \item \textbf{Efficient algorithms.} 
    In the parameter regime where exact community recovery is possible from $(G_{1}, G_{2})$ (see Theorem~\ref{thm:comm_recovery}), it is important to understand whether this is possible efficiently (in time polynomial in $n$). 
    The algorithm that we developed to prove Theorem~\ref{thm:comm_recovery} is not efficient; 
    specifically, the $k$-core matching step is inefficient, while the other steps are efficient. 
    Finding efficient algorithms for graph matching has been 
    the motivating force behind several 
    recent works (e.g.,~\cite{mossel2019seeded,barak2019,ding2021efficient,fan2020spectral,mao2021random}), 
    culminating in the recent breakthrough work of Mao, Rudelson, and Tikhomirov~\cite{mao2021exact}, who developed an efficient algorithm for graph matching in correlated Erd\H{o}s-R\'enyi random graphs with constant noise. 
    This promisingly suggests that efficient algorithms exist in the setting of the current paper as well. 
    
    We note, however, that using a $k$-core matching in this paper was a careful choice motivated by the desirable property that, with high probability, all nodes in this partial matching are known to be correctly matched. 
    This raises the possibility that developing efficient algorithms for the full regime of Theorem~\ref{thm:comm_recovery} may require significant new ideas beyond extending the work of~\cite{mao2021exact} to correlated SBMs (which, in itself, is an interesting open problem). 
    
    \item \textbf{Three or more correlated graphs.} What happens in the case of several correlated SBMs? 
    Achieving exact community recovery down to the threshold in Theorem~\ref{thm:comm_recovery} requires carefully passing information between the two correlated graphs. It would be interesting to understand how this generalizes to three or more graphs. 
    
    \item \textbf{Beyond exact community recovery.} 
    While here we focus on exact community recovery, it is of great interest to understand how multiple correlated SBMs can help with recovering communities in other parameter regimes. We conjecture that synthesizing information from a second, correlated graph can help in all settings. 
    
    For instance, when only almost exact community recovery is possible from $(G_{1}, G_{2})$, we conjecture that the optimal error rate is of smaller order than if only $G_{1}$ were known. 
    Similarly, in the partial recovery regime, we conjecture that a larger fraction of nodes can be recovered when given $(G_{1}, G_{2})$, as compared to when only $G_{1}$ is given. 
    Finally, we conjecture that the threshold for community detection decreases in the case of multiple correlated SBMs, compared to a single SBM. 
    Understanding all of these regimes quantitatively is an important direction for future work. 
    
    \item \textbf{General correlated stochastic block models.} 
    We focused here on the simplest setting of the SBM with two balanced communities. A natural future direction is to extend our results to more general SBMs with multiple communities, which are understood well in the single graph setting~\cite{Abbe_survey}. 
\end{itemize}


\subsection{Notation}\label{sec:notation}

Recall that the underlying vertex set is $V = [n] := \{1, 2, \ldots, n \}$. 
We denote by $\mathcal{S}_{n}$ the set of permutations of $[n]$. 
Recall that 
$V^{+} : = \{ i \in [n]: \sigma_{*}(i) = + 1 \}$ and 
$V^{-} : = \{ i \in [n] : \sigma_{*}(i) = -1 \}$ 
denote the vertices in the two communities. 
To emphasize the different vertex labels in $G_{1}$ and $G_{2}$, 
we define 
$\boldsymbol{\sigma_{*}^{1}} := \boldsymbol{\sigma_{*}}$ 
and 
$\boldsymbol{\sigma_{*}^{2}} := \boldsymbol{\sigma_{*}} \circ \pi_{*}^{-1}$, 
which are the community labels in~$G_{1}$ and $G_{2}$, respectively. 
Accordingly, we define $V_{1}^{+} := V^{+}$ and $V_{1}^{-} := V^{-}$, 
as well as 
$V_{2}^{+} : = \{ i \in [n]: \sigma_{*}^{2}(i) = + 1 \}$ and 
$V_{2}^{-} : = \{ i \in [n] : \sigma_{*}^{2}(i) = -1 \}$, 
to denote the two communities in the two graphs. 

Let $\binom{[n]}{2} : = \{ \{i,j\} : i,j \in [n], i \neq j \}$ denote the set of all unordered vertex pairs. We 
use $(i,j)$, $(j,i)$, and $\{i,j\}$ interchangeably to denote the unordered pair consisting of $i$ and $j$. 
Given 
a community labeling 
$\boldsymbol{\sigma}$, we define the sets 
$\cE^+(\boldsymbol{\sigma}) : = \left \{ (i,j) \in \binom{[n]}{2} : \sigma(i) \sigma(j) = +1 \right \}$ 
and 
$\cE^-(\boldsymbol{\sigma} ) : = \left \{ (i,j) \in \binom{[n]}{2} : \sigma(i) \sigma(j) = -1 \right \}$. 
In words, $\cE^+(\boldsymbol{\sigma})$ is the set of {\it intra-community} vertex pairs, and $\cE^-(\boldsymbol{\sigma})$ is the set of {\it inter-community} vertex pairs. 
Note 
that $\cE^+(\boldsymbol{\sigma})$ and $\cE^-(\boldsymbol{\sigma})$ partition~$\binom{[n]}{2}$. 

Let $A$ be the adjacency matrix of $G_{1}$, 
let $B$ be the adjacency matrix of $G_{2}$, 
and let $B'$ be the adjacency matrix of $G_{2}'$. 
Note that, by construction, we have that 
$B_{i,j}' = B_{\pi_{*}(i),\pi_{*}(j)}$ for every $i,j$. 
By the construction of the correlated SBMs, we have the following probabilities for every $(i,j) \in \binom{[n]}{2}$: 
\begin{align*}
\p \left( \left( A_{i,j}, B'_{i,j} \right) = (1,1) \, \middle| \, \boldsymbol{\sigma_{*}} \right) 
&= \begin{cases}
s^2 p &\text{if } \sigma_{*}(i) = \sigma_{*}(j), \\
s^2 q &\text{if } \sigma_{*}(i) \neq \sigma_{*}(j);
\end{cases}\\
\p \left( \left( A_{i,j}, B'_{i,j} \right) = (1,0) \, \middle| \, \boldsymbol{\sigma_{*}} \right) 
&= \begin{cases}
s(1-s) p &\text{if }  \sigma_{*}(i) = \sigma_{*}(j), \\
s(1-s) q &\text{if }  \sigma_{*}(i) \neq \sigma_{*}(j);
\end{cases}\\
\p \left( \left( A_{i,j}, B'_{i,j} \right) = (0,1) \, \middle| \, \boldsymbol{\sigma_{*}} \right) 
&= \begin{cases}
s(1-s) p &\text{if }  \sigma_{*}(i) = \sigma_{*}(j),  \\
s(1-s) q &\text{if }  \sigma_{*}(i) \neq \sigma_{*}(j);
\end{cases}\\
\p \left( \left( A_{i,j}, B'_{i,j} \right) = (0,0) \, \middle| \, \boldsymbol{\sigma_{*}} \right) 
&= \begin{cases}
1 - p(2s - s^2) &\text{if }  \sigma_{*}(i) = \sigma_{*}(j),  \\
1 - q(2s - s^2) &\text{if }  \sigma_{*}(i) \neq \sigma_{*}(j).
\end{cases}
\end{align*}
For brevity, for $i,j \in \{0,1\}$ we write 
\[
p_{ij} : = \p \left( \left( A_{1,2}, B'_{1,2} \right) = (i,j) \, \middle| \, \boldsymbol{\sigma_{*}} \right) \qquad \text{if $\sigma_{*}(1) = \sigma_{*}(2)$}
\]
and 
\[
q_{ij} : = \p \left( \left( A_{1,2}, B'_{1,2} \right) = (i,j) \, \middle| \, \boldsymbol{\sigma_{*}} \right) \qquad \text{if $\sigma_{*}(1) \neq \sigma_{*}(2)$}.
\] 

We also utilize some common notation for general graphs $G$. If $R$ is a subset of the vertex set of~$G$, we let $G \{R \}$ denote the induced subgraph of $G$ corresponding to $R$. 
For a vertex $i$ in $G$, 
we let 
$\cN_G(i)$ be the set of neighbors of $i$ in $G$. 
We abbreviate $\cN_{G_{1}}(i)$ as $\cN_{1}(i)$, 
and similarly $\cN_{G_{2}}(i)$ as~$\cN_{2}(i)$. 
We also let $\mathrm{deg}_G(i) = | \cN_G(i) |$ denote the degree of $i$ in $G$. 

For an event $\cA$, we denote by $\mathbf{1} \left( \cA \right)$ the indicator of $\cA$, which is $1$ if $\cA$ occurs and $0$ otherwise. Given a function $f$ and a subset $M$ of its domain, we let $f\{M \}$ denote the restriction of $f$ to $M$. We also let $f(M)$ denote the image of $M$ under $f$. 

Throughout the paper we use standard asymptotic notation 
and all limits are as $n \to \infty$.

\subsection{Organization}

The rest of the paper is devoted to the proofs of Theorems~\ref{thm:comm_recovery} and~\ref{thm:comm_recovery_impossibility} and is structured as follows. 
First, 
we describe the recovery algorithm in detail in Section~\ref{sec:recovery_algorithm}. 
Section~\ref{sec:proof_prelims} contains preliminary lemmas which are useful throughout, followed by three sections where the three main steps of the algorithm are analyzed: 
Section~\ref{sec:k-core_analysis} contains the analysis of the $k$-core estimator, 
Section~\ref{sec:labeling_proofs} proves the correctness of the estimated community labels for the matched vertices, 
and Section~\ref{sec:classify_rest} deals with classifying the remaining vertices. 
The different steps are combined into a proof of Theorem~\ref{thm:comm_recovery} in Section~\ref{sec:thm_proof_altogether}. 
Finally, Section~\ref{sec:impossibility_proof} contains the proof of the impossibility result, Theorem~\ref{thm:comm_recovery_impossibility}.


\section{The recovery algorithm}\label{sec:recovery_algorithm}
Our recovery algorithm begins by forming a matching between a subset of the vertices in $G_1$ and a subset of the vertices in $G_2$. Formally, we have the following definitions of a matching and a $k$-core matching.
\begin{definition}\label{def:matching}
Let $G_1$ and $G_2$ be two graphs with vertex set $[n]$. The pair $(M, \mu)$ is a matching between $G_1$ and $G_2$ if
\begin{itemize}
    \item $M \subseteq [n]$,
    \item $\mu : M \to [n]$, and
    \item $\mu$ is injective. 
\end{itemize}
\end{definition}

Given a matching $(M, \mu)$, we introduce the following related notation. We let $G_1 \lor_{\mu} G_2$ be the \emph{union graph}, whose vertex set is $M$, and whose edge set is $\{\{i,j\} : i, j \in M, A_{ij} + B_{\mu(i), \mu(j)} \geq 1 \}$. In other words, the union graph contains edges that appear in either graph, relative to the matching~$\mu$. Similarly, let $G_1 \land_{\mu} G_2$ be the \emph{intersection graph}, whose vertex set is $M$, and whose edge set is $\{\{i,j\} : i, j \in M, A_{ij} = B_{\mu(i), \mu(j)} = 1 \}$. The intersection graph contains all edges appearing in both graphs, relative to the matching~$\mu$. Let $G_1 \setminus_{\mu} G_2$ be the graph whose vertex set is $M$, and whose edge set is $\{\{i,j\} : i, j \in M, A_{ij} = 1, B_{\mu(i), \mu(j)} = 0 \}$. In other words, $G_1 \setminus_{\mu} G_2$ contains edges appearing in $G_1$ but not $G_2$, again, relative to the matching. Finally, let $G_2 \setminus_{\mu} G_1$ be the graph whose vertex set is $M$, and whose edge set is $\{\{i,j\} : i, j \in M, A_{ij} = 0, B_{\mu(i), \mu(j)} = 1 \}$. Note that all four definitions use vertex numbering relative to $G_1$. If $\pi: [n] \to [n]$ is a permutation, then the notation $G_1 \lor_{\pi} G_2$, $G_1 \land_{\pi} G_2$, $G_1 \setminus_{\pi} G_2$, and $G_2 \setminus_{\pi} G_1$ is defined according to the matching~$([n], \pi)$. 

In order to introduce our matching algorithm, we require the following definition. We let $\mathrm{d}_{\mathrm{min}}(G)$ be the minimal degree in a graph $G$.
\begin{definition}\label{def:k-core-matching}
A matching $(M, \mu)$ is a \emph{$k$-core matching} of $(G_1, G_2)$ if $\mathrm{d}_{\mathrm{min}}(G_1 \land_{\mu} G_2) \geq k$ (i.e., for every $i \in M$, the degree of $i$ in the graph $G_1 \land_{\mu} G_2$ is at least $k$). A matching $(M, \mu)$ is called a \emph{maximal $k$-core matching} if it involves the greatest number of vertices, among all $k$-core matchings. 
\end{definition}
The term \emph{$k$-core matching} comes from the notion of a \emph{$k$-core}. The $k$-core of a graph $G$ is the maximal subgraph with minimum degree at least $k$. 
Our first step is to produce a maximal $k$-core matching of the graphs $(G_1, G_2)$; see Figure \ref{fig:k_core} for an illustration.

\begin{breakablealgorithm}
\caption{$k$-core matching}\label{alg:k-core}
\begin{algorithmic}[1]
\Require{Pair of graphs $(G_1, G_2)$ on $n$ vertices, $k \in [n]$.}
\Ensure{A matching $(\widehat{M}, \wh{\mu})$  of $(G_1, G_2)$.}
\State By enumerating all possible matchings, find the maximal $k$-core matching $(\widehat{M}, \wh{\mu})$ of $G_1$ and~$G_2$. 
\end{algorithmic}
\end{breakablealgorithm}
Let $(\widehat{M}, \wh{\mu})$ be the matching found by Algorithm~\ref{alg:k-core} with $k=13$. We will show that $\wh{M}$ coincides with the maximal $k$-core of $G_1 \land_{\pi_*} G_2$, with high probability (see Lemma~\ref{lemma:k_core_sbm_v2}). Furthermore, we will show that a vanishing fraction of the vertices is excluded from this matching. Specifically, let $F := [n] \setminus \widehat{M}$ be the set of vertices which are excluded from the matching; we will show that $|F| \leq n^{1 - s^2 \dconnab + o(1)}$ with high probability (see Lemma~\ref{lemma:number-outside-core}). 

Our next step is to leverage the correctly matched region in order to find the correct communities. In particular, the goal of the following algorithm is to perfectly recover the communities in the subset of vertices found by the $k$-core matching (with $k=13$). The algorithm  
involves several intermediate steps which we detail below. 

\begin{breakablealgorithm}
\caption{Labeling the $k$-core}\label{alg:labeling-k-core}
\begin{algorithmic}[1]
\Require{Pair of graphs $(G_1, G_2)$ on $n$ vertices, a $k$-core matching $(M, \mu)$, parameters $\alpha, \beta, s, \epsilon > 0$.}
\Ensure{A labeling of $G_1 \{M\}$.}
\Statex\textbf{Finding a preliminary labeling}
\State\label{step:MNS1} Apply Algorithm~\ref{alg:MNS} to the graph $G_1$ and parameters $(s\alpha, s\beta, \epsilon)$, obtaining a labeling $\widehat{\boldsymbol{\sigma}}_1$.  
\vspace{12pt}
\Statex\textbf{Expanding $F := [n] \setminus M$}
\State Let $F := [n] \setminus M$. Apply Algorithm~\ref{alg:Luczak} to 
$(G_2, [n] \setminus \mu(M))$, 
obtaining the set 
$F' \supseteq [n] \setminus \mu(M)$. 
Let $\overline{F} := [n] \setminus \mu^{-1}([n] \setminus F')$, 
and note that $\overline{F} \supseteq F$. 
\vspace{12pt}
\Statex\textbf{Classifying $[n] \setminus {\overline{F}}$ and $\overline{F} \setminus F$}
\State\label{step:classify_bulk} For each $i \in [n] \setminus \overline{F}$, set $\wh{\sigma}(i) \in \{-1,1\}$ according to the neighborhood majority (resp., minority)~of  $\widehat{\boldsymbol{\sigma}}_1$ with respect to the graph $(G_1 \lor_{\mu} G_2) \{[n] \setminus \overline{F}\}$
if $\alpha > \beta$ (resp., $\alpha < \beta$). 
\State For each $i \in \overline{F} \setminus F$, set $\wh{\sigma}(i) \in \{-1,1\}$ according to the neighborhood majority (resp., minority) of $\widehat{\boldsymbol{\sigma}}$ with respect to the graph $(G_1 \setminus_{\mu} G_2) \{([n] \setminus \overline{F}) \cup \{i\} \}$ if $\alpha > \beta$ (resp.,~$\alpha < \beta$).  \label{step:alg2-9}
\vspace{12pt}
\State Return $\wh{\boldsymbol{\sigma}} : M \to \{-1,1\}$.
\end{algorithmic}
\end{breakablealgorithm}
Step \ref{step:MNS1} in Algorithm~\ref{alg:labeling-k-core} results in almost exact recovery of the community labels in $G_1$. We then need to transform the almost-exact labeling to an exact labeling. We would like to do this by a neighborhood vote in the union graph $G_1 \lor_{\wh{\mu}} G_2$; however, the contribution of edges from vertices in $\wh{M}$ to vertices in $F$ is hard to analyze. To remedy this difficulty, we expand the set $F$ into the set $\overline{F} \supseteq F$, such that any vertex of $G_2$ outside of $\pi_*(\overline{F})$ has at most one neighbor in $\pi_*(\overline{F})$, with high probability. This expansion, which is constructed by a method of \Luczak \cite{Luczak1991}, is guaranteed to satisfy $|\overline{F}| \leq 3 |F|$ with high probability (see Lemma~\ref{lemma:luczak}). Given the expanded set, we classify all $i \in [n] \setminus \overline{F}$ according to the majority of neighborhood values of $\wh{\boldsymbol{\sigma}}_1$ in the union graph $(G_1 \lor_{\wh{\mu}} G_2)\{[n] \setminus \overline{F}\}$. See Figure \ref{fig:classifying_k_core} for an illustration of this procedure.

In order to show that the majority vote correctly rectifies the labels, we first show a general result on almost-exact recovery in SBMs, which relates the set of incorrectly classified vertices to the set of vertices with weak majorities with respect to the ground-truth labeling (Lemma \ref{lemma:MNS}), adapted from \cite{mossel2016consistency}. Namely, let $J$ be the set of incorrectly classified vertices of a graph $G$ drawn from the SBM, and let $I_{\epsilon}(G)$ be the set of vertices in $G$ which do not have $\epsilon \log(n)$ majorities with respect to the ground-truth labeling $\boldsymbol{\sigma}_*$ (see Definitions~\ref{eq:I_epsilon_definition-1} and~\ref{eq:I_epsilon_definition-2} for formal definitions.) We will show that with high probability, for suitable $\epsilon> 0$, it holds that $J \subseteq I_{\epsilon}(G)$. Given this general result, it follows that the error set of $\widehat{\boldsymbol{\sigma}}_1$ is contained within the set $I_{\epsilon}(G_1)$. We then show that on the graphs $G_1$ and $(G_2 \setminus_{\wh{\mu}} G_1)\{[n] \setminus \overline{F}\}$, each vertex has few neighbors in $I_{\epsilon}(G_1)$ (Lemmas~\ref{lemma:I_internal} and~\ref{lemma:neighbors_g2_minus_g1_Iepsilon}). In turn, this allows us to show that on the union graph $(G_1 \lor_{\mu} G_2)\{[n] \setminus \overline{F}\}$, the neighborhood labels of a given vertex $i$ with respect to $\widehat{\boldsymbol{\sigma}}_1$ are close to the neighborhood labels with respect to $\boldsymbol{\sigma}_*$. Finally, we show that each vertex in $(G_1 \lor_{\mu} G_2)\{[n] \setminus \overline{F}\}$ has an $\epsilon \log n$ majority with respect to $\boldsymbol{\sigma}_*$ (Lemma \ref{lemma:union_graph_muhat_majority}). Therefore, taking a majority with respect to $\widehat{\boldsymbol{\sigma}}_1$ on $(G_1 \lor_{\mu} G_2)\{[n] \setminus \overline{F}\}$ transforms the almost-exact labeling to the correct labeling.

To complete the labeling of the $k$-core, we classify vertices in $\overline{F} \setminus F$ according to the majority of neighborhood values of $\boldsymbol{\widehat{\sigma}}\{[n] \setminus{\overline{F}}\}$, with respect to the graph $(G_1 \setminus_{\wh{\mu}} G_2)\{([n] \setminus \overline{F}) \cup \{i\} \}$. We are able to do so because the edges in the graph $(G_1 \setminus_{\wh{\mu}} G_2)\{([n] \setminus \overline{F}) \cup \{i\} \}$ are (nearly) independent of the construction of $\overline{F}$, conditioned on the random partition representation of correlated SBMs (see Section \ref{sec:alt_construction} for details).

We now provide the details of the subroutines used by Algorithm \ref{alg:labeling-k-core}. 
\begin{breakablealgorithm}
\caption{Almost-exact community recovery \cite[Algorithm 1]{mossel2016consistency}}\label{alg:MNS}
\begin{algorithmic}[1]
\Require{A graph $G$ on $n$ vertices, parameters $\alpha, \beta, \epsilon > 0$.}
\Ensure{A labeling on $G$ given by $\widehat{\boldsymbol{\sigma}}: [n] \to \{-1,1\}$.}
\State Choose a positive integer $m$ satisfying $( \log ( \epsilon m ( 200 \max \{ 1, \alpha, \beta \} )^{-1} ) - 1 ) \epsilon / 2 > 1$. Initialize two empty sets, $W_+$ and $W_-$.
\State Using the spectral method of \cite[Section 3.2]{Abbe2020}, find a community partition of $[n]$, denoted by~$(U_+, U_-)$.
\State Partition $[n]$ randomly into $\{U_{1}, \ldots, U_{m}\}$. 
\For{$i \in [m]$} 
\State Using the spectral method of \cite[Section 3.2]{Abbe2020}, find a community partition $(U_{i,+},U_{i,-})$ of $G \left\{ [n] \setminus U_{i} \right\}$. If $|U_{i, +} \Delta U_{+}| \ge n/2$, then swap $U_{i, +}$ and $U_{i,-}$. 
\State For $v \in U_i$, insert $v$ into $W_+$ or $W_-$ according to 
    its neighborhood majority (resp., minority) in $U_{i,+} \cup U_{i,-}$ if $\alpha > \beta$ (resp., $\alpha < \beta$). 
\EndFor
\State For $i \in W_+$, set $\wh{\sigma}(i) = 1$, and for $i \in W_-$, set $\wh{\sigma}(i) = -1$. Return $\wh{\boldsymbol{\sigma}}$.
\end{algorithmic}
\end{breakablealgorithm}

\begin{breakablealgorithm}
\caption{\Luczak expansion}\label{alg:Luczak}
\begin{algorithmic}[1]
\Require{A graph $G$ on $n$ vertices, a set $U \subseteq [n]$.}
\Ensure{A set $\overline{U} \subseteq [n]$ such that $\overline{U} \supseteq U$ and for all $i \in [n] \setminus \overline{U}$, $i$ has at most one neighbor in $\overline{U}$.}
\State Let $U_0 = U$.
\For{$i \in \{0, 1, \dots, n\}$}
        \State Let $U_{i+1}'$ be the set of vertices outside $U_i$ that have at least two neighbors in $U_i$.
        \If{$U_{i+1}' = \emptyset$}
        Return $U_i$.
        \Else 
        \State Set $U_{i+1} = U_i \cup \{v\}$, where $v$ is an arbitrarily chosen vertex in $U_{i+1}'$.
        \EndIf 
\EndFor
\end{algorithmic}
\end{breakablealgorithm}

The complete algorithm appears below. First, the $k$-core matching is found. Next, the vertices comprising the matching are labeled. Finally, the vertices excluded from the matching are labeled according to neighborhood labels in the graph $G_1$, restricted to the vertices comprising the matching.
\begin{breakablealgorithm}
\caption{Full community recovery}\label{alg:main}
\begin{algorithmic}[1]
\Require{Pair of graphs $(G_1, G_2)$ on $n$ vertices, $k \in [n]$, and $\epsilon > 0$.}
\Ensure{A labeling of $G_1$ given by $\widehat{\boldsymbol{\sigma}}: [n] \to \{-1,1\}$.}
\State Apply Algorithm \ref{alg:k-core} on input $(G_1, G_2, k)$, obtaining a matching $(\widehat{M}, \wh{\mu})$. \label{step:alg3-1}
\State Apply Algorithm \ref{alg:labeling-k-core} on input $(G_1, G_2, \widehat{M}, \wh{\mu}, \epsilon)$, obtaining a labeling $\wh{\boldsymbol{\sigma}} : \widehat{M} \to \{-1,1\}$.\label{step:alg3-2}
\State For $i \in [n] \setminus \widehat{M}$, classify $i$ according to its neighborhood majority (resp., minority) in the graph $G_1 \{\widehat{M} \cup \{i\}\}$ if $\alpha > \beta$ (resp., $\alpha < \beta$). \label{step:alg3-3}
\end{algorithmic}
\end{breakablealgorithm}

\begin{theorem}\label{thm:comm_recovery_reduction}
Fix constants $\alpha, \beta > 0$ and $s \in [0,1]$. 
Let $(G_1, G_2) \sim \mathrm{CSBM}\left( n,\frac{\alpha \log n}{n} ,\frac{\beta \log n}{n}, s \right)$.  
Suppose that~\eqref{eq:community_achievability} and~\eqref{eq:tradeoff_achievability} hold.
Let $\epsilon > 0$ satisfy 
\begin{equation}\label{eq:eps_condition}
\left(1 - (1 - s)^{2} \right) \dchab > 1 + 2\epsilon | \log(\alpha / \beta) | \hspace{1cm} \text{ and } \hspace{1cm} 0 < \epsilon \le \frac{s \dchab}{4 | \log (\alpha / \beta) |}.
\end{equation}
Then Algorithm~\ref{alg:main} on input $(G_1, G_2, 13, \epsilon)$ correctly labels all of the vertices in $G_1$, with high probability.
\end{theorem} 
Since~\eqref{eq:community_achievability} and~\eqref{eq:tradeoff_achievability} 
imply
the existence of $\epsilon > 0$ such that~\eqref{eq:eps_condition} holds, 
proving Theorem~\ref{thm:comm_recovery} reduces to proving Theorem~\ref{thm:comm_recovery_reduction}. 
The proof of Theorem~\ref{thm:comm_recovery_reduction} can be found in Section~\ref{sec:thm_proof_altogether}, with the supporting results proved in Sections~\ref{sec:k-core_analysis},~\ref{sec:labeling_proofs}, and~\ref{sec:classify_rest}.


\section{Preliminary results}\label{sec:proof_prelims}

\subsection{Binomial probabilities}
Since binomial differences appear frequently in our analysis, we include two useful results here. The first can be found in the proof of~\cite[Proposition~2.8]{mossel2016consistency}.
\begin{lemma}
\label{lemma:majority_probability}
Suppose that $\alpha \ge \beta$. Let $Y \sim \mathrm{Bin}(m^+, \alpha \log(n) / n)$ and $Z \sim \mathrm{Bin}(m^-, \beta \log(n) / n)$ be independent. 
If $m^+  = (1 + o(1)) n / 2$ and $m^- = (1 + o(1)) n / 2$, then 
\[
\p ( Y < Z) = n^{- \dch(\alpha, \beta) + o(1)}.
\]
\end{lemma}

\begin{remark}
The precise result in \cite[Proposition 2.8]{mossel2016consistency} characterizes $\p( Y \le Z)$ when $m^+ = m^- = n/2$. It is readily seen from their proof of the result that it also holds for the slightly more general values of $m^+$ and $m^-$ that we consider and when the inequality between $Y$ and $Z$ is strict. 
\end{remark}

The following result follows from \cite[Lemma 8 of Supplement]{Abbe2020}.
\begin{lemma}\label{lemma:binomial-tail}
Suppose that $\alpha > \beta$. Let $Y \sim \mathrm{Bin}\left(m_+, \frac{\alpha \log n}{n}\right)$ and $Z \sim \mathrm{Bin}\left(m_-, \frac{\beta \log n}{n}\right)$ be independent. If $m^+  = (1 + o(1)) n / 2$ and $m^- = (1 + o(1)) n / 2$, then for any $\epsilon > 0$, 
\[\mathbb{P}\left(Y - Z \leq \epsilon \log n \right) \leq n^{-\left(\dchab - \frac{\epsilon \log( \alpha / \beta)}{2}\right) + o(1)}.\]
\end{lemma}
\begin{remark}
Note that \cite[Lemma 8 of Supplement]{Abbe2020} treats the case where $m_+ = m_- = n/2$, obtaining a tail bound of $n^{-\left(\dchab - \frac{\epsilon \log(\alpha / \beta)}{2}\right)}$. Allowing for $m_+, m_- = (1+o(1)) n/2$ is reflected by an additional $o(1)$ in the exponent.
\end{remark}

\subsection{A useful construction of correlated SBMs}
\label{sec:alt_construction}

In this section, we detail a useful alternate construction of correlated SBMs that highlights the independent regions of $G_1$ and $G_2$. 

To begin, we construct a random partition $\{\cE_{00}, \cE_{01}, \cE_{10}, \cE_{11}\}$ of $\binom{[n]}{2}$ as follows. Independently for each $\{i,j \} \in \binom{[n]}{2}$, we add $\{i,j \}$ to $\cE_{00}$ with probability $(1 - s)^2$, to $\cE_{01}$ with probability $s(1 - s)$, to $\cE_{10}$ with probability $s(1 - s)$, and to $\cE_{11}$ with probability $s^2$. Subsequently, for each pair $\{i,j \}$, we construct an edge between $i$ and $j$ with probability $p$ if the two vertices are in the same community, else with probability $q$ if the two vertices are in different communities. The graph $G_1$ is constructed using the edges formed in $\cE_{10} \cup \cE_{11}$ and the graph $G_2'$ is constructed using the edges formed in~$\cE_{01} \cup \cE_{11}$. The graph $G_2$ is then generated from $G_2'$ and $\pi_*$ by relabeling the vertices of $G_2'$ according to~$\pi_*$. The usefulness of this construction is that it provides an alternate way to generate correlated SBMs and highlights regions of the two graphs that are independent of each other (see Figure~\ref{fig:random_partition} for an illustration). This is formally stated in the following result.

\begin{figure}
    \centering
    \includegraphics[scale=0.2]{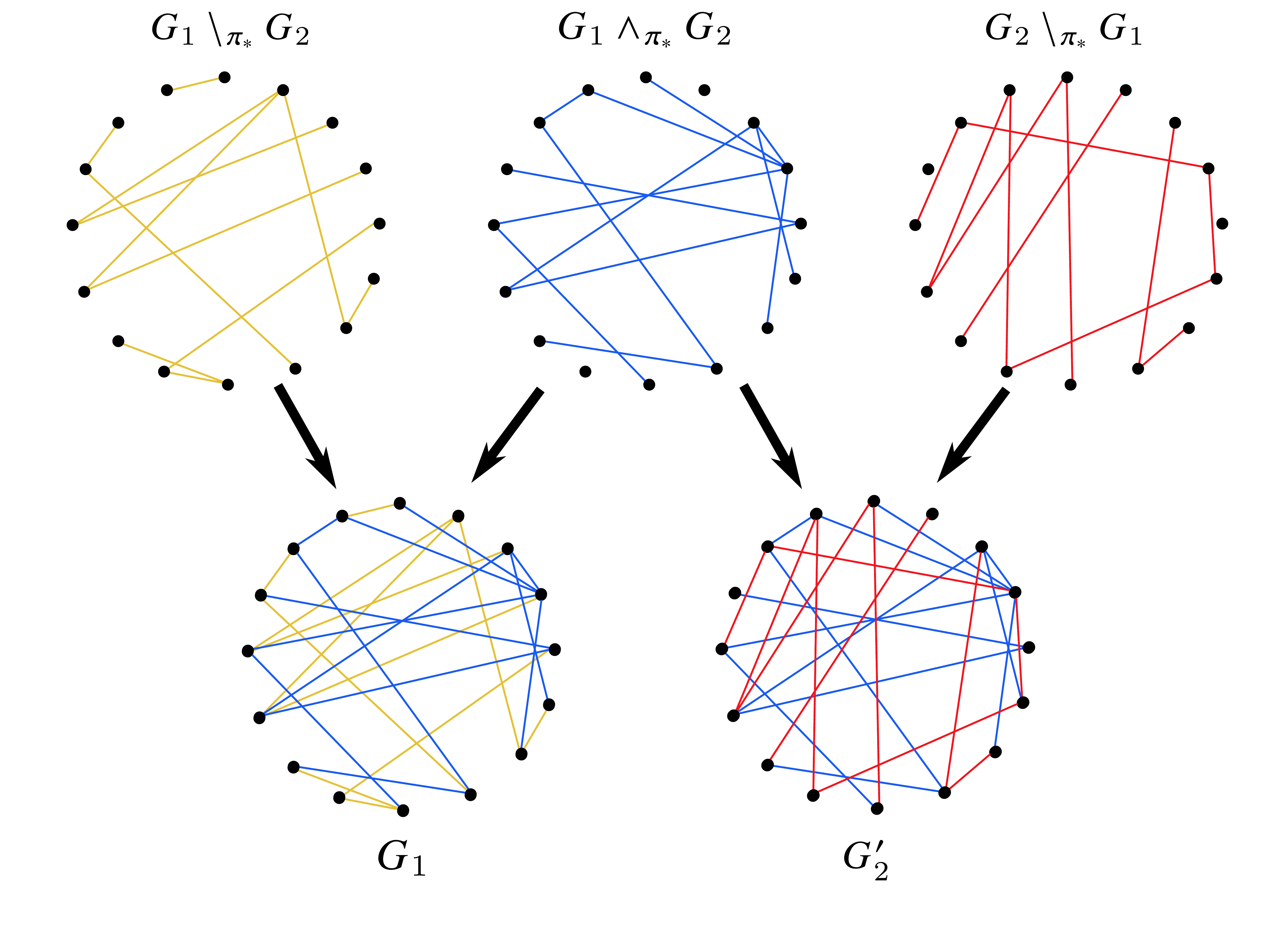}
    \caption{Decomposition of correlated $G_1$ and $G_2'$ into $G_1 \setminus_{\pi_*} G_2, G_1 \land_{\pi_*} G_2$, and $G_2 \setminus_{\pi_*} G_1$, which are conditionally independent given $\pi_{*}$, $\boldsymbol{\sigma_{*}}$, and the random partition $\{\cE_{00}, \cE_{10}, \cE_{11}, \cE_{01}\}$ (see Lemma~\ref{lemma:random_partition}).}
    \label{fig:random_partition}
\end{figure}

\begin{lemma}
\label{lemma:random_partition}
The random partition construction of correlated SBMs above is equivalent to the original construction described in Section~\ref{sec:CSBM}. 
Moreover, conditioned on $\pi_*$, $\boldsymbol{\sigma_{*}}$, $\cE_{00}$, $\cE_{01}$, $\cE_{10}$, and~$\cE_{11}$, the graphs $G_1 \setminus_{\pi_*} G_2$, $G_1 \land_{\pi_*} G_2$, and $G_2 \setminus_{\pi_*} G_1$ are mutually independent.
\end{lemma}

\begin{proof}
To prove the first statement, 
it suffices to check that the marginal distribution of the pair $(A_{i,j}, B_{\pi_*(i), \pi_*(j)})$ is the same under both constructions, 
due to the independence over vertex pairs. 
Under the original construction, 
if $\sigma_*^1(i) = \sigma_*^1(j)$, then 
\begin{equation}
\label{eq:construction1_marginal}
\p \left( (A_{i,j}, B_{\pi_*(i), \pi_*(j)}) = (a,b) \, \middle| \, \pi_*, \boldsymbol{\sigma_{*}^{1}} \right) = \begin{cases}
s^2 p &\text{ if } (a,b) = (1,1), \\
s(1 - s) p &\text{ if } (a,b) \in \{ (1,0), (0,1) \}, \\
1 - \left( 1 - (1 - s)^{2} \right) p &\text{ if } (a,b) = (0,0).
\end{cases}
\end{equation}
When $\sigma_*^1(i) = - \sigma_*^1(j)$, the joint distribution is the same as in \eqref{eq:construction1_marginal}, with $p$ replaced by $q$. 
In the second construction, if $(a,b)  \neq (0,0)$, then we have, for $i,j$ satisfying $\sigma_*^1(i) = \sigma_*^1(j)$, that
\begin{align}
\p \left( (A_{i,j}, B_{\pi_*(i), \pi_*(j)} ) = (a,b) \, \middle| \, \pi_*, \boldsymbol{\sigma_{*}^{1}} \right) & = \p \left ( \{i,j \} \in \cE_{ab} \, \middle| \, \pi_*, \boldsymbol{\sigma_{*}^{1}} \right) \cdot p \nonumber  \\
\label{eq:construction2_marginal}
& = \begin{cases}
s^2 p &\text{ if } (a,b) = (1,1), \\
s(1 - s)p &\text{ if } (a,b) \in \{ (1,0), (0,1) \}.
\end{cases}
\end{align}
When $\sigma_*^1(i) = -\sigma_*^1(j)$, the joint distribution is the same as in~\eqref{eq:construction2_marginal}, with $p$ replaced by $q$. 
Since the probabilities agree when $(a,b) \neq (0,0)$, they must also agree when $(a,b) = (0,0)$, since all the probabilities sum to one. 

To prove the second statement, notice that 
the graph $G_1 \setminus_{\pi_*} G_2$ is comprised of edges in $\cE_{10}$, 
$G_1 \land_{\pi_*} G_2$ is comprised of edges in $\cE_{11}$, 
and $G_2 \setminus_{\pi_*} G_1$ is comprised of edges in $\cE_{01}$. 
Conditioned on $\pi_*$, $\boldsymbol{\sigma_{*}^{1}}$, and the random partition $\{\cE_{00}, \cE_{10}, \cE_{11}, \cE_{01}\}$, the formation of edges in $\cE_{10}$, $\cE_{11}$, and $\cE_{01}$ are mutually independent. 
\end{proof}

\begin{remark}
A useful consequence of the random partition construction is that marginally, the graphs 
$G_1 \setminus_{\pi_*} G_2$, 
$G_1 \land_{\pi_*} G_2$, 
and $G_2 \setminus_{\pi_*} G_1$ 
are SBMs. 
However, conditioned on the partition $\{ \cE_{00}, \cE_{01}, \cE_{10}, \cE_{11} \}$, these graphs are \emph{not} SBMs, since, for instance, the set of potential edges in $G_1 \setminus_{\pi_*} G_2$ is given by $\cE_{10}$ rather than $\binom{[n]}{2}$. 
\end{remark}

When we utilize the random partition construction in later proofs, it is useful to consider a high-probability event where the partition behaves in a nice manner. Formally, we introduce the following event. 

\begin{definition}[The event $\cF$]
\label{def:F}
For $a,b \in \{0,1\}$, define the constant
$$
s_{ab} : = \begin{cases}
s^2 &\text{ if } (a,b) = (1,1), \\
s(1 - s) &\text{ if } (a,b) \in \{(1,0), (0,1) \}, \\
(1 - s)^2 &\text{ if } (a,b) = (0,0).
\end{cases}
$$
The event $\cF$ holds if and only if 
$$n / 2 - n^{3/4} \le |V^+|, |V^-| \le n / 2 + n^{3/4}$$
and also 
the following conditions hold for all $a,b \in \{0,1 \}$ and all $i \in [n]$:
\begin{itemize}
    \item $s_{ab} \left( |V^{\sigma_*(i)}| - n^{3/4} \right) \le | \{ j : \{i,j \} \in \cE_{ab} \cap \cE^+(\boldsymbol{\sigma_{*}^{1}}) \}|  \le s_{ab} \left( |V^{\sigma_*(i)}| + n^{3/4} \right)$;
    \item $s_{ab} \left( |V^{- \sigma_*(i)}| - n^{3/4} \right) \le | \{ j : \{i,j \} \in \cE_{ab} \cap \cE^-(\boldsymbol{\sigma_{*}^{1}}) \}|  \le s_{ab} \left( |V^{-\sigma_*(i)}| + n^{3/4} \right)$. 
\end{itemize}
\end{definition}
The first condition ensures that the two communities are balanced.
The next two conditions stipulate that for each vertex $i \in [n]$, the number of potential neighbors to both communities in each component of the partition is balanced (and approximately the expected size, according to the weights $\{s_{ab}\}_{a,b \in \{0,1\}}$). 
The following result shows that $\cF$ holds with high probability. 

\begin{lemma}
\label{lemma:F}
Define $s_{\min} : = \min_{a,b \in \{0,1 \}} s_{ab}$. Then $\p(\cF^c) \le 100 n e^{ - \frac{1}{2} s_{\min}^2 \sqrt{n} }$. 
\end{lemma}

\begin{proof}
Define the event 
$$
\cG : = \left \{ | | V^+ | - n / 2 |, | |V^- | - n / 2 | < n^{3/4} \right \}.
$$
By Hoeffding's inequality, we have that 
$$
\p \left( ||V^+ | - n/ 2| \ge n^{3/4} \right), \p \left( | |V^-| - n/2 | \ge n^{3/4} \right) \le 2e^{ - \sqrt{n}}.
$$
Thus, by a union bound, we have that $\p ( \cG^c) \le 4 e^{- \sqrt{n}}$. 

We next study the remaining conditions of the event $\cF$. 
We assume that $s\in(0,1)$, and hence $s_{ab} \in (0,1)$ for all $a,b \in \{0,1\}$; 
for $s \in \{0,1\}$ we have that $\cG \subseteq \cF$, so we are done.
Condition on $\pi_*$ and $\boldsymbol{\sigma_{*}^{1}}$. Fix $i \in [n]$. Notice that
$$
K_{ab}^+(i) : = | \{ j : \{i,j \} \in \cE_{ab} \cap  \cE^+(\boldsymbol{\sigma_{*}^{1}}) \} | \sim \mathrm{Bin}\left( |V^{\sigma_*(i)}| - 1, s_{ab} \right).
$$
Hoeffding's inequality implies that 
\begin{align*}
\p \left( \left| K_{ab}^+(i) - s_{ab}\left( |V^{\sigma_*(i)}| - 1 \right) \right| \ge \frac{s_{ab}}{2}  n^{3/4}  \, \middle| \, \pi_*, \boldsymbol{\sigma_{*}} \right) \mathbf{1}(\cG) 
&\le 2 \mathrm{exp} \left( - \frac{s_{ab}^{2} n^{3/2}}{2 |V^{\sigma_*(i)}|} \right) \mathbf{1}(\cG) \\
&\le 2e^{- (1 - o(1)) s_{ab}^{2}\sqrt{n}}.
\end{align*}
Taking a union bound over $i \in [n]$ shows that
\begin{multline*}
\p \left( \exists i \in [n] : | K_{ab}^+(i) - s_{ab} |V^{\sigma_*(i)} | | \ge  s_{ab} n^{3/4} \right) 
\le \sum\limits_{i=1}^{n} \p \left( | K_{ab}^+(i) - s_{ab} ( |V^{\sigma_*(i)} | - 1 ) | \ge \frac{s_{ab}}{2} n^{3/4} \right) \\
\le \sum\limits_{i=1}^{n} \left( \E \left[ \p \left(   | K_{ab}^+(i) - s_{ab} ( |V^{\sigma_*(i)} | - 1 ) | \ge \frac{s_{ab}}{2} n^{3/4}  \, \middle| \, \pi_*, \boldsymbol{\sigma_{*}} \right) \mathbf{1}( \cG )  \right] + \p ( \cG^c) \right) 
\le 6n e^{ - (1 - o(1)) s_{ab}^{2}\sqrt{n}}.
\end{multline*}
Similarly, defining
$$
K_{ab}^-(i) : = | \{j : \{i,j \} \in \cE_{ab} \cap \cE^-(\boldsymbol{\sigma_{*}^{1}}) \} | \sim \mathrm{Bin} \left( |V^{-\sigma_*(i)}|, s_{ab} \right),
$$
an identical analysis shows that 
$$
\p \left( \exists i \in [n] : | K_{ab}^-(i) - s_{ab} | V^{- \sigma_*(i)}| | \ge n^{3/4} \right) \le 6n e^{ - (1 - o(1)) s_{ab}^{2}\sqrt{n}}. 
$$
The conclusion then follows by a union bound. 
\end{proof}


\section{Analysis of the $k$-core estimator (Algorithm \ref{alg:k-core})}\label{sec:k-core_analysis}

\subsection{General results on correctness}
Recall the definition of a matching (Definition \ref{def:matching}) and a $k$-core matching (Definition \ref{def:k-core-matching}). We now introduce some additional definitions. Given a graph $G$ with vertex set $[n]$, let $\deg_G(i)$ be the degree of vertex $i \in [n]$ within $G$. For a matching $(M, \mu)$, define 
$$
f(M, \mu, G_1, G_2, \pi_*) : = \sum\limits_{i \in M: \mu(i) \neq \pi_*(i)} \mathrm{deg}_{G_1 \land_{\mu} G_2}(i).
$$
For brevity, we sometimes write $f(\mu)$. In words, this is the sum, over vertices that are incorrectly matched by $\mu$, of degrees in the intersection graph $G_1 \land_{\mu} G_2$.
\begin{definition}
Given two graphs $(G_1, G_2)$ with vertex set $[n]$, a matching $(M, \mu)$ is a \emph{weak} $k$-core matching of $(G_1, G_2)$ if $f(\mu) \ge k | \{ i \in M: \mu(i) \neq \pi_*(i) \}|$. 
\end{definition}
In other words, a matching $(M, \mu)$ is a weak $k$-core matching if the {\it average degree} of an incorrectly-matched vertex is at least $k$. Note that if $(M,\mu)$ is a $k$-core matching, then it is also a weak $k$-core matching. 

Finally, we introduce the notion of a maximal matching. 
\begin{definition}
A matching $(M, \mu)$ of $(G_1, G_2)$ is $\pi_*$-maximal if for every $i \in [n]$, either $i \in M$ or $\pi_*(i) \in \mu(M)$, where $\mu(M)$ is the image of $M$ under $\mu$.
\end{definition}
We can always extend a matching $(M, \mu)$ to a $\pi_*$-maximal matching by adding all possible input-output pairs of the form $(i, \pi_*(i) )$ to $(M, \mu)$ that do not break the one-to-one property of $\mu$; see Lemma \ref{lemma:maximal_matching} for details. Moreover, this extension preserves the weak $k$-core matching property. 
\begin{lemma}
\label{lemma:maximal_matching}
For any matching
$(M, \mu)$ of $(G_1, G_2)$, there exists a $\pi_*$-maximal matching $(M', \mu')$ such that $f(\mu) \le f(\mu')$.
\end{lemma}

\begin{proof}
If $(M, \mu)$ is not $\pi_*$-maximal, then the set 
$$
S: = \{ i \in [n] : i \notin M \text{ and } \pi_*(i) \notin \mu(M) \}
$$
is nonempty. We can then construct $(M', \mu')$ by defining $M' := M \cup S$, and $\mu'(i) := \mu(i)$ for $i \in M$ and $\mu'(i) := \pi_*(i)$ for $i \in S$. 
Clearly, $(M', \mu')$ is $\pi_*$-maximal by construction. Furthermore, 
\begin{align*}
f(\mu') & = \sum\limits_{i \in M' : \mu'(i) \neq \pi_*(i) } \mathrm{deg}_{G_1 \land_{\mu'} G_2}(i) 
 = \sum\limits_{i \in M: \mu'(i) \neq \pi_*(i) } \mathrm{deg}_{G_1 \land_{\mu'} G_2}(i) \\
& \ge \sum\limits_{i \in M: \mu(i) \neq \pi_*(i)} \mathrm{deg}_{G_1 \land_\mu G_2 }(i) = f(\mu).
\end{align*}
Above, the second equality follows since $\mu'(i) = \pi_*(i)$ for $i \in S$, 
and the inequality on the last line follows since $\mu(i) = \mu'(i)$ for $i \in M$ 
and also $\mathrm{deg}_{G_1 \land_{\mu'} G_2}(i) \ge \mathrm{deg}_{G_1 \land_{\mu} G_2}(i)$, 
since $G_1 \land_{\mu} G_2$ is a subgraph of $G_1 \land_{\mu'} G_2$. 
\end{proof}

We also define the set of $\pi_*$-maximal matchings with $d$ errors:
\[
\cM(d) : = \left \{(M, \mu) : \text{$(M, \mu)$ is $\pi_*$-maximal and }| \{ i \in M: \mu(i) \neq \pi_*(i) \}| = d \right \}.
\]
\begin{remark}\label{rem:Md}
While the set $\cM(d)$ depends on $\pi_{*}$ (since errors are measured relative to $\pi_{*}$), 
the set 
$\{ (M, \pi_{*}^{-1} \circ \mu) : (M,\mu) \in \cM(d) \}$ 
does \emph{not} depend on $\pi_{*}$. 
Therefore, certain properties of $\cM(d)$ also do not depend on $\pi_{*}$. 
In particular, its size $|\cM(d)|$ is a deterministic function of $n$ and $d$. 
\end{remark}
The usefulness of defining the set $\cM(d)$ is that it is much smaller than the set of matchings with $d$ errors (without the maximality condition). For instance, when $d = 1$, the number of matchings is at least the number of ways to choose~$M$, which in turn is $2^n - 1$. The following lemma shows that, on the other hand, the size of $\cM(d)$ is polynomial in $n$. This was previously proven in \cite{cullina2020partial}, and provided here for completeness since the proof is short. 

\begin{lemma}
\label{lemma:Md_size}
For any $1 \le d \le n$, we have that $|\cM(d) | \le n^{2d} / d!$. 
\end{lemma}

\begin{proof}
For a $\pi_*$-maximal $(M,\mu)$, define $M' : = \{ i \in M : \mu(i) \neq \pi_*(i) \}$, which is the set of vertices in $G_1$ that are incorrectly matched by $\mu$. By the definition of a $\pi_*$-maximal matching, $(M, \mu)$ is fully specified by $\{ (i, \mu(i)) \}_{i \in M'}$. Moreover, $|M'| = |\mu(M')| = d$ for $(M, \mu) \in \cM(d)$. The number of ways to choose $M'$ and $\mu(M')$ is $\binom{n}{d}^2$, and the number of potential bijections between $M'$ and $\mu(M')$ is \emph{at most} $d!$. The reason why this is an upper bound (and not an equality) is that based on the choice of $M'$ and $\mu(M')$, certain bijections may not be legal (e.g., if $\mu(M') = \pi_*(M')$ then the bijection must be chosen so that $\mu(i) \neq \pi_*(i)$ for $i \in M'$). Putting everything together shows that 
\[
| \cM(d) | \le d! \binom{n}{d}^2 \le d! \left( \frac{n^d}{d!} \right)^2 = \frac{n^{2d}}{d!}. \qedhere
\]
\end{proof}

Using Lemmas \ref{lemma:maximal_matching} and \ref{lemma:Md_size}, Cullina, Kiyavash, Mittal and Poor \cite{cullina2020partial} proved the following result, which provides conditions under which any $k$-core matching is guaranteed to be correct. Since the proof is short, we provide it here for completeness. We emphasize that the following result applies to any distribution over a pair of graphs $(G_1, G_2)$ on $n$ vertices; subsequently, in Section~\ref{sec:k-core-SBM}, we will apply this result (and Corollary~\ref{cor:k_core_correctness}) to correlated SBMs.   

\begin{lemma}
\label{lemma:k_core_correctness}
Let $(G_1, G_2)$ be a pair of random graphs on the vertex set $[n]$ with ground-truth matching $\pi_*$. For any positive integer $k$, define the quantity
\[
\xi : = \max\limits_{1 \le d \le n} \max\limits_{(M, \mu) \in \cM(d) } \p \left( f(\mu) \ge k d \right)^{1/d}.
\]
Then
\[
\p \left( \exists (M, \mu) : \mathrm{d}_{\min}( G_1 \land_{\mu} G_2) \ge k \text{ and } \exists i \in M \text{ s.t. } \mu(i) \neq \pi_*(i) \right) \le e^{n^2 \xi} - 1. 
\]
\end{lemma}
We note, as a continuation of Remark~\ref{rem:Md}, 
that while $\pi_{*}$ appears in the definition of $\xi$, 
in fact, 
$\xi$ is a deterministic constant that does not depend on $\pi_{*}$. 
\begin{proof}
We can rewrite the probability of interest by segmenting the event $\{\exists i \in M \text{ s.t. } \mu(i) \neq \pi_*(i)\}$ according to the number of discrepancies between $\mu$ and $\pi_*$:
\begin{multline*}
\p \left( \exists (M, \mu) : \mathrm{d}_{\min}( G_1 \land_{\mu} G_2) \ge k \text{ and } \exists i \in M \text{ s.t. } \mu(i) \neq \pi_*(i) \right) \\
\begin{aligned}
&= \sum\limits_{d = 1}^n \p \left( \exists (M, \mu) : \mathrm{d}_{\min}(G_1 \land_{\mu} G_2) \ge k \text{ and } |\{i \in M: \mu(i) \neq \pi_*(i) \}| = d \right) \\
& \stackrel{(a)}{\le} \sum\limits_{d = 1}^n \p \left( \exists (M, \mu) : f(\mu) \ge kd \text{ and } | \{i \in M: \mu(i) \neq \pi_*(i) \}| = d \right) \\
& \stackrel{(b)}{=} \sum\limits_{d = 1}^n \p \left( \exists (M, \mu) \in \cM(d) : f(\mu) \ge kd \right),
\end{aligned}
\end{multline*}
where $(a)$ follows since any $k$-core matching is also a weak $k$-core matching and $(b)$ follows from the dominance of $\pi_{*}$-maximal matchings established in Lemma~\ref{lemma:maximal_matching}. 
By a union bound we have that 
\[
\p \left( \exists (M, \mu) \in \cM(d) : f(\mu) \ge kd \right) 
\leq | \cM(d) |\max_{(M, \mu) \in \cM(d)} \p \left( f(\mu) \ge kd \right) 
\leq \frac{(n^2 \xi)^d }{d!},
\]
where in the second inequality we used Lemma~\ref{lemma:Md_size} and the definition of $\xi$. 
The conclusion follows by combining the two displays above and summing over $d \in [n]$. 
\end{proof}

Lemma~\ref{lemma:k_core_correctness} implies the following useful corollary. 

\begin{corollary}
\label{cor:k_core_correctness}
Recalling the definitions 
of $(\wh{M}, \wh{\mu})$ from Algorithm~\ref{alg:k-core} 
and 
of $\xi$ from Lemma~\ref{lemma:k_core_correctness}, 
we have that 
\[
\p \left( \widehat{M} \text{ is the $k$-core of } G_1 \land_{\pi_*} G_2 \text{ and } \wh{\mu}\{\widehat{M}\} = \pi_*\{\widehat{M}\} \right) \ge 2 - \mathrm{exp}(n^2 \xi).
\]
\end{corollary}

\begin{proof}
Let $M_*$ be the vertex set of the $k$-core of $G_1 \land_{\pi_*} G_2$
and observe that 
$(M_{*}, \pi_{*} \{ M_{*} \} )$ is a $k$-core matching. 
Let us now define the event 
$$
\cH : = \left \{ \text{for all $k$-core matchings $(M, \mu)$, $\mu(i) = \pi_*(i)$ for all $i \in M$} \right\}.
$$
We first claim that $\cH$ implies that $\widehat{M}$ is the $k$-core of $G_1 \land_{\pi_*} G_2$, and thus 
\[
\p \left( \widehat{M} \text{ is the $k$-core of } G_1 \land_{\pi_*} G_2 \text{ and } \wh{\mu}\{\widehat{M}\} = \pi_*\{\widehat{M}\} \right) \ge \p(\cH).
\] 
To show this, first note that $\cH$ implies $\wh{\mu}\{\widehat{M}\} = \pi_*\{\widehat{M}\}$. To show that $\widehat{M}$ is the $k$-core of $G_1 \land_{\pi_*} G_2$, suppose there existed a $k$-core matching $(M, \mu)$ for which $|M \setminus M_*| > 0$. 
On the event~$\cH$, we have that $\mu\{M\} = \pi_*\{M\}$, which implies that the subgraph of $G_1 \land_{\pi_*} G_2$ corresponding to $M$ has minimum degree $k$.
This in turn implies that the subgraph of $G_1 \land_{\pi_*} G_2$ corresponding to $M \cup M_*$ has minimum degree $k$, which contradicts the maximality of $M_*$. Hence on $\cH$, $(M_{*},\pi_{*} \{ M_{*} \})$ is the maximum $k$-core matching. To conclude, note that Lemma~\ref{lemma:k_core_correctness} implies that
\[
\p(\cH) \ge 1 - (\mathrm{exp}(n^2 \xi) -1) = 2 - \mathrm{exp}(n^2 \xi). \qedhere
\]
\end{proof}

\subsection{Correctness of the $k$-core estimator for correlated SBMs}\label{sec:k-core-SBM}
The main result of this section is the following lemma.
\begin{lemma}\label{lemma:k_core_sbm_v2}
Fix constants $\alpha, \beta > 0$ and $s \in [0,1]$. Let $(G_1, G_2) \sim \CSBM\left(n, \frac{\alpha \log n}{n}, \frac{\beta \log n}{n}, s \right)$. Let $M_*$ be the set of vertices of the $13$-core in the graph $G_1 \land_{\pi_*} G_2$.   Let $(\widehat{M},\wh{\mu})$ be the output of Algorithm \ref{alg:k-core}, with $k=13$. Then $\mathbb{P}\left((\widehat{M}, \wh{\mu}) = (M_*, \pi_*\{M_*\}) \right) = 1 - o(1).$
\end{lemma}
\begin{remark}\label{remark:M_hat_star_swap}
Therefore, any result that holds with high probability for $(M_*, \pi_*\{M_*\})$ also holds with high probability for $(\widehat{M}, \wh{\mu})$, so we can effectively replace $(\widehat{M}, \wh{\mu})$ by $(M_*, \pi_*\{M_*\})$ in any analysis. 
\end{remark}

In light of Corollary~\ref{cor:k_core_correctness}, the strategy to show that the $k$-core estimator matches the $k$-core of $G_1 \land_{\pi_*} G_2$ is to bound $\xi$ for correlated SBMs. Previously, Cullina, Kiyavash, Mittal and Poor~\cite{cullina2020partial} proved a version of Lemma~\ref{lemma:k_core_sbm_v2} for correlated Erd\H{o}s--R\'{e}nyi graphs with constant average degree. Their methods provided a fairly tight characterization of $\xi$ in that setting using generating functions corresponding to $f(\mu)$. While it is natural to expect that it is possible to generalize their arguments to correlated SBMs, our proof presents a simpler, looser approach that avoids dealing with complicated generating functions. The proof of Lemma~\ref{lemma:k_core_sbm_v2} relies on the following result, which bounds $\mathbb{P}(f(\mu) \geq kd)$ for SBMs with general parameters.

\begin{lemma}\label{lemma:k_core_sbm_general}
Fix $n > 0$. Given parameters $p, q, s \in [0,1]$, let $(G_1, G_2) \sim \CSBM\left(n, p, q, s\right)$. 
Let $d, k \in [n]$. For any matching $(M, \mu) \in \cM(d)$ and any $\theta > 0$, we have that 
\[
\mathbb{P}\left(f(\mu) \geq k d \right) \leq 3  \exp \left[d \left(-\theta k + \frac{1}{2} s^2 \gamma \left(e^{2\theta} -1 \right) + n (s \gamma)^2 \left(e^{6\theta} -1 \right)  \right) \right],
\]
where $\gamma = \max\{p, q\}$. 
\end{lemma}

\begin{proof}
Define the following sets: 
\begin{align*}
\cA(\mu) & : = \left \{ (i,j) \in [n]^2: i \in M \text{ and } \mu(i) \neq \pi_*(i) \right \}; \\
\cB(\mu) & : = \left \{ (i,j) \in [n]^2 : \mu(i) = \pi_*(j) \text{ and } \mu(j) = \pi_*(i) \right \}; \\
\cC(\mu) & : = \cA(\mu) \setminus \cB(\mu). 
\end{align*}
In words, $\cA(\mu)$ is the set of vertex pairs where one of the endpoints is mismatched by $\mu$; 
$\cB(\mu)$ is the set of vertex pairs in $\cA(\mu)$ where the vertices are transposed by $\mu$, resulting in a {\it correctly} matched edge; 
and $\cC(\mu)$ is the set of remaining vertex pairs in $\cA(\mu)$ that induce {\it incorrectly} matched edges. It is useful to note that for $(i,j) \in \cB(\mu)$, both $i$ and $j$ are one of the $d$ misclassified vertices, so~$|\cB(\mu)| \leq d$. 
Furthermore, a simple counting argument shows that $| \cA(\mu) | \le d n$.

Using the sets $\cA(\mu)$, $\cB(\mu)$, and $\cC(\mu)$, we can write
\begin{align*}
f(\mu) & = \sum\limits_{i \in M: \mu(i) \neq \pi_*(i)} \mathrm{deg}_{G_1 \land_{\mu} G_2} (i) 
= \sum\limits_{(i,j) \in \cA(\mu)} A_{ij} B_{\mu(i) \mu(j)} \\
& = \sum\limits_{(i,j) \in \cB(\mu)} A_{ij} B_{\mu(i) \mu(j)} + \sum\limits_{(i,j) \in \cC(\mu)} A_{ij} B_{\mu(i) \mu(j)}\\
&= 2\sum\limits_{(i,j) \in \cB(\mu): i < j} A_{ij} B_{\mu(i) \mu(j)} + \sum\limits_{(i,j) \in \cC(\mu)} A_{ij} B_{\mu(i) \mu(j)}. 
\end{align*} 
Recall that 
$\cE^+(\boldsymbol{\sigma_{*}})$ is the set of intra-community vertex pairs and $\cE^-(\boldsymbol{\sigma_{*}})$ is the set of inter-community vertex pairs.
For $(i,j) \in \cB(\mu)$, we have, conditionally on $\boldsymbol{\sigma_{*}}$, that
\[
A_{ij} B_{\mu(i) \mu(j)} = A_{ij} B_{\pi_*(i) \pi_*(j)} \sim \begin{cases}
\mathrm{Bern} \left( s^2 p \right) &\text{ if } (i,j) \in \cE^+(\boldsymbol{\sigma_{*}}), \\
\mathrm{Bern} \left( s^2 q \right) &\text{ if } (i,j) \in \cE^-(\boldsymbol{\sigma_{*}}).
\end{cases}
\]
Recalling that $\gamma = \max \{p, q \}$,  let $X_{\cB} \sim \mathrm{Bin} \left(\frac{1}{2} | \cB(\mu)|, s^2 \gamma \right)$. We then have the stochastic domination
$$
\sum\limits_{(i,j) \in \cB(\mu): i < j} A_{ij} B_{\mu(i) \mu(j)} \stleq X_{\cB}. 
$$

The analysis of the summation over $\cC(\mu)$ is more challenging since the terms in the summation are correlated. To handle the correlation, we split the summation into three parts, where within each part, all terms are independent. To this end, consider two unordered pairs $\{i,j\}$ and $\{\ell,m\}$ such that $\{i,j\} \neq \{\ell,m\}$. Note that the variables $A_{i,j} B_{\mu(i), \mu(j)}$ and $A_{\ell, m} B_{\mu(\ell), \mu(m)}$ are dependent if and only if 
\begin{equation}
\{\mu(i), \mu(j)\} = \{\pi_*(\ell), \pi_*(m)\} ~~~\mathrm{or}~~~ \{\mu(\ell), \mu(m)\} = \{\pi_*(i), \pi_*(j)\}. \label{eq:dependency}
\end{equation}
We can now construct a graph $H$ with vertex set $\{\{i,j\} : (i,j) \in \cC(\mu)\}$ which captures dependencies between terms of the form $A_{i,j} B_{\mu(i), \mu(j)}$,  
so there is an edge between $\{i,j\}$ and $\{\ell, m\}$ if and only if~\eqref{eq:dependency} holds. Note that each unordered pair $\{i,j\}$ may correspond to one or both of the ordered pairs $(i,j)$ and $(j,i)$. By construction, $A_{i,j} B_{\mu(i), \mu(j)}$ and $A_{\ell, m} B_{\mu(\ell), \mu(m)}$ are independent if and only if there is no edge between $\{i,j\}$ and $\{\ell, m\}$ in $H$. Furthermore, if $S$ is a subset of the vertices of $H$ such that no two elements of $S$ are connected by an edge, then the collection of random variables $\{A_{i,j} B_{\mu(i), \mu(j)} : \{i,j\} \in S\}$ are mutually independent.

Note further that any unordered pair $\{i,j\}$ has at most two neighbors in $H$, so $H$ is 3-colorable. Let $\{ \cC_1, \cC_2, \cC_3 \}$ be the partition of $\{\{i,j\} : (i,j) \in \cC(\mu)\}$ induced by the coloring. Since there are no edges between any two elements of $\cC_1$, we can express 
the sum corresponding to $\cC_{1}$ 
as a sum of two independent binomial random variables:
\begin{align*}
X_{\cC_1} : = \sum\limits_{\{i,j\} \in \cC_1} A_{i,j} B_{\mu(i), \mu(j)} & \stackrel{d}{=} \mathrm{Bin} \left( | \cC_1 \cap \cE^+(\boldsymbol{\sigma_{*}}) |, \left( s p \right)^2 \right) + \mathrm{Bin} \left( | \cC_1 \cap \cE^-(\boldsymbol{\sigma_{*}})|, \left(s q \right)^2 \right) \\
& \stleq \mathrm{Bin} \left( | \cC_1|, \left( s \gamma \right)^2 \right) 
\stleq \mathrm{Bin} \left( | \cC(\mu) |, \left(s \gamma \right)^2 \right).
\end{align*}
Similarly, we have that  
\begin{align*}
X_{\cC_2} : = \sum\limits_{\{i,j\} \in \cC_2} A_{i,j} B_{\mu(i), \mu(j)} & \stleq \mathrm{Bin} \left( | \cC(\mu) |, \left( s \gamma \right)^2 \right), \\
X_{\cC_3} : = \sum\limits_{\{i,j\} \in \cC_3} A_{i,j} B_{\mu(i), \mu(j)} & \stleq \mathrm{Bin} \left( | \cC(\mu) |, \left(s \gamma \right)^2 \right).
\end{align*}
Accounting for the fact that each vertex $\{i,j\}$ in $H$ may correspond to one or two ordered pairs in~$\cC(\mu)$, we obtain \[\sum_{(i,j) \in \cC(\mu)} A_{ij} B_{\mu(i) \mu(j)} \stleq 2\left(X_{\cC_1} + X_{\cC_2} + X_{\cC_3} \right), \]
Putting everything together, using a union bound we can write
\begin{equation}
\p ( f(\mu) \geq kd ) = \p \left( 2\left(X_{\cB} + X_{\cC_1} + X_{\cC_2} + X_{\cC_3}\right) \geq kd \right) \le \sum\limits_{i = 1}^3 \p \left( 2X_{\cB} + 6X_{\cC_i} \geq kd \right). \label{eq:3_way_union}
\end{equation}
The terms in the summation can be handled with a Chernoff bound. For $i \in \{1,2,3\}$ and any $\theta > 0$,
\begin{align*}
\p \left(2X_{\cB} + 6X_{\cC_i} \geq kd \right) &= \p \left( e^{\theta\left( 2X_{\cB} + 6X_{\cC_i}\right)} \geq e^{\theta kd} \right)\\
&\leq e^{-\theta kd} \mathbb{E}\left[e^{\theta\left( 2X_{\cB} + 6X_{\cC_i}\right)} \right]\\
&=e^{-\theta kd} \mathbb{E}\left[e^{2\theta X_{\cB}}\right] \mathbb{E}\left[e^{6\theta X_{\cC_i}} \right].
\end{align*}
If $Z \sim \mathrm{Bin}(m,p)$, then $\mathbb{E}\left[e^{t Z}\right] = \left(1 + p(e^t -1) \right)^m \leq \exp\left(mp(e^t -1) \right)$. Using this fact, along with the bounds $|\cB(\mu)| \leq d$ and $|\cC(\mu)| \leq |\cA(\mu)| \leq dn$, we obtain that 
\begin{align*}
\p \left( 2X_{\cB} + 6X_{\cC_i} \geq kd \right)
&\leq \exp\left(-\theta kd + \frac{1}{2}|\cB(\mu)| s^2 \gamma \left(e^{2\theta} -1 \right) + |\cC(\mu)| (s \gamma)^2 \left(e^{6\theta} -1 \right)  \right)\\
&\leq \exp \left[d \left(-\theta k + \frac{1}{2} s^2 \gamma \left(e^{2\theta} -1 \right) + n (s \gamma)^2 \left(e^{6\theta} -1 \right)  \right) \right].
\end{align*}
Substituting into \eqref{eq:3_way_union}, we obtain the desired result. 
\end{proof}

\begin{proof}[Proof of Lemma \ref{lemma:k_core_sbm_v2}]
Fix $d \in [n]$ and a matching $(M, \mu) \in \cM(d)$. By Lemma \ref{lemma:k_core_sbm_general}, we have that 
\[
\mathbb{P}\left(f(\mu) \geq 13 d \right) \leq 3  \exp \left[d \left(-13\theta + \frac{1}{2} s^2 \gamma \left(e^{2\theta} -1 \right) + n (s \gamma)^2 \left(e^{6\theta} -1 \right)  \right) \right],
\]
where $\gamma = \max\{\alpha,\beta\} \frac{\log n}{n}$. Setting $\theta = c\log n$ for a constant $c \in \left( \frac{2}{13}, \frac{1}{6}\right)$, we thus have that 
\[
\mathbb{P}\left(f(\mu) \geq 13d \right) \leq 3  \exp \left[d \left(-13 c \log n + o(1) \right) \right].
\]
Therefore, $\xi \leq 3 \exp \left(-13 c \log n + o(1) \right) \leq 4 n^{-13 c}$ for $n$ sufficiently large. Since $13 c > 2$, this implies that $n^2 \xi = o(1)$. The proof is complete by Corollary~\ref{cor:k_core_correctness}.
\end{proof}

\subsection{The size of the $k$-core of $G_1 \land_{\pi_*} G_2$}
\label{subsec:k_core_size}

We follow the method of \Luczak \cite{Luczak1991} in upper bounding $|F|$, the number of vertices outside of the $k$-core of an SBM. In particular, we refine the results of \Luczak for random graphs in the logarithmic degree regime. In the remainder of this section, unless otherwise stated, we assume that $G \sim \SBM(n, \alpha \log (n) / n, \beta \log (n) / n)$, where $\alpha, \beta >0$ are treated as generic parameters. The following result shows that it is unlikely that there exist well-connected small subgraphs of $G$. 

\begin{lemma}
\label{lemma:subgraph_density}
Let $a > 1$ and $\epsilon > 0$ be fixed. Then for $n$ sufficiently large,
$$
\p \left( \text{there exists $\cS \subset [n]$ such that $| \cS | \le n^{1-\epsilon}$ and $G\{ \cS \}$ has at least $a|\cS|$ edges} \right) \le \frac{2}{\log n}.
$$
\end{lemma}

\begin{proof}
Let $\cS$ be an $m$-vertex subset of $[n]$, and let $X_{\cS}$ be the indicator variable that is 1 if the subgraph induced by $\cS$ has at least $am$ edges. 
Set $\gamma : = \max \{ \alpha, \beta \}$. We can bound
\begin{align*}
\E [ X_{\cS} ] & \le \binom{\binom{m}{2}}{am} \left( \gamma  \frac{\log n}{n} \right)^{am}  \le \binom{m^{2}}{am} \left( \gamma  \frac{\log n}{n} \right)^{am}  \le \left( \frac{m}{n} \cdot \frac{\gamma e  \log n}{a} \right)^{am}.
\end{align*}
Summing over all possible $m$-vertex subsets, we have that
\[
\E \left[ \sum\limits_{\cS \subset [n] : |\cS| = m} X_{\cS} \right] \le \binom{n}{m} \left( \frac{m}{n} \cdot \frac{\gamma e  \log n}{a} \right)^{am} \le 
\left( \left( \frac{\gamma e^{1 + 1/a}  \log n}{a} \right)^a \left( \frac{m}{n} \right)^{a-1} \right)^m.
\]
Next, note that if $m \le n^{1 - \epsilon}$ then for $n$ sufficiently large, 
\[
\left( \frac{\gamma e^{1 + 1/a}  \log n}{a} \right)^a \left( \frac{m}{n} \right)^{a-1} \le \frac{1}{\log n}.
\]
Hence, for $n$ sufficiently large we have that
\[
\E \left[ \sum\limits_{\cS \subset [n] : | \cS | \le n^{1 - \epsilon}} X_{\cS} \right] \le \sum\limits_{m = 1}^{n^{1 - \epsilon}} \left( \frac{1}{\log n} \right)^m \le \frac{1}{\log n} \left( 1 - \frac{1}{\log n} \right)^{-1} \le \frac{2}{\log n}.
\]
Finally, Markov's inequality implies 
the claim. 
\end{proof}

Our next result uses Lemma \ref{lemma:subgraph_density} to study Algorithm \ref{alg:Luczak}.

\begin{lemma}[\Luczak expansion]
\label{lemma:luczak} 
Let $G \sim \SBM(n, \alpha \log (n) / n, \beta \log (n) / n)$ for fixed $\alpha, \beta > 0$, and fix $c \in (0,1)$. 
Then, for $n$ sufficiently large, the following holds with probability $1-o(1)$. 
For every $U \subseteq [n]$ such that $|U| \leq n^{c}$, 
denoting by $\overline{U}$ the output of Algorithm~\ref{alg:Luczak} on input $(G,U)$, 
we have that 
$|\overline{U}| \le 3 |U|$ and for all $v \in [n] \setminus \overline{U}$, $v$ has at most one neighbor in $\overline{U}$ with respect to $G$. 
\end{lemma}

A version of Lemma \ref{lemma:luczak} specialized to the analysis of $k$-core sizes was previously proven in~\cite{Luczak1991}. Since the proof is short and the \Luczak expansion appears as a subroutine in our community recovery algorithm, we provide the full proof of Lemma \ref{lemma:luczak} here. 

\begin{proof}[Proof of Lemma \ref{lemma:luczak}]
Let $U_m = \overline{U}$ be the final set produced by Algorithm~\ref{alg:Luczak} on input $(G,U)$. From the construction, it is clear that $U_m \supseteq U_0$. Moreover, if $v \in [n] \setminus U_m$, then $v$ cannot have more than one neighbor in~$U_m$, else we could construct another set $U_{m + 1}$, contradicting the maximality of $U_m$. This holds for any $U \subseteq [n]$. 

Now let $\epsilon := (1-c)/2$ and let $n$ be sufficiently large such that $3 n^{c} \leq n^{1-\epsilon}$. 
Let $\cH$ denote the event described in Lemma~\ref{lemma:subgraph_density} with this $\epsilon$ and $a = 4/3$. 
Since $\p(\cH) = 1-o(1)$, 
to prove the claim it suffices to show that 
on the event $\cH$ we have that 
$|\overline{U}| \le 3 |U|$ 
for every $U \subseteq [n]$ such that $|U| \leq n^{c}$. 

Suppose that for some $U \subseteq [n]$ such that $|U| \leq n^{c}$ 
we have that $|U_{m}| > 3 |U|$. 
Then, since exactly one vertex is added in each step of the construction, there must exist $1 \le \ell \le m$ such that $|U_\ell| = 3 |U|$. 
Let $E_\ell$ denote the number of edges in $G\{U_\ell \}$. Since each vertex added in the construction introduces at least two new edges, we have that
\[
E_\ell \ge 2 \ell = 2(|U_\ell| - |U|) = \frac{4}{3} |U_\ell |.
\]
However, on the event $\cH$ this is not possible, since 
$|U_\ell| \le 3 n^c \leq n^{1-\epsilon}$. 
\end{proof}

Let $F$ be the set of vertices outside the $k$-core of $G$. A straightforward application of Lemma~\ref{lemma:luczak} allows us to bound $|F|$, as was done in \cite{Luczak1991}.

\begin{lemma}\label{lemma:number-outside-core} 
Let $G \sim \SBM(n, \alpha \log (n) / n, \beta \log (n) / n)$ for fixed $\alpha, \beta > 0$. 
Fix $k \geq 1$. 
With probability $1 - o(1)$, we have that $|F| \le n^{1 - (\alpha + \beta)/2 + o(1)}$. 
\end{lemma}

\begin{proof}
Define $U$ to be the set of vertices with degree at most $k$ in $G$, and let $\overline{U} \supseteq U$ be the set produced by Algorithm~\ref{alg:Luczak}. 
We claim that $G \{ [n] \setminus \overline{U} \}$ has minimum degree at least $k$. 
To see why, note that if $v \in [n] \setminus \overline{U}$, then $v \notin U$, which implies that $v$ has degree at least $k + 1$ in $G$. However $v$ can have at most one neighbor in $\overline{U}$ by construction, so $v$ must have at least $k$ neighbors in $[n] \setminus \overline{U}$, and the claim follows. 

Since the $k$-core is the vertex set corresponding to the \emph{largest} induced subgraph with minimum degree at least $k$, 
we have the bound 
$|F | \le |\overline{U}|$. 
Furthermore, by Lemma~\ref{lemma:luczak} we have, with high probability, that $|\overline{U}| \le 3 |U|$. 

It thus remains to bound $|U|$. Define the event 
$\cH := \{ n / 2 - n^{3/4} \le |V^+ |, |V^-| \le n/2 + n^{3/4} \}$.
Bounding the expectation of $|U| \mathbf{1}( \cH)$ and letting $\gamma : = \max \{ \alpha, \beta \}$, we have
\begin{align*}
\E [ |U|\mathbf{1}(\cH) ] 
& \le n \sum\limits_{i = 0}^{k} \sum\limits_{j = 0}^{i} \binom{(1 + o(1))\frac{n}{2}}{j} \binom{(1 + o(1))\frac{n}{2}}{i - j} p^j (1 - p)^{(1-o(1))\frac{n}{2} -j} q^{i - j} (1 - q)^{(1-o(1))\frac{n}{2} - i + j} \\
& \le 2 n\left( 1 -  \frac{ \alpha \log n}{n} \right)^{(1 - o(1))\frac{n}{2}} \left( 1 -  \frac{ \beta \log n}{n} \right)^{(1 - o(1)) \frac{n}{2}}  \sum\limits_{i = 0}^k \sum\limits_{j = 0}^i \left( \frac{(1 + o(1))n}{2} \right)^i \left(  \frac{\gamma \log n}{n} \right)^i \\
& \le 2n^{1 - (\alpha + \beta)/2 + o(1)} \sum\limits_{i = 0}^k (i+1) \left( \frac{ (1 + o(1)) \gamma \log n}{2} \right)^i \\
& \le 2(k+1)^2 \left( \frac{\gamma \log n}{2} \right)^k n^{1 - (\alpha + \beta) / 2 + o(1)}.
\end{align*}
Turning to the probability of interest, we obtain that 
\[
\p \left( |U| > (\log n) \E [ |U| \mathbf{1}( \cH)] \right) 
\le \p \left( | U | \mathbf{1} ( \cH) > (\log n) \E [ |U| \mathbf{1} (\cH)] \right) + \p ( \cH^c) 
\le \frac{1}{\log n} + o(1) = o(1),
\]
where we have used Markov's inequality to bound the first term on the right hand side in the first line, and Lemma \ref{lemma:F} to bound $\p(\cH^c)$. Putting everything together, we have shown that with probability $1 - o(1)$, we have that 
$|F| \le 3 | U | \le n^{1 - (\alpha + \beta) / 2 + o(1)}$. 
\end{proof}

\section{Labeling the $k$-core matching (Algorithm \ref{alg:labeling-k-core})}\label{sec:labeling_proofs}

\subsection{Almost exact recovery in a single SBM}

For the purposes of this section, we assume that $G \sim \SBM(n, \alpha \log(n) / n, \beta \log (n) / n)$, where $\alpha, \beta>0$ are generic parameters and $G$ has communities of size $(1 + o(1)) n / 2$.  It is well known that in the logarithmic degree regime, almost exact recovery of communities is always possible from a single SBM (see, e.g., \cite{Abbe_survey}). In our case, it is of further interest to precisely control the size of the error set---in particular, we require the error set to be as small as possible in an information-theoretic sense---as well as the \emph{geometry} of the error set, since it will be subsequently used in later parts of Algorithm~\ref{alg:labeling-k-core}. To properly control the size and geometry of the error set, we leverage the community recovery algorithm of 
Mossel, Neeman, and Sly 
\cite{mossel2016consistency} (Algorithm~\ref{alg:MNS}). The purpose of the refinements in Algorithm~\ref{alg:MNS} is to characterize the geometry of the error set, which is what the main result of this subsection describes. Before stating the result, we define some useful notation. Let $\gamma = \max\{\alpha, \beta\}$, and define, for a vertex $v$ in $G$, the quantity 
$$
\maj_G(v) : = | \cN(v) \cap V^{\sigma_*(v)} | - | \cN(v) \cap V^{- \sigma_*(v)} |.
$$
If $\alpha > \beta$, define the set 
\begin{equation}
I_{\epsilon}(G) : = \left \{ v \in [n] : \maj_G(v) \le \epsilon \log n \text{ or } | \cN(v) | \ge 100 \max \{1, \gamma \} \log n \right \}. \label{eq:I_epsilon_definition-1}
\end{equation}
If $\alpha < \beta$, let
\begin{equation}
I_{\epsilon}(G) : = \left \{ v \in [n] : - \maj_G(v) \le \epsilon \log n \text{ or } | \cN(v) | \ge 100 \max \{1, \gamma \} \log n \right \}. \label{eq:I_epsilon_definition-2}
\end{equation}
If $\alpha > \beta$, we say that a vertex $v$ has a \emph{$k$-majority} if $\maj_G(v) \geq k$. If $\alpha < \beta$, we say that a vertex $v$ has a $k$-majority if $-\maj_G(v) \geq k$. 

\begin{lemma}
\label{lemma:MNS}
Let $\epsilon > 0$. With probability $1 - o(1)$, Algorithm~\ref{alg:MNS} on input $(G, \alpha, \beta, \epsilon)$ correctly classifies all vertices in $[n] \setminus I_{\epsilon}(G)$.
\end{lemma}

\begin{proof}
Our proof is adapted from \cite[Proposition 4.3]{mossel2016consistency}, with minor changes due to the more general parameter regime we consider. We will assume that $\alpha > \beta$; the case $\alpha < \beta$ is similar, and full details can be found in \cite{mossel2016consistency}.

Fix $i \in [m]$ and $v \in (U_i \cap V^+) \setminus I_\epsilon$, and define $k_+ : = |\cN(v) \cap V^{+} |$ and $k_- : = | \cN(v) \cap V^{-}|$ to be the number of same-community and different-community neighbors of $v$, respectively. We also define $k_{+,i} : =  | \cN(v) \cap V^{+} \cap U_i |$ and $k_{-,i} : = | \cN(v) \cap V^{-} \cap U_i |$ to be the number of same-community and different-community neighbors of $v$ in $U_i$, respectively. Finally, we define $k_{+, \neg i} : = |( \cN(v) \cap V^{+} ) \setminus U_i | = k_+ - k_{+, i}$ and $k_{-, \neg i} : = | ( \cN(v) \cap V^{- } ) \setminus U_i | = k_- - k_{-, i}$ to be the number of same-community and different-community neighbors of $v$ in $[n] \setminus U_i$, respectively. Note in particular that conditioned on $k_+$ and $k_-$, we have that $k_{+, i} \sim \mathrm{Bin}( k_+, 1/m)$ and $k_{-, i} \sim \mathrm{Bin} ( k_-, 1/m)$.

We first show that $k_{+, \neg i} - k_{-, \neg i} \ge \frac{\epsilon}{2} \log n$ with high probability. To this end, we have the lower bound
$$
k_{+, \neg i} - k_{-, \neg i} = k_+ - k_- - (k_{+,i} - k_{-,i}) \ge \epsilon \log n - ( k_{+,i} + k_{-,i}),
$$
where the final inequality uses $k_+ - k_- \ge \epsilon \log n$ for $v \notin I_\epsilon$. It therefore suffices to show that $k_{+,i} + k_{-,i} \le \frac{\epsilon}{2} \log n$. Noticing that $k_{+,i} + k_{-,i} \sim \Bin(| \cN(v) |, 1/m)$, which is stochastically dominated by $\Bin( 100 \max \{ 1, \gamma \} \log n , 1/m )$ for $v \notin I_\epsilon$, a Chernoff bound implies that 
\begin{align*}
\p \left( k_{+,i} + k_{-,i} \ge \frac{\epsilon}{2} \log n \right) & \le \inf\limits_{\theta > 0} \left \{ e^{- \theta \frac{\epsilon}{2}\log n} \left( 1 + \frac{1}{m} ( e^\theta - 1) \right)^{ 100 \max \{ 1, \gamma\} \log n} \right \} \\
& \le \mathrm{exp} \left( - \sup\limits_{\theta > 0} \left \{ \frac{\theta \epsilon}{2} - \frac{100 \max \{ 1, \gamma \}}{m} ( e^\theta - 1) \right \} \log n \right) \\
& \le \mathrm{exp} \left( - \left( \log \left( \frac{\epsilon m}{200 \max \{ 1, \gamma \}} \right) - 1 \right) \frac{\epsilon}{2} \log n \right) = o(n^{-1}).
\end{align*}
In the display above, the first inequality on the third line uses  $\theta = \log ( \epsilon m  / (200 \max \{ 1, \gamma \}))$ and the final equality follows from our choice of $m$ in Algorithm \ref{alg:MNS}. It therefore holds that $k_{+, \neg i} - k_{-, \neg i} \ge \frac{\epsilon}{2} \log n$ for all $v \notin I_\epsilon$ with probability $1 - o(1)$.

From this point, the proof of the lemma is identical to the one provided in \cite{mossel2016consistency}; we state a brief overview of the proof here and defer the interested reader to \cite[Proposition 4.3]{mossel2016consistency} for the details. Consider the random variables $X^+ : = | \cN(v) \cap U_{i,+} \cap V^-|$ and $X^{-} : = | \cN(v) \cap U_{i,-} \cap V^+|$, which correspond to the number of neighbors of $v$ in $[n] \setminus U_i$ that are misclassified by the partition $(U_{i,+}, U_{i,-})$. We can write 
\begin{align*}
| \cN(v) \cap U_{i,+} | & = k_{+, \neg i} - X^+ + X^-, \\
| \cN(v) \cap U_{i,-} | & = k_{-, \neg i} - X^- + X^+.
\end{align*}
It follows that $v$ is correctly classified if $|X^+ - X^-| < | k_{+, \neg i} - k_{-, \neg i} |/2$. Moving forward, our goal is to show that this inequality holds with high probability.

Following \cite{mossel2016consistency}, we let $E^- : = | U_{i,-} \cap V^+|$ and $E^+ : = | U_{i,+} \cap V^-|$ denote the total number of vertices of each type misclassified by the partition $(U_{i, +}, U_{i,-})$ that are potential neighbors of $v$. In particular, $E^+ = o(n)$ and $E^- = o(n)$. As this is a partition corresponding to the structure of $G \{[n] \setminus U_i\}$, the neighbors of $v$ in $[n] \setminus U_i$ are independent of the partition. Moreover, the vertices in $U_{i,-} \cap V^+$ are equally likely to be neighbors of $v$, with the same holding for vertices in $U_{i,+} \cap V^-$. We can therefore generate $X^-$ by randomly choosing $k_{+, \neg i}$ vertices of $V^+ \setminus U_{i}$ without replacement, with $X^-$ being equal to the number of sampled vertices in $U_{i, -} \cap V^+$. Conditioned on $E^-$, $E^+$, $k_{+, \neg i}$, and $k_{-, \neg i}$, we thus have the distributional representations
\begin{align*}
X^- & \stackrel{d}{=} \mathrm{HyperGeom} ( |V^+ \setminus U_i|, k_{+, \neg i}, E^-) \hspace{0.5cm} \text{and} \hspace{0.5cm} X^+ \stackrel{d}{=} \mathrm{HyperGeom} ( |V^- \setminus U_i |, k_{-, \neg i}, E^+ ).
\end{align*}
From here, one can utilize properties of Hypergeometric random variables to show that with high probability,
$$
\left| \E \left[ X^- \, \middle| \, |V^+ \setminus U_i |, k_{+, \neg i}, E^- \right] - \E \left[ X^+ \, \middle| \, | V^- \setminus U_i |, k_{-, \neg i}, E^+ \right] \right| 
= o(1) | k_{+, \neg i} - k_{-, \neg i} |.
$$
Finally, the lemma follows from showing that $X^+$ and $X^-$ concentrate around their means and by taking a union bound over all vertices. An identical analysis holds for the case where $v \in (U_i \cap V^-) \setminus I_\epsilon(G)$. 
\end{proof}

The next few results further characterize the set $I_{\epsilon}(G)$. 

\begin{lemma}
\label{lemma:vertex_degree_bound}
Denote $\gamma : = \max \{\alpha, \beta\}$. Then for every $\boldsymbol{\sigma_{*}}$ we have that 
$$
\p \left( \forall i \in [n], | \cN(i) | \le 100 \max \{1, \gamma \} \log(n) \, \middle| \, \boldsymbol{\sigma_{*}} \right) \geq 1- n^{-99}.
$$
\end{lemma}

\begin{proof}
Fix $i \in [n]$, and let $X \sim \Bin\left(n, \gamma \log(n)/n\right)$. Then for every $\boldsymbol{\sigma_{*}}$ we have that 
\begin{align*}
\p \left( | \cN(i) | \ge 100 \max \{1, \gamma \} \log n \, \middle| \, \boldsymbol{\sigma_{*}} \right) & \le \p ( X \ge 100 \max \{1, \gamma \} \log n ) \\
& \le \mathrm{exp} \left( - \frac{ ( 100 \max \{1, \gamma \} \log n )^2 }{ 2 \gamma \log n + \frac{200}{3} \max \{1, \gamma \} \log n } \right) \\
& \le \mathrm{exp} \left( - 100 \max \{1, \gamma \} \log n \right)
\le n^{- 100}. 
\end{align*}
The second line uses Bernstein's inequality and the third uses $2 \gamma + \frac{200}{3} \max \{1, \gamma \} \le 100 \max \{1, \gamma \}$. We conclude by taking a union bound.
\end{proof}

\begin{lemma}
\label{lemma:number-incorrectly-classified}
Suppose that $\dchab < 99$. Then for every $\epsilon > 0$ we have that  
$$
\E [ | I_\epsilon(G) | ] \le 3n^{ 1 - \dchab + \epsilon | \log(\alpha / \beta) |}.
$$
\end{lemma}

\begin{proof}
Assume that $\alpha > \beta$. 
Let $\cG$ be the event that 
$n/ 2 - n^{3/4} \le | V^+ |, | V^-| \le n/2 + n^{3/4}$ 
and let $\cH$ be the event defined in Lemma~\ref{lemma:vertex_degree_bound}. 
For any $i \in [n]$, we have that
$$
\p \left( \{ i \in I_\epsilon(G) \} \cap \cH \, \middle| \, \boldsymbol{\sigma_{*}} \right) 
\le \p \left( \maj_G(i) \le \epsilon \log n \, \middle| \, \boldsymbol{\sigma_{*}} \right).
$$
Notice that $\maj_G(i) \stackrel{d}{=} Y - Z$, where $Y, Z$ are independent with 
$$
Y \sim \mathrm{Bin} \left( |V^{\sigma_*(i)}| - 1, \alpha \frac{\log n}{n} \right) \qquad \text{and} \qquad Z \sim \mathrm{Bin} \left( |V^{-\sigma_*(i)} |, \beta \frac{\log n}{n} \right).
$$
Since $|V^+| = (1 - o(1)) n / 2$ and $|V^-| = (1 - o(1)) n / 2$ on the event $\cG$, Lemma \ref{lemma:binomial-tail} implies that
$$
\p \left( \maj_G(i) \le \epsilon \log n \, \middle| \, \boldsymbol{\sigma_{*}} \right) \mathbf{1}(\cG)  
\le n^{ - \dch(\alpha, \beta) + \frac{\epsilon \log ( \alpha / \beta )}{2} + o(1) } 
\le n^{ - \dchab + \epsilon \log (\alpha / \beta)},
$$
where the final inequality holds for $n$ sufficiently large. Putting everything together, we can bound the probability that $i \in I_\epsilon(G)$ as
\begin{align*}
\p ( i \in I_\epsilon(G) ) & \le \p ( \{ i \in I_\epsilon(G) \} \cap \cG \cap \cH ) + \p (\cG^c) + \p( \cH^c) \\
& \le \E \left[ \p \left( \{ i \in I_\epsilon(G) \} \cap \cH \, \middle| \, \boldsymbol{\sigma_*} \right) \mathbf{1} ( \cG ) \right] + 2n^{-99} \\
& \le n^{ - \dchab + \epsilon \log(\alpha / \beta) } + 2 n^{ - 99} \le 3n^{ - \dchab + \epsilon \log ( \alpha / \beta) },
\end{align*}
where the inequality on the second line uses Lemmas~\ref{lemma:F} and~\ref{lemma:vertex_degree_bound} and the final inequality follows from $\dchab - \epsilon \log (\alpha / \beta) < 99$ under our assumption that $\dchab < 99$. Finally, to bound the expectation of $|I_\epsilon(G)|$, we can write
$$
\E [ | I_\epsilon (G) |] = \sum\limits_{i \in [n] } \p ( i \in I_\epsilon(G) ) \le 3n^{1 - \dchab + \epsilon \log (\alpha / \beta)}.
$$
The case $\alpha < \beta$ follows from identical arguments. 
\end{proof}

\begin{lemma}
\label{lemma:I_internal}
If  $0 < \epsilon < \frac{\dchab}{2 |\log (\alpha / \beta)|}$, then
$$
\p \left( \forall i \in [n], | \cN(i) \cap I_\epsilon(G) | \le 2 \left \lceil \dchab^{-1} \right \rceil \, \middle| \, \boldsymbol{\sigma_{*}} \right) = 1 - o(1).
$$
\end{lemma}

\begin{proof}
Let us assume $\alpha > \beta$, let $\cG$ be the event where $n/2 - n^{3/4} \le |V^+ |, |V^- | \le n/2 + n^{3/4}$, and let $\cH$ be the event defined in Lemma \ref{lemma:vertex_degree_bound}. Fix $S \subset [n]$, and assume that $|S|$ is of constant size with respect to $n$. In the calculations below, we will use the representation 
$$
\maj_G(j) = \sigma_*(j) \sum\limits_{k \in \cN(j)} \sigma_*(k).
$$
We can then write
\begin{align*}
\p \left(  \{ S \subset \cN(i) \cap I_\epsilon(G) \} \cap \cH \, \middle| \, \boldsymbol{\sigma_{*}} \right) 
& \le \p \left( \forall j \in S,  i \sim j \text{ and } \sigma_*(j) \sum\limits_{k \in \cN(j) }  \sigma_*(k) \le  \epsilon \log n \, \middle| \, \boldsymbol{\sigma_{*}} \right ) \\
& \le \p \left( \forall j \in S, i \sim j \text{ and } \sigma_*(j) \sum\limits_{k \in \cN(j) \setminus ( S \cup \{i \})}  \sigma_*(k) \le 2 \epsilon \log n \, \middle| \, \boldsymbol{\sigma_{*}} \right),
\end{align*}
where the inequality on the second line is due to $| S \cup \{i \}| \le \epsilon \log n$ for $n$ sufficiently large. To simplify the right hand side, let us introduce the events
$$
E_j : = \left \{ i \sim j \text{ and } \sigma_*(j) \sum\limits_{k \in \cN(j) \setminus ( S \cup \{i \})}  \sigma_*(k) \le 2 \epsilon \log n \right \}, \qquad j \in S.
$$
Notice in particular that conditioned on $\boldsymbol{\sigma_{*}}$, the event $E_j$ depends on the neighbors of $j$ outside of the set $S$, hence the $E_j$'s are conditionally independent with respect to $\boldsymbol{\sigma_{*}}$. We can therefore bound
\begin{equation}
\label{eq:Ej_independence_bound}
\p \left( \{ S \subset \cN(i) \cap I_\epsilon(G)  \} \cap \cH \, \middle| \, \boldsymbol{\sigma_{*}} \right) 
\le \prod\limits_{j \in S} \p \left( E_j \, \middle| \, \boldsymbol{\sigma_{*}} \right).
\end{equation}
To bound $\p\left(E_j \, \middle| \, \boldsymbol{\sigma_{*}} \right)$, 
notice that the events 
$\{i \sim j \}$ and 
$\{ \sigma_*(j) \sum_{k \in \cN(j) \setminus ( S \cup \{i \}) } \sigma_*(k) \le 2 \epsilon \log n \}$ 
are independent conditioned on $\boldsymbol{\sigma_{*}}$. Furthermore, on $\cG$ we have the distributional representation 
$$
\sigma_*(j) \sum\limits_{k \in \cN(j) \setminus (S \cup \{ i \}) } \sigma_*(k) \stackrel{d} = Y - Z,
$$
where, since $|S| = o(n)$, $Y$ and $Z$ are independent with 
$$
Y \sim \mathrm{Bin} \left( (1 - o(1))\frac{n}{2} - 1, \alpha \frac{\log n}{n} \right), \qquad Z \sim \mathrm{Bin} \left((1 - o(1)) \frac{n}{2}, \beta \frac{\log n}{n} \right).
$$
Using $\p \left(i \sim j \, \middle| \, \boldsymbol{\sigma_{*}} \right) \le \gamma \log(n) / n$, as well as Lemma~\ref{lemma:binomial-tail}, shows that 
$$
\p \left( E_j \, \middle| \, \boldsymbol{\sigma_{*}} \right) \mathbf{1}(\cG)
\le \frac{\gamma \log n}{n} \cdot n^{ - \dchab + \epsilon \log(\alpha / \beta) + o(1)} = n^{ - 1 - \dchab + \epsilon \log(\alpha / \beta) + o(1)}.
$$
Combining the bound in the display above with~\eqref{eq:Ej_independence_bound} shows that 
$$
\p \left( \{ S \subset \cN(i) \cap I_\epsilon(G) \} \cap \cH \, \middle| \, \boldsymbol{\sigma_{*}} \right) \mathbf{1}(\cG)  
\le n^{ - |S| ( 1 + \dchab - \epsilon \log(\alpha / \beta) - o(1))}.
$$
We may now take a union bound over subsets $S$ of size $m$ to obtain  
\begin{align*}
\p \left( \{| \cN(i) \cap I_\epsilon(G) | \ge m \} \cap \cH \, \middle| \, \boldsymbol{\sigma_{*}} \right) \mathbf{1}(\cG)
& \le \sum\limits_{S \subset [n] : |S| = m} \p \left(\{ S \subset \cN(i) \cap I_\epsilon(G) \} \cap \cH \, \middle| \, \boldsymbol{\sigma_{*}} \right) \mathbf{1}(\cG) \\
& \le n^{- m \left ( \dch(\alpha, \beta) - \epsilon \log \left( \alpha / \beta \right) - o(1) \right)}.
\end{align*}
The final inequality uses the fact that the number of $m$-element subsets of $[n]$ is at most $n^m$. Taking a union bound over $i \in [n]$ now shows that 
\begin{align}
\p \left( \{ \exists i \in [n] : | \cN(i) \cap I_\epsilon(G) | \ge m \} \cap \cH \, \middle| \, \boldsymbol{\sigma_{*}} \right) \mathbf{1}(\cG) 
&\le n^{1 -m \left ( \dch(\alpha, \beta) - \epsilon \log \left( \alpha / \beta \right) - o(1) \right)} \nonumber \\
&\le n^{1 - m \dchab / 2 + o(1)}, \label{eq:errors_in_neighborhood_conditional_prob}
\end{align}
where the final inequality uses the assumption that $\epsilon \le \frac{\dchab }{2 \log(\alpha / \beta)}$. Moreover, the right hand side is $o(1)$ provided $m > 2\dch(\alpha, \beta)^{-1}$, so we may set $m := \lceil 2 \dchab^{-1} \rceil + 1$ for all of the arguments to hold. The desired result follows from the bound
\begin{align*}
\p ( \exists i \in [n] : | \cN(i) \cap I_\epsilon(G) | \ge m ) & \le \p ( \{ \exists i \in [n]: | \cN(i) \cap I_\epsilon(G) | \ge m \} \cap \cG \cap \cH ) + \p ( \cG^c) + \p ( \cH^c) \\
& \le \E \left[ \p \left( \{ \exists i \in [n] : | \cN(i) \cap I_\epsilon(G) | \ge m \} \cap \cH \, \middle| \, \boldsymbol{\sigma_*} \right) \mathbf{1}(\cG) \right] + o(1) \\
& = o(1),
\end{align*}
where the inequality on the first line is due to a union bound, the inequality on the second line uses the tower rule as well as Lemmas \ref{lemma:vertex_degree_bound} and \ref{lemma:F}, and the final inequality is due to \eqref{eq:errors_in_neighborhood_conditional_prob} as well as our assumption on $\epsilon$. The case $\alpha < \beta$ follows from identical arguments. 
\end{proof}

\subsection{From almost exact to exact recovery in $[n] \setminus \overline{F}$}

In this subsection, our main result is the following.

\begin{lemma}
\label{lemma:alg_outside_Fbar} Suppose that
$$
\left(1 - (1 - s)^{2} \right) \dchab > 1 + 2 \epsilon | \log (\alpha / \beta) | \hspace{1cm} \text{ and } \hspace{1cm}  0 < \epsilon \le \frac{ s \dchab}{4 | \log (\alpha / \beta) |}.
$$
Then, with probability $1 - o(1)$, Algorithm \ref{alg:labeling-k-core} run on the input $(G_1, G_2, (\widehat{M}, \wh{\mu}), \alpha, \beta, s, \epsilon)$ correctly labels all vertices in $[n] \setminus \overline{F}$. 
\end{lemma}

The proof of Lemma~\ref{lemma:alg_outside_Fbar} follows from several intermediate steps. First, we show in Lemmas~\ref{lemma:union_graph_majority} and~\ref{lemma:union_graph_muhat_majority} that all vertices in $(G_1 \lor_{\widehat{\mu}} G_2) \{[n] \setminus \overline{F}\}$ have an $\epsilon \log n$ majority that is aligned with the ground-truth community label. 

\begin{lemma}
\label{lemma:union_graph_majority}
If $\left(1 - (1 - s)^{2} \right) \dchab > 1 + \epsilon | \log(\alpha / \beta) |$, then with probability $1 - o(1)$, we have for all $i \in [n]$ that $\maj_{G_1 \lor_{\pi_*} G_2} (i) \ge \epsilon \log n$. 
\end{lemma}

\begin{proof}
The proof follows directly from Lemma \ref{lemma:binomial-tail} and a union bound.
\end{proof}

\begin{lemma}
\label{lemma:union_graph_muhat_majority}
Suppose that $(1 - (1 - s)^2) \dchab > 1 + 2\epsilon | \log (\alpha / \beta) |$. Then with probability $1 - o(1)$, all vertices in $[n] \setminus \overline{F}$ have an $\epsilon \log n$ majority in $(G_1 \lor_{\widehat{\mu}} G_2) \{ [n] \setminus \overline{F} \}$. 
\end{lemma}

\begin{proof}
Recall that $F_* : = [n] \setminus M_*$ is the set of vertices outside the 13-core of $G_1 \land_{\pi_*} G_2$, and that $\overline{F}_*$ is the \Luczak expansion of $F_*$ with respect to $G_2$. In light of Lemma~\ref{lemma:k_core_sbm_v2} and Remark~\ref{remark:M_hat_star_swap}, it suffices to replace $\overline{F}$ with $\overline{F}_*$ and $G_1 \lor_{\widehat{\mu}} G_2$ with $( G_1 \lor_{\pi_*} G_2 ) \{ [n] \setminus F_* \}$ in our analysis. For brevity, we denote $H : = (G_1 \lor_{\pi_*} G_2) \{ [n] \setminus \overline{F}_* \}$ in the remainder of this proof.

To begin, notice that we can lower bound the neighborhood majority of $i \in [n] \setminus \overline{F}_*$ in $H$ as 
\begin{align}
\maj_H(i) & = \sigma_*(i)\sum\limits_{j \in \cN_H(i)} \sigma_*(j) \nonumber 
\ge \sigma_*(i) \sum\limits_{j \in \cN_{G_1 \lor_{\pi_*} G_2 } (i) } \sigma_*(j) - | \cN_{G_1 \lor_{\pi_*} G_2}(i) \cap \overline{F}_* | \nonumber \\
\label{eq:majH_lower_bound}
& = \maj_{G_1 \lor_{\pi_*} G_2} (i) - | \cN_{G_1 \setminus_{\pi_*} G_2 }(i) \cap \overline{F}_* | - | \cN_{G_2}(i) \cap \overline{F}_* |.
\end{align}
Through identical arguments, we also have the upper bound 
\begin{equation}
\label{eq:majH_upper_bound}
\maj_H(i) \le \maj_{G_1 \lor_{\pi_*} G_2}(i) + | \cN_{G_1 \setminus_{\pi_*} G_2 }(i) \cap \overline{F}_* | + | \cN_{G_2}(i) \cap \overline{F}_* |.
\end{equation}
Combining \eqref{eq:majH_lower_bound} and \eqref{eq:majH_upper_bound}, we obtain 
\begin{equation}
\label{eq:majH_absolute_value}
\left| \maj_H(i) - \maj_{G_1 \lor_{\pi_*} G_2} (i) \right| \le | \cN_{G_1 \setminus_{\pi_*} G_1} (i) \cap \overline{F}_* | + | \cN_{G_2} (i) \cap \overline{F}_* |. 
\end{equation}
In particular, since $| \maj_{G_1 \lor_{\pi_*} G_2}(i) | \ge 2 \epsilon \log n$ for all $i \in [n]$ with probability $1 - o(1)$ under the condition $(1 - (1 - s)^2) \dchab > 1 + 2 \epsilon | \log (\alpha / \beta) |$ by Lemma \ref{lemma:union_graph_majority}, it is ensured that $\maj_H(i) \ge \epsilon \log n$ provided the right hand side of \eqref{eq:majH_absolute_value} is at most $\epsilon \log n$ for all $i \in [n]$. We show that this is indeed the case by bounding the two terms on the right hand side of \eqref{eq:majH_absolute_value} separately.

We start by bounding $| \cN_{G_1 \setminus_{\pi_*} G_2} (i) \cap \overline{F}_*|$. Conditioned on $\pi_*$, $\boldsymbol{\sigma_*}$, and $\boldsymbol{\cE} := \{ \cE_{00}, \cE_{01}, \cE_{10}, \cE_{11} \}$, 
the graph $G_1 \setminus_{\pi_*} G_2$ 
is independent of $\overline{F}_*$ by Lemma~\ref{lemma:random_partition}, 
since $\overline{F}_*$ depends only on $G_1 \land_{\pi_*} G_2$ and~$G_2$. Thus we can stochastically dominate $| \cN_{G_1 \setminus_{\pi_*} G_2} (i) \cap \overline{F}_* |$ by a Poisson random variable $X$ with mean $\lambda_n$ given~by 
$$
\lambda_n := \gamma \frac{\log n}{n} \left| \{ j \in \overline{F}_*: \{ i,j \} \in \cE_{10} \} \right | \le \gamma \frac{\log n}{n} | \overline{F}_* |, 
$$
where $\gamma : = \max \{\alpha, \beta \}$. Next, for a fixed $\delta > 0$, define the event 
$$
\cG : = \left \{ | \overline{F}_* | \le n^{1 - s^2 \dconnab + \delta}  \right \},
$$
and notice that on $\cG$, $\lambda_n \le n^{ - s^2 \dconnab + \delta + o(1) }$. Thus for any positive integer $m$ we have that 
\begin{align*}
\p \left( \{ | \cN_{G_1 \setminus_{\pi_*} G_2 }(i) \cap \overline{F}_* | \ge m \} \cap \cG \right) 
& \le \p ( \{ X \ge m \} \cap \cG  ) 
= \E \left [ \p \left( X \ge m \, \middle| \, | \overline{F}_* |, \boldsymbol{\cE}, \boldsymbol{\sigma_*}, \pi_* \right) \mathbf{1} ( \cG ) \right ] \\
& \le \E \left [ \left( \inf\limits_{\theta > 0} e^{-\theta m + \lambda_n ( e^\theta - 1) } \right) \mathbf{1}( \cG) \right ] \\
& \le \E [ e \lambda_n^m \mathbf{1} ( \cG ) ]
\le n^{- m ( s^2 \dconnab - \delta - o(1)) }.
\end{align*}
Above, the equality on the first line is due to the tower rule and since $\cG$ is measurable with respect to $\overline{F}_*$; the inequality on the second line is due to a Chernoff bound; the inequality on the third line follows from setting $\theta = \log (1 / \lambda_n)$ (which is valid since $\lambda_n = o(1)$ if $\cG$ holds); the final inequality uses the upper bound for $\lambda_n$ on $\cG$. Taking a union bound over $i \in [n]$ shows that 
\begin{equation*}
\label{eq:neighbors_g1_minus_g2_fbar_1}
\p \left( \{ \exists i \in [n]: | \cN_{G_1 \setminus_{\pi_*} G_2}(i) \cap \overline{F}_* | \ge m \} \cap \cG \right) \le n^{ 1- m ( s^2 \dconnab - \delta - o(1))}.
\end{equation*}
In particular, for $\delta$ taken to be sufficiently small, the right hand side is $o(1)$ if $m > ( s^2 \dconnab )^{-1}$ (it suffices to set $m = \lceil ( s^2 \dconnab )^{-1} \rceil + 1$). Finally, since $\p ( \cG) = 1 - o(1)$ in light of Lemmas \ref{lemma:number-outside-core} and \ref{lemma:luczak}, we obtain 
\begin{equation}
\label{eq:neighbors_g1_minus_g2_fbar_2}
\p \left( \forall i \in [n], | \cN_{G_1 \setminus_{\pi_*} G_2} (i) \cap \overline{F}_* | \le \left \lceil ( s^2 \dconnab )^{-1} \right \rceil  \right) = 1 - o(1).
\end{equation}
Turning to the second term on the right hand side of \eqref{eq:majH_absolute_value}, notice that by the definition of the \Luczak expansion, every $i \in [n] \setminus \overline{F}_*$ can have at most one neighbor in $\overline{F}_*$ with respect to the graph $G_2$, hence $| \cN_{G_2}(i) \cap \overline{F}_* | \le 1$. This observation coupled with \eqref{eq:neighbors_g1_minus_g2_fbar_2} shows that with probability $1 - o(1)$, it holds for all $i \in [n] \setminus \overline{F}_*$ that $| \maj_H(i) - \maj_{G_1 \lor_{\pi_*} G_2} | \le \epsilon \log n$. 
\end{proof}

Next, in Lemma \ref{lemma:neighbors_g2_minus_g1_Iepsilon}, we show that with probability $1 - o(1)$, each vertex in $G_2 \setminus_{\pi_*} G_1$ (and therefore also $G_2 \setminus_{\widehat{\mu}} G_1$) has a small number of neighbors in $I_\epsilon(G_1)$. 

\begin{lemma}
\label{lemma:neighbors_g2_minus_g1_Iepsilon}
If $0 < \epsilon \le \frac{ s \dchab }{4 | \log (\alpha / \beta) |}$, then 
$$
\p \left( \forall i \in [n], | \cN_{G_2 \setminus_{\pi_*} G_1 }(i) \cap I_\epsilon(G_1) | \le 2 \left \lceil (s \dchab )^{-1} \right \rceil  \right) = 1 - o(1).
$$
\end{lemma}

\begin{proof}
Since $I_\epsilon(G_1)$ depends on $G_1$ alone, it follows that $I_\epsilon(G_1)$ and $G_2 \setminus_{\pi_*} G_1$ are conditionally independent given $\pi_{*}$, $\boldsymbol{\sigma_*}$, and $\boldsymbol{\cE}$. Hence we can stochastically dominate $| \cN_{G_2 \setminus_{\pi_*} G_1}(i) \cap I_\epsilon(G_1) |$ by a Poisson random variable $X$ with mean $\lambda_n$ given by 
$$
\lambda_n := \gamma \frac{\log n}{n} \left | \{ j \in I_\epsilon(G_1) :  \{i,j \} \in \cE_{01} \} \right | \le \gamma \frac{\log n}{n} | I_\epsilon(G_1) |.
$$
Next, define the event 
$$
\cG : = \left \{ | I_\epsilon(G_1) | \le n^{1 - s \dchab + 2 \epsilon | \log (\alpha / \beta) |} \right \}.
$$
Notice that $\p ( \cG) = 1 - o(1)$ by Lemma \ref{lemma:number-incorrectly-classified} and Markov's inequality, provided $s \dchab < 99$. We may assume so without loss of generality, since if $s \dchab > 1$ we could simply recover all communities from $G_1$ or $G_2$ alone. Following identical arguments as the proof of Lemma \ref{lemma:union_graph_muhat_majority}, we arrive at 
$$
\p \left( \exists i \in [n] : | \cN_{G_2 \setminus_{\pi_*} G_1}(i) \cap I_\epsilon(G_1) | \ge m \right) = o(1), 
$$
provided $m > ( s \dchab - 2 \epsilon | \log (\alpha / \beta) |)^{-1}$. In particular, if $0 < \epsilon \le \frac{ s \dchab }{4 | \log (\alpha / \beta) |}$ then it suffices to set $m = 2 \lceil ( s \dchab )^{-1} \rceil + 1$. 
\end{proof}

Putting all of our intermediate results together shows that for each vertex in $(G_2 \setminus_{\widehat{\mu}} G_1) \{[n] \setminus \overline{F} \}$, the number of incorrectly-classified neighbors is small, hence Algorithm \ref{alg:labeling-k-core} succeeds with probability $1 - o(1)$. We prove this formally below. 

\begin{proof}[Proof of Lemma \ref{lemma:alg_outside_Fbar}]
In light of Lemma \ref{lemma:k_core_sbm_v2}, it suffices to prove the claim with $G_1 \lor_{\widehat{\mu}} G_2$ replaced by $(G_1 \lor_{\pi_*} G_2) \{ [n] \setminus \overline{F}_* \}$. As a shorthand, we denote the latter graph by $H$. We also set $H_1 : = G_1 \{ [n] \setminus \overline{F}_* \}$ and $H_{2 \setminus 1} : = ( G_2 \setminus_{\pi_* } G_1) \{ [n] \setminus \overline{F}_* \}$. Let $\widehat{\boldsymbol{\sigma}}_1$ be the almost-exact community labels for $G_1$ produced by Algorithm \ref{alg:MNS}, and notice that in light of Lemma \ref{lemma:MNS}, the set of errors are contained in $I_\epsilon(G_1)$ with probability $1 - o(1)$. Moving forward, we may therefore condition on this event. 

We can now compare the neighborhood majorities in $H$ corresponding to $\widehat{\boldsymbol{\sigma}}_1$ with the true majority in $H$ as follows. For any $i \in [n] \setminus \overline{F}_*$, we have that 
\begin{align*}
\left| \sigma_*(i) \sum\limits_{j \in \cN_H(i)} \widehat{{\sigma}}_1(j) - \maj_H(i) \right| 
& = \left | \sum\limits_{j \in \cN_H(i) } \sigma_*(i) \left( \widehat{\sigma}_1(j) - \sigma_*(j) \right) \right |  
\le | \cN_H(i) \cap I_\epsilon(G_1) | \\
& = | \cN_{H_1}(i) \cap I_\epsilon(G_1) | + | \cN_{H_{2 \setminus 1}}(i) \cap I_\epsilon(G_1) | \\
& \le | \cN_{G_1}(i) \cap I_\epsilon(G_1) | + | \cN_{G_2 \setminus_{\pi_*} G_1}(i) \cap I_\epsilon(G_1) | 
\le \frac{\epsilon}{2} \log n.
\end{align*}
Above, the first inequality follows since $\widehat{\sigma}_1(j) \neq \sigma_*(j)$ implies that $j \in I_\epsilon(G_1)$; the next inequality follows since $H_1$ is a subgraph of $G_1$ and $H_{2 \setminus 1}$ is a subgraph of $G_2 \setminus_{\pi_*} G_1$; and the final inequality is due to Lemmas~\ref{lemma:I_internal} and~\ref{lemma:neighbors_g2_minus_g1_Iepsilon}. To conclude the proof, notice that since $\maj_H(i)  \ge \epsilon \log n$ for $i \in [n] \setminus \overline{F}_*$ by Lemma~\ref{lemma:union_graph_muhat_majority}, it follows that the sign of all neighborhood majorities are equal to the ground truth community label for all vertices in $H$, with probability $1 - o(1)$. 
\end{proof}

\subsection{Classifying $\overline{F} \setminus F$}

The remaining set to classify is $\overline{F} \setminus F$, which we can do via a simple majority vote. 

\begin{lemma}
\label{lemma:C2_minus_C1}
Suppose that $\alpha, \beta, \epsilon > 0$ satisfy the following conditions:
\[
(1 - (1 - s)^2) \dchab > 1 + 2 \epsilon | \log(\alpha / \beta) |, 
\qquad \qquad \qquad
0 < \epsilon \le \frac{ s \dchab }{4 | \log(\alpha / \beta)| }, 
\]
\[
s^2 \dconnab + s(1 - s) \dchab > 1.
\] 
Then, with high probability, Step \ref{step:alg2-9} of Algorithm \ref{alg:labeling-k-core} with input $(G_1, G_2, (\widehat{M}, \wh{\mu}), \alpha, \beta, s, \epsilon)$ correctly labels all vertices in $\overline{F} \setminus F$.
\end{lemma}

\begin{proof}
We provide a proof for the case $\alpha > \beta$; the case $\alpha < \beta$ follows from identical arguments. For $i \in [n]$, define the graphs $\widehat{H}_i : = ( G_1 \setminus_{\widehat{\mu}} G_2) \{ \widehat{M} \cup \{i \} \}$ and $H_i : = (G_1 \setminus_{\pi_*} G_2) \{ M_* \cup \{i \} \}$. We also define $E_i$ to be the event where $i$ has a majority of at most $\epsilon' \log n$ in $\widehat{H}_i$ with respect to the community labeling $\widehat{\boldsymbol{\sigma}}$. Our goal is to upper-bound $\mathbb{P}\left(\cup_i (\{i \in \overline{F}\} \cap E_i) \right)$. 

To study this probability, it will be useful to define a ``nice'' event based on our previous results on Algorithm \ref{alg:labeling-k-core}. Let $F_* := [n] \setminus M_*$, and let $\overline{F}_*'$ be the result of applying Algorithm \ref{alg:Luczak} to $(G_2, \pi_*(F_*))$. Let $\overline{F}_* := \pi_*^{-1}(\overline{F_*}')$. That is, $\overline{F}_*$ is the analogue of $\overline{F}$. Finally, let $\widehat{\boldsymbol{\sigma}}$ be the labeling produced by Step \ref{step:classify_bulk} of Algorithm \ref{alg:labeling-k-core}. For a fixed $\delta > 0$, we now define the event $\cH$, which holds if and only if:
\begin{itemize}

\item $\overline{F} = \overline{F}_*$;

\item $G_1 \setminus_{\widehat{\mu}} G_2 = (G_1 \setminus_{\pi_*} G_2 ) \{ M_* \}$; 

\item $\widehat{\sigma}(i) = \sigma_*(i)$ for all $i \in [n] \setminus \overline{F}_*$; 

\item The event $\cF$ holds (see Definition \ref{def:F}); 

\item $| \overline{F}_*| \le n^{ 1 - s^2 \dconnab + \delta }$.

\end{itemize}
By Lemmas \ref{lemma:F}, \ref{lemma:k_core_sbm_v2}, \ref{lemma:luczak}, \ref{lemma:number-outside-core},  and \ref{lemma:alg_outside_Fbar}, we have that $\p( \cH) = 1  - o(1)$. Furthermore, if we define the event $E_{*,i} : = \{ \maj_{H_i}(i) \le \epsilon' \log n \}$, we have that
\begin{align}
\p \left( \bigcup\limits_{i \in [n]} (\{ i \in \overline{F} \} \cap E_i) \right) & \le \p \left( \left( \bigcup\limits_{i \in [n]} (\{ i \in \overline{F} \} \cap E_i) \right) \cap \cH \right) + \p ( \cH^c) \nonumber \\
& \le \p \left( \bigcup\limits_{i \in [n]} (\{ i \in \overline{F}_*\} \cap E_{*,i} \cap \{ | \overline{F}_* | \le n^{1 - s^2 \dconnab + \delta} \} \cap \cF ) \right) + o(1) \nonumber \\
\label{eq:Fbar_prob_decomposition}
& \le \sum\limits_{i = 1}^n \p \left( \{ i \in \overline{F}_* \} \cap E_{*,i} \cap \{ | \overline{F}_* | \le n^{1 - s^2 \dconnab + \delta} \} \cap \cF \right) + o(1).
\end{align}
Using the tower rule, we may rewrite the terms in the summation on the right hand side as 
\begin{equation}
\label{eq:E*_tower_rule}
\E \left[ \p \left( E_{*,i} \, \middle| \, \pi_*, \boldsymbol{\sigma_*}, \boldsymbol{\cE}, \overline{F}_* \right) \mathbf{1}( i \in \overline{F}_*) \mathbf{1} ( \{ | \overline{F}_* | \le n^{1 - s^2 \dconnab + \delta } \} \cap \cF ) \right].
\end{equation}
An important goal is to therefore understand $\p\left( E_{*,i} \, \middle| \, \pi_*, \boldsymbol{\sigma_*}, \boldsymbol{\cE}, \overline{F}_* \right)$. To this end, notice that conditionally on $\pi_*$, $\boldsymbol{\sigma_*}$, and $\boldsymbol{\cE}$, we have the distributional representation $\maj_{H_i}(i) \stackrel{d}{=} Y - Z$, where $Y$ and $Z$ are independent with 
\begin{align*}
Y & \sim \Bin \left ( | \{ j \in M_* : \{i,j \} \in \cE_{10} \cap \cE^+ ( \boldsymbol{\sigma_*^1} ) \}|, \alpha \frac{\log n}{n} \right), \\
Z & \sim \Bin \left ( | \{ j \in M_* : \{i,j \} \in \cE_{10} \cap \cE^- ( \boldsymbol{\sigma_*^1} ) \}|, \beta \frac{\log n}{n} \right).
\end{align*}
Fixing a sufficiently small $\delta > 0$, it holds on the event $\{ | \overline{F}_* | \le n^{1 - s^2 \dconnab + \delta } \} \cap \cF$ (see Definition~\ref{def:F} for a definition of the event $\cF$) that $ | \{ j \in M_* : \{i,j \} \in \cE_{10} \cap \cE^+ ( \boldsymbol{\sigma_*^1} ) \} | = (1 - o(1)) s(1 - s) n / 2$ and $| \{ j \in M_* : \{ i,j \} \in \cE_{10} \cap \cE^-(\boldsymbol{\sigma_*^1}) \} | = (1 - o(1)) s(1 - s) n / 2$. Lemma \ref{lemma:binomial-tail} therefore implies
$$
\p \left( E_{*, i} \, \middle| \, \pi_* , \boldsymbol{\sigma_*} , \boldsymbol{\cE}, \overline{F}_* \right)\mathbf{1} (\{ | \overline{F}_* | \le n^{1 - s^2 \dconnab + \delta} \} \cap \cF) 
\le n^{ - s(1 - s) \dchab + \frac{\epsilon'}{2} \log ( \alpha / \beta) + o(1)}.
$$
Utilizing the expression in \eqref{eq:E*_tower_rule} and taking a union bound over $1 \le i \le n$, we obtain 
\begin{align*}
& \sum\limits_{i = 1}^n  \p ( \{ i \in \overline{F}_* \} \cap E_{*,i} \cap \{ | \overline{F}_* | \le n^{1 - s^2 \dconnab + \delta } \} \cap \cF ) \\
 & \hspace{2cm} \le n^{-s(1 - s) \dchab + \frac{\epsilon'}{2} \log(\alpha / \beta) + o(1)} \sum\limits_{i = 1}^n \E [ \mathbf{1}( i \in \overline{F}_* ) \mathbf{1} ( | \overline{F}_* | \le n^{1 - s^2 \dconnab + \delta} )] \\
& \hspace{2cm} = n^{-s(1 - s) \dchab + \frac{\epsilon'}{2} \log(\alpha / \beta) + o(1)} \E [ | \overline{F}_* | \mathbf{1} ( | \overline{F}_* | \le n^{1 - s^2 \dconnab + \delta} )] \\
& \hspace{2cm} \le n^{ 1 - s(1 - s) \dchab - s^2 \dconnab + \frac{\epsilon'}{2} \log(\alpha / \beta) + \delta + o(1) }.
\end{align*}
Under the condition $s(1 - s) \dchab + s^2 \dconnab > 1$, we can choose $\epsilon'$ and $\delta$ to be small enough so that the bound on the right hand side is $o(1)$. In light of \eqref{eq:Fbar_prob_decomposition}, this proves the lemma. 
\end{proof}

\section{Classifying the rest (Algorithm \ref{alg:main})}\label{sec:classify_rest}
What remains is to classify the set of vertices outside of the $k$-core matching. Again, this can be handled through a simple majority vote.

\begin{lemma}\label{lemma:classifying-F}
Suppose that $\alpha, \beta, \epsilon$ satisfy the same conditions stated in Lemma~\ref{lemma:C2_minus_C1}. Then Step~\ref{step:alg3-3} in Algorithm~\ref{alg:main} on input $(G_1, G_2, 13, \epsilon)$ correctly labels all $i \in F$ with high probability. 
\end{lemma}

\begin{proof}
For $\epsilon' > 0$, let $E_i$ be the event that $i$ does not have an $\epsilon' \log n$ neighborhood majority in the graph $(G_1 \setminus_{\pi_{*}} G_2) \{ ([n] \setminus F) \cup \{i\} \}$. By nearly identical arguments to the proof of Lemma \ref{lemma:C2_minus_C1}, we can show that 
\[
\mathbb{P}\left(\cup_{i \in [n]} (\{i \in F\} \cap E_i) \right) \leq n^{1-s^2 \dconnab -s(1-s) \dchab + O(\epsilon')} + o(1).
\]
For $\epsilon'$ sufficiently small, the right hand side is $o(1)$ by the assumption on the parameters.

Next, due to the maximality of the $13$-core, any $i \in F$ has at most $12$ neighbors in the graph $(G_1 \land_{\pi_{*}} G_2) \{([n] \setminus F) \cup \{i\} \}$. Therefore, any $i \in F$ has an $(\epsilon' \log n - 12)$ majority in $G_1\{([n] \setminus F) \cup \{i\} \}$, with high probability. 
\end{proof}


\section{Proof of Theorem~\ref{thm:comm_recovery_reduction}} \label{sec:thm_proof_altogether} 
Finally, we prove Theorem~\ref{thm:comm_recovery_reduction}, from which Theorem~\ref{thm:comm_recovery} directly follows.
\begin{proof}[Proof of Theorem \ref{thm:comm_recovery_reduction}] 
By Lemma \ref{lemma:k_core_sbm_v2}, the matching $(\widehat{M}, \wh{\mu})$ obtained in Step \ref{step:alg3-1} coincides with the matching $(M_*, \pi_*\{M_*\})$, with high probability. Subsequently, Lemmas~\ref{lemma:alg_outside_Fbar},~\ref{lemma:C2_minus_C1}, and~\ref{lemma:classifying-F} respectively show that the vertices in the sets $[n] \setminus \overline{F}$, $\overline{F} \setminus F$, and $F$ are correctly labeled with high probability.
\end{proof}



\section{Impossibility of Exact Community Recovery}\label{sec:impossibility_proof}

In this section we prove Theorem~\ref{thm:comm_recovery_impossibility}, which states the conditions under which exact community recovery is impossible. Since impossibility under the condition 
$\left(1 - (1 - s)^{2} \right) \dchab < 1$ 
was already proven in~\cite{RS21}, here we prove impossibility in the regime
\begin{equation}
\label{eq:impossibility_2}
s^{2} \dconnab + s(1-s) \dchab < 1.
\end{equation}
To do so, we study the performance of the maximum a posteriori (MAP) estimator for the communities in $G_1$. We will show that, even when supplied with additional information, such as all the correct community labels in~$G_{2}$ and most of the true vertex matching $\pi_*$, the MAP estimator fails to exactly recover communities with probability bounded away from zero if the condition~\eqref{eq:impossibility_2} holds. Since the MAP estimator is \emph{optimal} in the sense that it maximizes the probability of correctness over all estimators (see, e.g., \cite[Chapter 4]{poor_book}), the result of Theorem~\ref{thm:comm_recovery_impossibility} follows.

\subsection{Notation}

We briefly review and introduce some notation. Let $\boldsymbol{\sigma_{*}^{1}}$ be the ground-truth community labels in~$G_1$, so $\sigma_*^1(i)$ is the ground-truth community label of $i \in [n]$. Correspondingly, we define the ground-truth community partition $(V_1^+, V_1^-)$, where
$$
V_1^+ : = \{ i \in [n] : \sigma_*^1(i) = +1 \} \qquad \text{ and } \qquad V_1^- : = \{ i \in [n] : \sigma_*^1(i) = -1 \}.
$$
We define $\boldsymbol{\sigma_{*}^{2}}$, $V_2^+$, and $V_2^-$ to be the analogous quantities for $G_2$. 
Since $\sigma_{*}^{2} (i) = \sigma_{*}^{1} ( \pi_{*}^{-1}(i))$, 
we have that ${\sigma}_*^2(\pi_*(i)) = \sigma_*^1(i)$, 
and also that 
$V_2^+ = \pi_*( V_1^+)$ and $V_2^- = \pi_*( V_1^-)$. 
Finally, for a given $\pi \in \cS_n$ and a set $K \subset [n]$, let $\pi \{K \} : = \{ (i, \pi(i) ) \}_{i \in K}$ denote the restriction of $\pi$ to the set $K$.

\subsection{The MAP estimator}

We begin by defining the {\it singleton set} of a permutation $\pi$ with respect to the adjacency matrices $A$ and $B$ to be 
$$
R(\pi, A, B) := \left \{ i \in [n] : \forall j \in [n], A_{i,j} B_{\pi(i), \pi(j)} = 0 \right \}.
$$
In words, $R(\pi, A, B)$ is the set of singletons in $G_1 \land_{\pi} G_2$, with respect to the labeling in $G_1$. For brevity, we also write $R_\pi : = R(\pi, A, B)$ and $R_* : = R(\pi_*, A, B)$. We will also consider a pruned version of $R_\pi$ that is formally described below. 

\begin{definition}[The set $S(\pi, A, B)$]
\label{def:S}
We have that $i \in S(\pi, A, B)$ if and only if the following conditions hold:
\begin{enumerate}

\item \label{item:S_T}
We have that $i \in R_\pi$.

\item \label{item:S_singleton}
Vertex $i$ is a singleton in $G_1 \{R_\pi \}$ (equivalently, $A_{i,j} = 0$ for all $j \in R_\pi$).

\item \label{item:S_neighbors}
If $j \in \cN_1(i)$, then $\pi(j) \notin \cN_2(\pi(R_\pi))$ (equivalently, if $A_{i,j} = 1$, then $B_{k, \pi(j) } = 0$ for all $k \in \pi(R_\pi)$). 

\end{enumerate}
For brevity, we also write $S_\pi := S(\pi, A, B)$ and $S_* := S(\pi_*, A, B)$. We remark that if we define the set $\overline{R}_\pi : = R_\pi \cup \pi^{-1} ( \cN_2( \pi(R_\pi)) )$, then conditions \#\ref{item:S_singleton} and \#\ref{item:S_neighbors} above can be succinctly stated as $A_{i,j} = 0$ for all $j \in \overline{R}_\pi$. As per our conventions, when $\pi = \pi_*$ we write $\overline{R}_*$.
\end{definition}

At a high level, $S_*$ is a set that contains no overlapping information across $G_1$ and $G_2$ under the assumption that $\pi_* = \pi$. As we shall see through a detailed study of the posterior distribution of $\pi_*$ and $\boldsymbol{\sigma_{*}^{1}}$, information from $G_2$ provides no non-trivial information about the community structure of $S_*$ in $G_1$. 

Moving forward, our goal is to study the MAP estimate equipped with the additional knowledge of $\boldsymbol{\sigma_{*}^{2}}$, $S_*$, and $\pi_*\{ [n] \setminus S_* \}$. It is useful to note that several important quantities can be readily obtained from this revealed information, as we detail next.  
\begin{itemize}
\item For $i \in [n] \setminus S_*$, the community label $\sigma_*^1(i)$ can be deduced via $\sigma_*^1(i) = \sigma_*^2(\pi_*(i))$. 
\item The image of $S_*$ under $\pi_*$ can be found via $\pi_*(S_*) = [n] \setminus \pi_* ( [n] \setminus S_*)$.
\item We may compute $|S_* \cap V_1^+|$ via the following formula: 
\begin{align*}
|S_* \cap V_1^+| 
&= | \{ i \in [n] : \sigma_{*}^{1}(i) = +1 \} | 
- | \{ i \in [n] \setminus S_{*} : \sigma_{*}^{1}(i) = +1 \}| \\
&= | \{ i \in [n] : \sigma_{*}^{2}(i) = +1 \} | 
- | \{ i \in [n] \setminus S_{*} : \sigma_{*}^{2}(\pi_{*}(i)) = +1 \}|, 
\end{align*}
where in the second equality we used that the communities have the same sizes in $G_{1}$ and~$G_{2}$. 
\item We may compute $|S_* \cap V_1^-|$ similarly, or also via $|S_* \cap V_1^-| = |S_{*}| - |S_* \cap V_1^{+}|$. 
\item Finally, define 
\[
\mathsf{maj}(i) := \sum\limits_{j \in \cN_1(i)} \sigma_*^1(j),
\]
and note that for $i \in S_{*}$ the quantity $\mathsf{maj}(i)$ can also be computed using the revealed information. 
This follows from condition~\#\ref{item:S_singleton} of Definition~\ref{def:S}: 
if $i \in S_{*}$, then $\cN_1(i) \subseteq [n] \setminus S_*$, and we have already noted that $\sigma_*^1(j)$ is known when~$j \in [n] \setminus S_*$. 
\end{itemize}

With these observations made, we are now ready to describe the MAP estimator given this additional information. 

\begin{theorem}[MAP estimator]
\label{thm:map}
Let $A$, $B$, $\boldsymbol{\sigma_{*}^{2}}$, $S_*$, and $\pi_*\{ [n] \setminus S_* \}$ be given, 
and let $\wh{\boldsymbol{\sigma}}_{\MAP}$ denote the MAP estimator given this information. 
For $i \in [n] \setminus S_{*}$, we have that 
$\wh{\sigma}_{\MAP} (i) = \sigma_*^2(\pi_*(i))$. 
For the vertices in $S_{*}$, the MAP estimator depends on whether $\alpha$ or $\beta$ is larger, as follows. 
\begin{itemize}
    \item If $\alpha > \beta$, then the MAP estimator assigns the label $+1$ to the vertices corresponding to the largest $|S_* \cap V_1^+|$ values in the collection $\{\mathsf{maj}(i) \}_{i \in S_*}$ (breaking ties arbitrarily), and it assigns the label $-1$ to the remaining vertices in $S_{*}$.  
    \item If $\alpha < \beta$, then it does the opposite. That is, the MAP estimator assigns the label $+1$ to the vertices corresponding to the smallest $|S_* \cap V_1^+|$ values in the collection $\{\mathsf{maj}(i) \}_{i \in S_*}$ (breaking ties arbitrarily), and it assigns the label $-1$ to the remaining vertices in $S_{*}$.
\end{itemize}
\end{theorem}

The structure of the MAP estimator highlights the intuition that $G_2$ contains no relevant information about communities in $S_*$: indeed, the assignment of communities depends only on $\maj$, which is a function of $G_1$ only. In the following corollary, we provide a simple condition for the failure of the MAP estimator.

\begin{corollary}
\label{cor:map}
If $\alpha > \beta$ and there exist $i \in S_* \cap V_1^+, j \in S_* \cap V_1^-$ such that $\maj(i) < \maj(j)$, then the MAP estimator fails (i.e., $\wh{\boldsymbol{\sigma}}_{\MAP} \neq \boldsymbol{\sigma_{*}}$). 
Similarly, if $\alpha < \beta$ and there exist $i \in S_* \cap V_1^+, j \in S_* \cap V_1^-$ such that $\maj(i) > \maj(j)$, then the MAP estimator fails. 
\end{corollary}

\begin{proof}
Suppose that $\alpha > \beta$. If the MAP estimator correctly classifies $i$ as $+1$, it follows from Theorem \ref{thm:map} that the MAP estimator also classifies $j$ as $+1$, which is incorrect.
A similar argument holds for the case $\alpha < \beta$. 
\end{proof}

\subsection{Analysis of the MAP estimator: Proof of Theorem \ref{thm:comm_recovery_impossibility}}

We start by defining a few useful quantities. First, we define the sigma algebra $\cI$, which contains the information we condition on to study the performance of the MAP estimator.

\begin{definition}[The sigma algebra $\cI$]
\label{def:I}
We let $\cI$ be the sigma algebra induced by the random variables $B$, $\pi_*$, $\boldsymbol{\sigma_{*}^{1}}$, $\cE_{00}$, $\cE_{01}$, $\cE_{10}$, $\cE_{11}$, and $R_*$. 
\end{definition}

An important consequence of Definition \ref{def:I} is that $\overline{R}_* = R_* \cup \pi_*^{-1} ( \cN_2 ( \pi_* ( R_*)))$ is $\cI$-measurable. Since $\boldsymbol{\sigma_{*}^{1}}$ is $\cI$-measurable, it follows that $\overline{R}_* \cap V_1^+$ and $\overline{R}_* \cap V_1^-$ are $\cI$-measurable as well.

Next, we describe a useful $\cI$-measurable event which concerns the number of vertices of each community in $R_*$ and $\overline{R}_*$.  

\begin{definition}[The event $\cG_\delta$]
Let $\delta > 0$. We say that the event $\cG_\delta$ holds if and only if 
$$
n^{1 - s^2 \dconn(\alpha, \beta) - \delta} \le | R_* \cap V_1^+|, | R_* \cap V_1^-|, | \overline{R}_* \cap V_1^+|, | \overline{R}_* \cap V_1^- | \le n^{1 - s^2 \dconn(\alpha, \beta) + \delta}.
$$
\end{definition}

Our next result shows that $\cG_\delta$ holds with high probability. Since the proof is straightforward but tedious, we defer it to Section~\ref{subsec:bad_sets}.

\begin{lemma}
\label{lemma:T_size}
For any fixed $\delta > 0$, we have that $\p ( \cG_\delta) = 1 - o(1)$ as $n \to \infty$.
\end{lemma}

We now turn to the proof of our impossibility result. For $i \in R_*$, define the indicator variable
$$
W_i : = \begin{cases}
\mathbf{1} ( i \in S_*, \maj(i) < 0  ) & \text{ if } i \in R_* \cap V_1^+, \\
\mathbf{1}( i \in S_*, \maj(i) > 0 ) & \text{ if } i \in R_* \cap V_1^-.
\end{cases}
$$
Our strategy is to show that $\sum_{i \in R_* \cap V_1^+} W_i > 0$ and $\sum_{i \in R_* \cap V_1^-} W_i > 0$ with high probability. This guarantees the existence of $i \in S_* \cap V_1^+$ and $j \in S_* \cap V_1^-$ such that $\maj(i) < 0 < \maj(j)$, which in turn implies that the MAP estimator fails in light of Corollary \ref{cor:map}. To carry out these ideas formally, we will study $\sum_{i \in R_* \cap V_1^+ } W_i$ and $\sum_{i \in R_* \cap V_1^+ } W_i$ through a second moment method. To this end, the following intermediate results establish useful bounds on the first and second moments, on the event that $\cF \cap \cG_\delta$ holds (see Definition \ref{def:F} for a formal definition of $\cF$). 

\begin{lemma}[First moment estimate]
\label{lemma:W_first_moment}
Fix $\delta > 0$ and denote $\theta : = 1 - s^2 \dconn(\alpha, \beta) - s(1 - s) \dch(\alpha, \beta)$. Then 
$$
\E \left[ \sum\limits_{i \in R_* \cap V_1^+ } W_i \, \middle| \, \cI \right] \mathbf{1}(\cF \cap \cG_\delta)  \ge \left (1 - n^{- s^2 \dconn(\alpha, \beta) + 2 \delta} \right ) n^{\theta - \delta - o(1)} \mathbf{1}( \cF \cap \cG_\delta) 
$$
and
$$
\E \left[ \sum\limits_{i \in R_* \cap V_1^- } W_i \, \middle| \, \cI \right] \mathbf{1}(\cF \cap \cG_\delta)  \ge \left (1 - n^{- s^2 \dconn(\alpha, \beta) + 2 \delta} \right ) n^{\theta - \delta- o(1)} \mathbf{1}( \cF \cap \cG_\delta).
$$
\end{lemma}

\begin{lemma}[Second moment estimate]
\label{lemma:W_second_moment}
Fix $\delta > 0$ and denote $\theta : = 1 - s^2 \dconn(\alpha, \beta) - s(1 - s) \dch(\alpha, \beta)$. If $\theta > 0$, it holds for sufficiently small $\delta$ and all $n$ large enough that
$$
\mathrm{Var} \left( \sum\limits_{i \in R_* \cap V_1^+ } W_i \, \middle| \, \cI \right) \mathbf{1}( \cF \cap \cG_\delta) \le n^{ 2 \theta - 3 \delta} \mathbf{1} (\cF \cap \cG_\delta)
$$
and 
$$
\mathrm{Var} \left( \sum\limits_{i \in R_* \cap V_1^- } W_i \, \middle| \, \cI \right) \mathbf{1}( \cF \cap \cG_\delta) \le n^{ 2 \theta - 3 \delta} \mathbf{1} (\cF \cap \cG_\delta).
$$
\end{lemma}

We defer the proofs of Lemmas \ref{lemma:W_first_moment} and \ref{lemma:W_second_moment} to Section \ref{subsec:moment_estimates}. We now utilize the first and second moment estimates to prove Theorem \ref{thm:comm_recovery_impossibility}.

\begin{proof}[Proof of Theorem \ref{thm:comm_recovery_impossibility}]
We start by proving that $\sum_{i \in R_* \cap V_1^+} W_i > 0$ holds with high probability. To establish a lower bound for this event, we may apply the second moment method (specifically, the Cauchy--Schwarz inequality) to obtain that 
\begin{multline*}
\label{eq:second_moment_method_1}
\p \left( \sum\limits_{i \in R_* \cap V_1^+} W_i > 0 \, \middle| \, \cI \right) 
\ge \frac{ \E \left[ \sum_{i \in R_* \cap V_1^+} W_i \, \middle| \, \cI \right]^2 }{ \E \left[ \left( \sum_{i \in R_* \cap V_1^+} W_i \right)^2 \, \middle| \, \cI \right]} \\
= \frac{\E \left[ \sum_{i \in R_* \cap V_1^+} W_i \, \middle| \, \cI \right]^2}{\E \left[ \sum_{i \in R_* \cap V_1^+} W_i \, \middle| \, \cI \right]^2 + \mathrm{Var} \left( \sum_{i \in R_* \cap V_1^+} W_i \, \middle| \, \cI \right)} 
\ge 1 - \frac{ \mathrm{Var} \left( \sum_{i \in R_* \cap V_1^+} W_i \, \middle| \, \cI \right) }{ \E \left[ \sum_{i \in R_* \cap V_1^+} W_i  \, \middle| \, \cI \right]^2 }.
\end{multline*}
To obtain a lower bound on the \emph{unconditional} probability, we can thus write, for a fixed $\delta > 0$ satisfying $\delta < s^{2} \dconnab / 2$, that 
\begin{align}
\p \left( \sum\limits_{i \in R_* \cap V_1^+} W_i > 0 \right) & \ge \p \left( \left \{ \sum\limits_{i \in R_* \cap V_1^+ } W_i > 0 \right \} \cap \cF \cap \cG_\delta \right) \nonumber \\
& = \E \left[ \p \left( \sum\limits_{i \in R_* \cap V_1^+} W_i > 0  \, \middle| \, \cI \right) \mathbf{1}(\cF \cap \cG_\delta) \right] \nonumber \\
& \ge \E \left[ \left( 1 - \frac{ \mathrm{Var} \left( \sum_{i \in R_* \cap V_1^+} W_i \, \middle| \, \cI \right) }{ \E \left[ \sum_{i \in R_* \cap V_1^+} W_i \, \middle| \, \cI \right]^2 } \right) \mathbf{1}(\cF \cap \cG_\delta) \right] \nonumber \\
& \ge \left( 1 - (1 - o(1)) n^{2 \theta - 3 \delta - (2 \theta - 2 \delta) + o(1)} \right) \p ( \cF \cap \cG_\delta) \nonumber 
= 1 - o(1), \nonumber 
\end{align}
where the equality on the second line follows since the events $\cF$ and $\cG_\delta$ are $\cI$-measurable, the inequality on the fourth line uses Lemmas~\ref{lemma:W_first_moment} and~\ref{lemma:W_second_moment}, and the final equality is due to Lemmas~\ref{lemma:F} and~\ref{lemma:T_size}. 
An identical analysis shows that $\p ( \sum_{i \in R_* \cap V_1^-} W_i > 0) = 1 - o(1)$. 
In light of Corollary~\ref{cor:map}, it follows that the MAP estimator fails with probability $1 - o(1)$ .
\end{proof}

\subsection{Properties of $S_\pi$}
\label{subsec:S_properties}

In this section, we prove some properties of the set $S_\pi$ that will be useful in studying the MAP estimator. To begin, we define the set 
\[
\cA(S_*, \pi_* \{[n] \setminus S_* \}) : = \{ \pi \in \cS_n: S_\pi = S_* \text{ and } \pi \{[n] \setminus S_{\pi} \} = \pi_* \{[n] \setminus S_* \} \}.
\]
For brevity, we sometimes write $\cA_*$ instead. The following result highlights the important property that for $\pi \in \cA_*$, the structure of $G_1 \land_{\pi} G_2$ is invariant.

\begin{lemma}
\label{lemma:A_edge_invariance}
For any $\pi\in \cA_*$, we have that 
$A_{i,j} B_{\pi(i), \pi(j)} = A_{i,j} B_{\pi_*(i), \pi_*(j)}$. 
Moreover, if $i \in S_*$ or $j \in S_*$, then $A_{i,j} B_{\pi(i), \pi(j)} = A_{i,j} B_{\pi_*(i), \pi_*(j)} = 0$.
\end{lemma}

\begin{proof}
If $i,j \in [n] \setminus S_*$, then $\pi(i) = \pi_*(i)$ and $\pi(j) = \pi_*(j)$, hence $A_{i,j} B_{\pi(i), \pi(j)} = A_{i,j} B_{\pi_*(i), \pi_*(j) }$. 
On the other hand, if $i \in S_*$ or $j \in S_*$, then $A_{i,j} B_{\pi(i), \pi(j)} = A_{i,j} B_{\pi_*(i), \pi_*(j)} = 0$ by the definition of $S_\pi = S_*$ (see Definition~\ref{def:S}, in particular condition~\#\ref{item:S_T}).
\end{proof}

The following lemma shows that $\cA_*$ is closed under permutations of $S_*$. For a more formal discussion, it will be useful to introduce the following definition. 

\begin{definition}[The permutation $P_{\pi, \rho}$]
\label{def:P}
Let $\rho$ be a permutation of $S_*$. The permutation $P_{\pi, \rho}$ is given by 
\[
P_{\pi, \rho}(i) := 
\begin{cases}
\pi(i) &\text{ if } i \in [n] \setminus S_{*}, \\
\pi(\rho(i)) &\text{ if } i \in S_{*}.
\end{cases}
\]
\end{definition}

\begin{lemma}
\label{lemma:A_permutation}
Let $\rho$ be a permutation of $S_*$. Then $P_{\pi_*, \rho} \in \cA_*$. 
\end{lemma}

\begin{proof}
We abbreviate $P = P_{\pi_*, \rho}$. It is clear from the construction of $P$ that $P \{[n] \setminus S_* \} = \pi_* \{[n] \setminus S_*\}$. To prove $S_P = S_*$, it is sufficient to show that $R_P = R_*$. Indeed, once this has been established, it follows readily that $\overline{R}_P = \overline{R}_*$, and thus $S_P = S_*$ by Definition \ref{def:S}.

We start by proving $R_* \subseteq R_P$; we do so by proving separately that $R_* \setminus S_* \subseteq R_P$ and $ S_* \subseteq R_P$. 

To prove the first claim, let $i \in R_* \setminus S_*$ and notice that $P(i) = \pi_*(i)$. Fix $j \in S_*$. Then, by condition~\#\ref{item:S_singleton} of Definition~\ref{def:S}, $A_{i,j} = 0$, so $A_{i,j} B_{P(i), P(j)} = 0$. On the other hand, if $j \in [n] \setminus S_*$, then $P(j) = \pi_*(j)$, so $A_{i,j} B_{P(i), P(j)} = A_{i,j} B_{\pi_*(i), \pi_*(j)} = 0$. Since we have shown that $A_{i,j} B_{P(i), P(j)} = 0$ for all $i \in R_* \setminus S_*$ and for all $j \in [n]$, it follows that $R_* \setminus S_* \subseteq R_P$ as desired. 

We now prove that $S_* \subseteq R_P$. Let $i \in S_*$. If $j \in R_*$, then by condition~\#\ref{item:S_singleton} of Definition~\ref{def:S} we have $A_{i,j} = 0$, hence $A_{i,j} B_{P(i), P(j)} = 0$. On the other hand if $j \in [n] \setminus R_*$ then by condition~\#\ref{item:S_neighbors} of Definition~\ref{def:S}, $A_{i,j} = 1$ implies $B_{P(i), P(j)} = 0$ since $\pi_*(j) = P(j)$ and $P(i) \in \pi_*( R_*)$; hence $A_{i,j} B_{P(i), P(j)} = 0$. Since we have shown that $A_{i,j} B_{P(i), P(j)} = 0$ for all $i \in S_*$ and $j \in [n]$, it follows that $S_* \subseteq R_P$. Putting both results together shows that $R_* \subseteq R_P$ as desired. 

Next, we prove that $R_P \subseteq R_*$. Suppose by way of contradiction that there exists $i \in R_P \setminus R_*$. Then there must exist $j \in [n]$ such that $A_{i,j} B_{P(i), P(j)} = 0$ and $A_{i,j} B_{\pi_*(i), \pi_*(j)} = 1$, which in turn implies that $(A_{i,j}, B_{P(i), P(j)}, B_{\pi_*(i), \pi_*(j)}) = (1,0,1)$. 
We consider two possibilities for $j$. 
If $j \in R_*$, then $A_{i,j} = 1$ implies $B_{\pi_*(i), \pi_*(j)} = 0$ by condition~\#\ref{item:S_neighbors} of Definition~\ref{def:S}, which is a contradiction. 
On the other hand, if $j \in [n] \setminus R_* \subseteq [n] \setminus S_*$, it follows that $P(j) = \pi_*(j)$. Since $i \in [n] \setminus R_* \subseteq [n] \setminus S_*$ as well, $P(i) = \pi_*(i)$ as well. However, this contradicts $(B_{P(i), P(j)}, B_{\pi_*(i), \pi_*(j)} ) = (0,1)$. 
Since all cases for $j$ lead to a contradiction, we have that $R_P \subseteq R_*$. 
\end{proof}

A useful consequence of Lemma \ref{lemma:A_permutation} is that the elements of $\cA_*$ can be parametrized by permutations of $S$. This is captured in the following corollary.

\begin{corollary}
\label{cor:A_permutation}
We have the representation
\begin{equation}
\label{eq:A_set_equality}
\cA_* = \left \{ P_{\pi_*, \rho} : \rho \text{ is a permutation of $S_*$} \right \}.
\end{equation}
\end{corollary}

\begin{proof}
Lemma \ref{lemma:A_permutation} shows that the right hand side of \eqref{eq:A_set_equality} is a subset of the left hand side of \eqref{eq:A_set_equality}. Moreover, from the definition of the set $\cA_*$, if $\pi \in \cA_*$, then $\pi$ and $\pi_*$ can only disagree on inputs from $S_*$. This implies that we can find a permutation $\rho$ on $S_*$ such that $\pi = P_{\pi_*, \rho}$. Hence the set on the left hand side of \eqref{eq:A_set_equality} is a subset of the right hand side of \eqref{eq:A_set_equality}. 
\end{proof}

\subsection{Deriving the MAP estimator: Proof of Theorem \ref{thm:map}}

\subsubsection{The posterior distribution of $\pi_*$}

We start by defining some notation. For a given permutation $\pi \in \cS_n$, define
\begin{align*}
\mu^+(\pi)_{ab} & := \sum\limits_{\{\pi(i), \pi(j) \} \in \cE^+(\boldsymbol{\sigma_{*}^{2}})} \mathbf{1}((A_{i, j},  B_{\pi(i),\pi(j)}) = (a,b)), \qquad \text{ for } a,b \in \{0,1 \}, \\
\mu^-(\pi)_{ab} & := \sum\limits_{\{\pi(i), \pi(j) \} \in \cE^-(\boldsymbol{\sigma_{*}^{2}})} \mathbf{1}((A_{i, j} ,B_{\pi(i),\pi(j)}) = (a,b)), \qquad \text{ for } a,b, \in \{0,1\}, \\
\nu^+(\pi) & := \sum\limits_{\{\pi(i), \pi(j) \} \in \cE^+(\boldsymbol{\sigma_{*}^{2}})} A_{i, j}, \\
\nu^-(\pi) & := \sum\limits_{\{\pi(i), \pi(j) \} \in \cE^-(\boldsymbol{\sigma_{*}^{2}})} A_{i, j}.
\end{align*}
In words, $\mu^+(\pi)_{ab}$ and $\mu^-(\pi)_{ab}$ capture the empirical joint distribution of correlated edges assuming $\pi_* = \pi$, and $\nu^+(\pi)$ and $\nu^-(\pi)$ count the number of intra-community and inter-community edges in $G_1$, respectively. Using these quantities, we can derive an exact expression for the posterior distribution of $\pi_*$ given $A$, $B$, and $\boldsymbol{\sigma_{*}^{2}}$. The proof is nearly identical to \cite[Lemma 3.1]{RS21}, but we include it here for completeness since there are a few small changes (e.g., we condition on the community labels in $G_2$ rather than in $G_1$) and the proof is short. 

\begin{lemma}
\label{lemma:posterior_v1}
Let $\pi \in \cS_n$. There is a constant $C_1 = C_1 ( A, B, \boldsymbol{\sigma_{*}^{2}})$ such that 
\[
\p \left( \pi_* = \pi \, \middle| \, A, B, \boldsymbol{\sigma_{*}^{2}} \right) = C_1 \left( \frac{p_{00} p_{11}}{ p_{01} p_{10}} \right)^{\mu^+(\pi)_{11}}\left( \frac{q_{00} q_{11}}{q_{01} q_{10}} \right)^{\mu^-(\pi)_{11}}  \left(\frac{p_{10}}{p_{00}} \right)^{\nu^+(\pi)}  \left( \frac{q_{10}}{q_{00}} \right)^{\nu^-(\pi)}.
\]
\end{lemma}

\begin{proof}
By Bayes' rule, we have that 
\[
\p \left(\pi_* = \pi \, \middle| \, A, B, \boldsymbol{\sigma_{*}^{2}} \right) = \frac{ \p \left( A, B \, \middle| \, \pi_* = \pi, \boldsymbol{\sigma_{*}^{2}} \right) \p \left(\pi_* = \pi \, \middle| \, \boldsymbol{\sigma_{*}^{2}} \right) }{ \p \left( A, B, \boldsymbol{\sigma_{*}^{2}} \right)}.
\]
Recall that in the construction of the correlated pair of SBMs $(G_1, G_2)$, the permutation $\pi_*$ is chosen independently of everything else, including the community labeling $\boldsymbol{\sigma_{*}^{2}}$. Hence we can write
\begin{equation}
\label{eq:posterior_bayes}
\p\left( \pi_* = \pi \, \middle| \, A, B, \boldsymbol{\sigma_{*}^{2}} \right) = c_1 \left(A, B, \boldsymbol{\sigma_{*}^{2}} \right) \cdot \p \left(A, B \, \middle| \, \pi_* = \pi, \boldsymbol{\sigma_{*}^{2}} \right),
\end{equation}
where $c_1 \left(A,B, \boldsymbol{\sigma_{*}^{2}} \right) := \left( n! \p \left(A, B, \boldsymbol{\sigma_{*}^{2}} \right) \right)^{-1}$. 
We now analyze $\p \left( A, B \, \middle| \, \pi_* = \pi, \boldsymbol{\sigma_{*}^{2}} \right)$. 
Recall that, given $\boldsymbol{\sigma_{*}^{2}}$ and $\pi_*$, the edge formation process in $G_1$ and $G_2$ is mutually independent across all vertex pairs. Hence we have that
\begin{equation}
\label{eq:p_q_prod}
\p \left(A, B \, \middle| \, \pi_* = \pi, \boldsymbol{\sigma_{*}^{2}} \right) = \left( p_{00}^{\mu^+(\pi)_{00}} p_{01}^{\mu^+(\pi)_{01}} p_{10}^{\mu^+(\pi)_{10}} p_{11}^{\mu^+(\pi)_{11}} \right) \left( q_{00}^{\mu^-(\pi)_{00}} q_{01}^{\mu^-(\pi)_{01}} q_{10}^{\mu^-(\pi)_{10}} q_{11}^{\mu^-(\pi)_{11}} \right).
\end{equation}
To simplify \eqref{eq:p_q_prod}, we can write
\begin{align*}
\mu^+(\pi)_{01} & = \sum\limits_{\{\pi(i), \pi(j) \} \in \cE^+(\boldsymbol{\sigma_{*}^{2}})} (1 - A_{i,j} ) B_{\pi(i), \pi(j)}  = \sum\limits_{\{i,j\} \in \cE^+(\boldsymbol{\sigma_{*}^{2}})} B_{i,j} - \mu^+(\pi)_{11}; \\
\mu^+(\pi)_{10} & = \sum\limits_{\{\pi(i), \pi(j)\} \in \cE^+(\boldsymbol{\sigma_{*}^{2}})} A_{i, j} (1 - B_{\pi(i), \pi(j)})  = \nu^+(\pi) - \mu^+(\pi)_{11}; \\
\mu^+(\pi)_{00} & = \sum\limits_{\{\pi(i), \pi(j) \} \in \cE^+(\boldsymbol{\sigma_{*}^{2}}) } ( 1 - A_{i,j} ) (1 - B_{\pi(i), \pi(j)}) \\ 
&= | \cE^+(\boldsymbol{\sigma_{*}^{2}}) | - \sum\limits_{(i,j) \in \cE^+(\boldsymbol{\sigma_{*}^{2}}) } B_{i,j} - \nu^+(\pi) + \mu^+(\pi)_{11}.
\end{align*}
In particular, note that the quantities $\sum_{\{i,j \} \in \cE^+(\boldsymbol{\sigma_{*}^{2}})} B_{i,j}$ and $|\cE^+(\boldsymbol{\sigma_{*}^{2}})|$ are measurable with respect to $B$ and $\boldsymbol{\sigma_{*}^{2}}$. Hence we can write
\begin{equation}
\label{eq:p_simplified}
 p_{00}^{\mu^+(\pi)_{00}} p_{01}^{\mu^+(\pi)_{01}} p_{10}^{\mu^+(\pi)_{10}} p_{11}^{\mu^+(\pi)_{11}} = c_2^+ \left( \frac{p_{00} p_{11}}{ p_{01} p_{10}} \right)^{\mu^+(\pi)_{11}} \left(\frac{p_{10}}{p_{00}} \right)^{\nu^+(\pi)},
\end{equation}
where 
$$
c_2^+ := p_{00}^{ | \cE^+(\boldsymbol{\sigma_{*}^{2}})|} \left( \frac{p_{01}}{p_{00}} \right)^{\sum_{\{i,j\} \in \cE^+(\boldsymbol{\sigma_{*}^{2}}) } B_{i,j} }.
$$
Replicating these arguments for $\mu^-(\pi)_{ab}$, we have that 
\begin{equation}
\label{eq:q_simplified}
q_{00}^{\mu^-(\pi)_{00}} q_{01}^{\mu^-(\pi)_{01}} q_{10}^{\mu^-(\pi)_{10}} q_{11}^{\mu^-(\pi)_{11}} = c_2^- \left( \frac{q_{00} q_{11}}{q_{01} q_{10}} \right)^{\mu^-(\pi)_{11}} \left( \frac{q_{10}}{q_{00}} \right)^{\nu^-(\pi)},
\end{equation}
where 
$$
c_2^- : = q_{00}^{ | \cE^-(\boldsymbol{\sigma_{*}^{2}})|} \left( \frac{q_{01}}{q_{00}} \right)^{\sum_{\{i,j\} \in \cE^-(\boldsymbol{\sigma_{*}^{2}}) } B_{i,j} }.
$$
Combining \eqref{eq:posterior_bayes}, \eqref{eq:p_q_prod}, \eqref{eq:p_simplified}, and \eqref{eq:q_simplified} proves the statement of the lemma, with $C_1 := c_1 c_2^+ c_2^-$. 
\end{proof}

Next, we build on Lemma \ref{lemma:posterior_v1} to obtain the posterior distribution of $\pi_*$ 
given not only $A$, $B$, and $\boldsymbol{\sigma_{*}^{2}}$, 
but also $S_*$ and $\pi_*\{[n] \setminus S_* \}$. To show this formally, we recall the definition of the set $\cA_*$ from Section~\ref{subsec:S_properties}.

\begin{lemma}[Posterior distribution]
\label{lemma:posterior_A}
There is a constant $C_2 = C_2(A, B, \boldsymbol{\sigma_{*}^{2}}, S_*, \pi_*\{[n] \setminus S_*\})$ such that
$$
\p \left( \pi_* = \pi \, \middle| \, A, B, \boldsymbol{\sigma_{*}^{2}}, S_*, \pi_* \{ [n] \setminus S_* \} \right) = C_2 \left( \sqrt{\frac{ p_{10} q_{00}}{p_{00} q_{10}}} \right)^{\nu^+(\pi) - \nu^-(\pi)} \mathbf{1}(\pi \in \cA_* ).
$$
\end{lemma}

\begin{proof}
By Bayes' rule we have that 
\begin{align*}
\p \left( \pi_* = \pi \, \middle| \, A, B, \boldsymbol{\sigma_{*}^{2}}, S_*, \pi_* \{[n] \setminus S_* \} \right) 
&= 
\frac{\p \left( \pi_* = \pi \, \middle| \, A, B, \boldsymbol{\sigma_{*}^{2}} \right) \p \left( S_*, \pi_* \{[n] \setminus S_* \} \, \middle| \, \pi_{*} = \pi, A, B, \boldsymbol{\sigma_{*}^{2}} \right)}{\p \left( S_*, \pi_* \{[n] \setminus S_* \} \, \middle| \, A, B, \boldsymbol{\sigma_{*}^{2}} \right)} \\
&= 
\frac{\p \left( \pi_* = \pi \, \middle| \, A, B, \boldsymbol{\sigma_{*}^{2}} \right)}{\p \left( S_*, \pi_* \{[n] \setminus S_* \} \, \middle| \, A, B, \boldsymbol{\sigma_{*}^{2}} \right)} \mathbf{1}( \pi \in \cA_* ).
\end{align*}
The 
probability in the denominator is a function of $A$, $B$, $\boldsymbol{\sigma_{*}^{2}}$, $S_*$, and $\pi_* \{[n] \setminus S_* \}$. Furthermore, by Lemma~\ref{lemma:A_edge_invariance}, $\mu^+(\pi)_{11}$ and $\mu^-(\pi)_{11}$ are constant over $\pi \in \cA_*$. Hence by Lemma~\ref{lemma:posterior_v1} we can write
\begin{equation}
\label{eq:posterior_v2}
\p \left( \pi_* = \pi \, \middle| \, A, B, \boldsymbol{\sigma_{*}^{2}}, S_*, \pi_* \{[n] \setminus S_* \} \right) = c_1 \left( \frac{p_{10}}{p_{00}} \right)^{\nu^+(\pi)} \left( \frac{q_{10}}{q_{00}} \right)^{\nu^-(\pi)} \mathbf{1}(\pi \in \cA_*),
\end{equation}
where 
\[
c_1 = \frac{C_1}{\p \left( S_*, \pi_* \{[n] \setminus S_* \} \, \middle| \, A, B, \boldsymbol{\sigma_{*}^{2}} \right)} \left( \frac{p_{00} p_{11}}{ p_{01} p_{10} } \right)^{\mu^+(\pi_*)_{11}} \left( \frac{q_{00} q_{11}}{ q_{01} q_{10}} \right)^{\mu^-(\pi_*)_{11}}.
\]
Above, $C_1$ is the same constant as in Lemma~\ref{lemma:posterior_v1}. In particular, $c_1$ depends only on $A$, $B$, $\boldsymbol{\sigma_{*}^{2}}$, $S_*$, and $\pi_* \{[n] \setminus S_* \}$. To simplify the right hand side of~\eqref{eq:posterior_v2} further, we can first write
\begin{equation}
\label{eq:posterior_v3}
\left( \frac{p_{10}}{p_{00}} \right)^{\nu^+(\pi)} \left( \frac{q_{10}}{q_{00}} \right)^{\nu^-(\pi)} = \left( \sqrt{\frac{p_{10} q_{10}}{p_{00} q_{00}} } \right)^{\nu^+(\pi) + \nu^-(\pi)} \left( \sqrt{ \frac{p_{10} q_{00}}{ p_{00} q_{10}} } \right)^{\nu^+(\pi) - \nu^-(\pi)}.
\end{equation}
Thus a simplification can be obtained by noting that $\nu^+(\pi) + \nu^-(\pi)$ depends only on $A$. Indeed, 
\begin{align*}
\nu^+(\pi) + \nu^-(\pi) & = \sum\limits_{\{\pi(i), \pi(j) \} \in \cE^+(\boldsymbol{\sigma_{*}^{2}})} A_{i,j} + \sum\limits_{\{\pi(i), \pi(j) \} \in \cE^-(\boldsymbol{\sigma_{*}^{2}})} A_{i,j}  = \sum\limits_{\{i,j \} \in \binom{[n]}{2}} A_{i,j}.
\end{align*}
The desired result now follows with 
\[
C_2 : = c_1 \left( \sqrt{ \frac{p_{10} q_{10}}{ p_{00} q_{00}} } \right)^{\sum_{\{i,j \} \in \binom{[n]}{2}} A_{i,j} }. \qedhere
\]
\end{proof}

\subsubsection{The posterior distribution of the community partition}

Our next few results will allow us to translate our characterization of the posterior distribution of $\pi_*$ to the posterior distribution of the community labeling. We proceed by defining some further notation. For a community partition $\mathbf{X} = (X^+, X^-)$ of $[n]$ in $G_1$, define the set 
$$
\cB(\mathbf{X}) : = \left \{ \pi \in \cA_*: \pi(X^+) = V_2^+, \pi(X^-) = V_2^- \right \}.
$$
In words, $\cB(\mathbf{X})$ is the set of permutations in $\cA_*$ which induce the community partition $\mathbf{X}$ in~$G_1$. 
Notice that if $\cB( \mathbf{X}) \neq \emptyset$, 
then the partition $\mathbf{X}$ must be compatible with $A$, $B$, $\boldsymbol{\sigma_{*}^{2}}$, $S_*$, and $\pi_* \{[n] \setminus S_* \}$. In particular, if $\boldsymbol{\sigma}_{\mathbf{X}}$ denotes the community memberships associated with $\mathbf{X}$, 
the following must~hold:
\begin{itemize}
    \item $\sigma_{\mathbf{X}}(i) = \sigma_*^2(\pi_*(i)) = \sigma_*^1(i)$ for $i \in [n] \setminus S_*$;
    \item $|S_* \cap X^+ | = |S_* \cap V_1^+|$ and $|S_* \cap X^-| = |S_* \cap V_1^-|$. 
\end{itemize}
The first condition must hold since we know the true vertex correspondence---and therefore the true community labels---outside of the set $S_*$. The second condition must hold since the number of vertices of each community in $S_*$ can be deduced by examining the community labels of $\pi_*(S_*)$ with respect to $\boldsymbol{\sigma_{*}^{2}}$. 

We proceed by establishing a few more useful results related to $\cB( \mathbf{X})$. The following lemma shows that the size of $\cB(\mathbf{X})$ does not depend on the specific choice of $\mathbf{X}$. 
\begin{lemma}
\label{lemma:B_size}
If $\cB( \mathbf{X})$ is nonempty, then $|\cB( \mathbf{X}) | = |S_* \cap V_1^+|! | S_* \cap V_1^-|!$.
\end{lemma}

\begin{proof}
Suppose that $\pi_0, \pi_1 \in \cB(\mathbf{X})$. By Corollary \ref{cor:A_permutation}, we can write $\pi_1 = P_{\pi_0, \rho}$ for some permutation $\rho$ on $S_*$. Notice that 
if $i \in S_* \cap X^+$, 
then $\rho(i) \in S_* \cap X^+$,  
and  
if $i \in S_* \cap X^-$, 
then $\rho(i) \in S_* \cap X^-$; otherwise, $\pi_0(i)$ and $\pi_1(i) = \pi_0(\rho(i))$ would have different community labels with respect to $\boldsymbol{\sigma_{*}^{2}}$, which would violate our assumption that $\pi_0, \pi_1 \in \cB(\mathbf{X})$. We can therefore decompose $\rho$ into two disjoint permutations $\rho^+$ and $\rho^-$, where $\rho^+$ is a permutation of $S_* \cap X^+$ and $\rho^-$ is a permutation of $S_* \cap X^-$. Since there are $|S_* \cap X^+|! = |S_* \cap V_1^+ |!$ choices for $\rho^+$ and $|S_* \cap X^-|! = |S_* \cap V_1^-|!$ choices for $\rho^-$, the desired result follows. 
\end{proof}

Our next result shows that $\nu^+(\pi) - \nu^-(\pi)$ is invariant over elements of $\cB(\mathbf{X})$, where recall that 
$$
\nu^+(\pi)  := \sum\limits_{\{\pi(i), \pi(j) \} \in \cE^+(\boldsymbol{\sigma_{*}^{2}})} A_{i, j} \qquad \text{and} \qquad \nu^-(\pi)  := \sum\limits_{\{\pi(i), \pi(j) \} \in \cE^-(\boldsymbol{\sigma_{*}^{2}})} A_{i, j}.
$$
We briefly recall some relevant notation. Let $\boldsymbol{\sigma}_{\mathbf{X}}$ be the community labels induced by $\mathbf{X}$, so that $\sigma_{\mathbf{X}}(i) = +1$ if $i \in X^+$ and $\sigma_X(i) = -1$ if $i \in X^-$. We also recall that 
$$
\maj(i) = \sum\limits_{j \in \cN_1(i)} \sigma_*^1(j) = \sum\limits_{j \in [n]} A_{i,j} \sigma_*^1(j).
$$
While $\maj(i)$ is generally not measurable with respect to $\left\{ A, B, \boldsymbol{\sigma_{*}^{2}}, S_*, \pi_* \{ [n] \setminus S_* \} \right\}$, it \emph{is} measurable if $i \in S_*$. The reason is that for $j \in [n] \setminus S_*$, $\pi_*(j)$ is known so we can deduce $\sigma_*^1(j) = \sigma_*^2(\pi_*(j))$. In addition, if $i \in S_*$, then all neighbors of $i$ in $G_1$ are in $[n] \setminus S_*$, so $\maj(i)$ is indeed measurable.

\begin{lemma}
\label{lemma:nu_invariance}
For all $\pi \in \cB(\mathbf{X})$, we have that 
$$
\nu^+(\pi) - \nu^-(\pi) = C_3 + \sum\limits_{i \in S_*} \sigma_{\mathbf{X}}(i) \maj(i),
$$
where $C_3$ depends on $A$, $B$, $\boldsymbol{\sigma_{*}^{2}}$, $S_*$, and  $\pi_* \{ [n] \setminus S_*\}$, but not on $\mathbf{X}$. 
\end{lemma}

\begin{proof}
Noting that 
$\sigma_{\mathbf{X}}(i) \sigma_{\mathbf{X}}(j) = 1$ 
for $\{\pi(i), \pi(j) \} \in \cE^+( \boldsymbol{\sigma_{*}^{2}})$ 
and that 
$\sigma_{\mathbf{X}}(i) \sigma_{\mathbf{X}}(j) = -1$ 
for $\{\pi(i), \pi(j) \} \in \cE^-(\boldsymbol{\sigma_{*}^{2}})$, 
we can write
\[
\nu^+(\pi) - \nu^-(\pi) = \sum\limits_{\{i,j \} \in \binom{[n]}{2}} A_{i,j} \sigma_{\mathbf{X}}(i) \sigma_{\mathbf{X}}(j).
\]
Since $\pi \in \cA_*$, 
we have $\pi(i) = \pi_*(i)$ for $i \in [n] \setminus S_*$, 
and hence $\sigma_{\mathbf{X}}(i) = \sigma_*^2(\pi_*(i)) = \sigma_*^1(i)$. 
Defining 
\[
C_3 : = \sum\limits_{\substack{\{i,j \} \in \binom{[n]}{2}: \\
i,j \in [n] \setminus S_* }} A_{i,j} \sigma_{\mathbf{X}}(i) \sigma_{\mathbf{X}}(j) = \sum\limits_{\substack{\{i,j \} \in \binom{[n]}{2}: \\
i,j \in [n] \setminus S_* }} A_{i,j} \sigma_*^1(i) \sigma_*^1(j),
\] 
it is clear that $C_3$ depends only on 
$A$, $\boldsymbol{\sigma_{*}^{2}}$, $S_*$, and $\pi_* \{[n] \setminus S_* \}$, and not on $\mathbf{X}$. 
Noting further that $A_{i,j} = 0$ when $i, j \in S_*$ by the construction of $S_*$, we have that 
\begin{align*}
\nu^+(\pi) - \nu^-(\pi) & = C_3 + \sum\limits_{i \in S_*, j \in [n] \setminus S_*} A_{i,j} \sigma_{\mathbf{X}}(i) \sigma_{\mathbf{X}}(j) 
= C_3 + \sum\limits_{i \in S_*} \sigma_{\mathbf{X}}(i) \sum\limits_{j \in [n] \setminus S_*} A_{i,j} \sigma_*^1(j) \\
& = C_3 + \sum\limits_{i \in S_*} \sigma_{\mathbf{X}}(i) \maj(i). \qedhere
\end{align*}
\end{proof}

We can now put everything together to derive the posterior probability of a given community partition. 

\begin{lemma}
\label{lemma:posterior_communities}
If $\cB(\mathbf{X})$ is nonempty, then
\[
\p \left( (V_1^+, V_1^-) = (X^+, X^-) \, \middle| \, A, B, \boldsymbol{\sigma_{*}^{2}}, S_*, \pi_* \{ [n] \setminus S_* \} \right) 
= C_4 \left( \frac{p_{10} q_{00}}{ p_{00} q_{10}} \right)^{\frac{1}{2} \sum_{i \in S_*} \sigma_{\mathbf{X}}(i) \maj(i)},
\]
where $C_4$ is a constant depending on $A$, $B$, $\boldsymbol{\sigma_{*}^{2}}$, $S_{*}$, and $\pi_* \{ [n] \setminus S_* \}$, but \emph{not} on the partition $\mathbf{X}$.
\end{lemma}

\begin{proof} 
To compute the posterior probability of a community partition, 
we may equivalently compute the posterior probability of the set of permutations that generate the community partition under consideration. 
We thus have that 
\begin{multline*}
\p \left( (V_1^+, V_1^-) = (X^+, X^-) \, \middle| \, A, B, \boldsymbol{\sigma_{*}^{2}}, S_*, \pi_* \{ [n] \setminus S_* \} \right) \\
= \sum_{\pi \in \cB( \mathbf{X} )} \p \left( \pi_{*} = \pi \, \middle| \,  A, B, \boldsymbol{\sigma_{*}^{2}}, S_*, \pi_* \{ [n] \setminus S_* \} \right) 
= \sum_{\pi \in \cB( \mathbf{X} )} C_{2} \left( \frac{p_{10} q_{00}}{p_{00}q_{10}} \right)^{\left( \nu^{+}(\pi) - \nu^{-}(\pi) \right)/2}, 
\end{multline*}
where the last equality follows by Lemma~\ref{lemma:posterior_A} and $C_{2}$ is the constant appearing in Lemma~\ref{lemma:posterior_A} (and note that $\pi \in \cA_{*}$ for every $\pi \in \cB(\mathbf{X})$). 
Now by Lemma~\ref{lemma:nu_invariance}, this sum is equal to 
\[
C_{2} \left( \frac{p_{10} q_{00}}{p_{00}q_{10}} \right)^{C_{3}/2} \sum_{\pi \in \cB( \mathbf{X} )} \left( \frac{p_{10} q_{00}}{p_{00}q_{10}} \right)^{\frac{1}{2} \sum_{i \in S_*} \sigma_{\mathbf{X}}(i) \maj(i)}. 
\]
Since the summand in the display above does not depend on $\pi$, 
and 
$|\cB( \mathbf{X}) | = |S_* \cap V_1^+|! | S_* \cap V_1^-|!$ 
by Lemma~\ref{lemma:B_size}, 
the desired result follows with 
\[
C_4 : = C_{2} \left( \frac{p_{10} q_{00}}{p_{00} q_{10}} \right)^{C_3 / 2} |S_* \cap V_1^+ |! | S_* \cap V_1^-|!  . \qedhere
\]
\end{proof}

The characterization of the MAP estimator follows as a corollary. 

\begin{proof}[Proof of Theorem \ref{thm:map}]
Suppose that $A$, $B$, $\boldsymbol{\sigma_{*}^{2}}$, $S_*$, and $\pi_* \{ [n] \setminus S_* \}$ are given. 
First of all, we have that 
$\wh{\sigma}_{\MAP}(i) = \sigma_{*}^{2}(\pi_{*}(i))$ 
for all $i \in [n] \setminus S_{*}$ 
(since for any community labeling not satisfying this, the posterior probability is $0$). 
For vertices in $S_{*}$, first 
note that 
$$
\frac{p_{10} q_{00} }{ p_{00} q_{10}} = \frac{\alpha}{\beta}(1 + o(1)).
$$
Thus, by Lemma~\ref{lemma:posterior_communities}, 
the MAP estimator for the community partition of $G_1$ maximizes (resp., minimizes) $\sum_{i \in S_*} \sigma_{\mathbf{X}}(i) \maj(i)$ if $\alpha > \beta$ (resp., $\alpha < \beta$), 
while respecting the constraint that 
$\left|S_* \cap X^+ \right| = |S_* \cap V_1^+|$ 
and $|S_* \cap X^-| = |S_* \cap V_1^-|$. 
The maximum is obtained by setting $\sigma_{\mathbf{X}}(i) = +1$ for $i \in S_{*}$ corresponding to the $|S_* \cap V_1^+|$ largest values of $\{\maj(i)\}_{i \in S_{*}}$ (breaking ties arbitrarily), and setting $\sigma_{\mathbf{X}}(i) = -1$ for the rest. The minimum is obtained by setting $\sigma_{\mathbf{X}}(i)$ in an opposite manner. 
\end{proof}

\subsection{Bounding the size of ``bad'' sets: Proof of Lemma \ref{lemma:T_size}}
\label{subsec:bad_sets}

\begin{proof}[Proof of Lemma \ref{lemma:T_size}]
For a given vertex $i \in [n]$, let $E_i$ be the event that $i$ is a singleton in~$G_1 \land_{\pi_*} G_2$. 
Throughout the proof, we assume that the communities are approximately balanced; 
specifically, we assume that the event 
$\cH : = \{ n/2 - n^{3/4} \le |V^+ |, |V^-| \le n/ 2 + n^{3/4} \}$ holds. 
Note that $\p(\cH) = 1-o(1)$ by Lemma~\ref{lemma:F}.

Conditioning on $\boldsymbol{\sigma_{*}^{1}}$, 
if $i \in V_1^+$, then we have that 
\begin{align}
\p \left( E_i \, \middle| \, \boldsymbol{\sigma_{*}^{1}} \right) \mathbf{1}(\cH) & = \left( 1 - s^2 \alpha \frac{\log n}{n} \right)^{ |V_1^+ | - 1} \left( 1 - s^2 \beta \frac{\log n}{n} \right)^{ |V_1^-|} \mathbf{1}( \cH) \nonumber \\
& \le \mathrm{exp} \left( - s^2 \frac{\log n}{n} ( \alpha ( |V_1^+| - 1) + \beta |V_1^-| ) \right) \mathbf{1}(\cH) \nonumber \\
\label{eq:pei_upper_bound}
& = \mathrm{exp} \left( - (1 - o(1)) s^2 \left( \frac{\alpha + \beta}{2} \right) \log n  \right) \mathbf{1} (\cH ) 
= n^{- s^2 (\alpha + \beta) / 2 + o(1)} \mathbf{1}( \cH),
\end{align}
where in the inequality on the second line we used that $1 - x \le e^{-x}$, 
and subsequently we used that 
$|V_{1}^{+}|, |V_{1}^{-}| = (1+o(1))n/2$ on the event $\cH$. The bound in~\eqref{eq:pei_upper_bound} implies that
\[
\E \left[ | R_* \cap V_1^+ | \, \middle| \, \boldsymbol{\sigma_{*}^{1}} \right] \mathbf{1} ( \cH) 
= \sum\limits_{i \in V_1^+} \p \left( E_i \, \middle| \, \boldsymbol{\sigma_{*}^{1}} \right) \mathbf{1} ( \cH)  \le | V_1^+ | n^{- s^2 ( \alpha + \beta) / 2 + o(1)} \mathbf{1} ( \cH) 
\leq n^{1 - s^2( \alpha + \beta) / 2 + o(1)}.
\]
Markov's inequality now implies that
\begin{equation}
\label{eq:T_upper_bound}
\p \left( | R_* \cap V_1^+ | \ge n^{1 - s^2 ( \alpha + \beta ) / 2 + \delta} \, \middle| \, \boldsymbol{\sigma_{*}^{1}} \right) \mathbf{1}( \cH)  
\le n^{-\delta + o(1)} = o(1),
\end{equation}
which in turn implies the unconditional bound 
$\p ( | R_* \cap V_1^+ | \ge n^{1 - s^2 (\alpha + \beta) / 2 + \delta} ) = o(1)$ 
as well. 

We now focus on deriving a probabilistic \emph{lower} bound for $|R_* \cap V_1^+|$. To this end, we start by deriving a matching lower bound for $\p \left(E_i \, \middle| \, \boldsymbol{\sigma_{*}^{1}} \right)$. A useful fact we shall use is that by Taylor's theorem, $\log(1 - x) \ge - (1 + \epsilon) x$ provided $0 < x < \epsilon / (1 + \epsilon)$. Thus for $\epsilon = \epsilon_n$ suitably small,
\begin{align*}
\left( \log \p \left( E_i \, \middle| \, \boldsymbol{\sigma_{*}^{1}} \right) \right) \mathbf{1}(\cH) & \ge \left( |V_1^+ | \log \left( 1 - s^2 \alpha \frac{\log n}{n} \right) + |V_1^-| \log \left( 1 - s^2 \beta \frac{\log n}{n} \right) \right) \mathbf{1} ( \cH)  \\
& \ge - (1 + \epsilon) s^2 \frac{\log n}{n} \left( \alpha |V_1^+| + \beta |V_2^-| \right) \mathbf{1} ( \cH)  \\
& = - (1 - o(1)) (1 + \epsilon)s^2 \left( \frac{\alpha + \beta}{2} \right) (\log n ) \mathbf{1}( \cH) .
\end{align*}
Since we used Taylor's theorem for $x = O( \log(n) / n)$, the inequality holds for $\epsilon = n^{-1/2}$. Hence
\begin{equation}
\label{eq:T_expectation_lower_bound}
\E \left[ | R_* \cap V_1^+ | \, \middle| \, \boldsymbol{\sigma_{*}^{1}} \right] \mathbf{1}( \cH)  = \sum\limits_{i \in V_1^+} \p \left( E_i \, \middle| \, \boldsymbol{\sigma_{*}^{1}} \right)  \mathbf{1} ( \cH) 
\ge n^{1 - s^2 ( \alpha + \beta) / 2 - o(1) } \mathbf{1} ( \cH).
\end{equation}
We next upper bound the variance of $|R_* \cap V_1^+|$. For distinct $i,j \in V_1^+$, we have that
\begin{align*}
\mathrm{Cov} \left( \mathbf{1}( E_i), \mathbf{1}(E_j) \, \middle| \, \boldsymbol{\sigma_{*}^{1}} \right)  & =  \p \left( E_i \cap E_j \, \middle| \, \boldsymbol{\sigma_{*}^{1}} \right) - \p \left(E_i \, \middle| \, \boldsymbol{\sigma_{*}^{1}} \right) \p \left( E_j \, \middle| \, \boldsymbol{\sigma_{*}^{1}} \right)   \\
& = \left( 1 - s^2 \alpha \frac{\log n}{n} \right)^{ 2 | V^+ | - 3} \left( 1 - s^2 \beta \frac{\log n}{n} \right)^{2 | V^-|} \left( 1 - \left( 1 - s^2\alpha \frac{\log n}{n} \right) \right).
\end{align*}
Notice that on the event $\cH$, we have that $2 | V^+| - 3 = (1 + o(1)) n$ and $2 | V^- | = (1 + o(1)) n$. Hence, using the inequality $1 - x \le e^{-x}$, we have that
\begin{align}
\mathrm{Cov} \left( \mathbf{1} ( E_i), \mathbf{1}(E_j) \, \middle| \, \boldsymbol{\sigma_{*}^{1}} \right) \mathbf{1} ( \cH) & \le s^2 \alpha \frac{\log n}{n} \mathrm{exp} \left( - (1 - o(1)) s^2 ( \alpha + \beta) \log n \right) \mathbf{1}(\cH) \nonumber \\
\label{eq:cov_eiej_upper_bound}
& = n^{-1-s^2(\alpha + \beta) + o(1)} \mathbf{1}( \cH).
\end{align}
It follows that 
\begin{align}
\mathrm{Var} \left( | R_* \cap V_1^+ | \, \middle| \, \boldsymbol{\sigma_{*}^{1}} \right) \mathbf{1}(\cH) 
& = \mathrm{Var} \left( \sum\limits_{i \in V_1^+ } \mathbf{1}(E_i) \, \middle| \, \boldsymbol{\sigma_{*}^{1}} \right) \mathbf{1}(\cH)  
= \sum\limits_{i,j \in V_1^+} \mathrm{Cov} \left( \mathbf{1} ( E_i), \mathbf{1}(E_j) \, \middle| \, \boldsymbol{\sigma_{*}^{1}}\right) \mathbf{1}(\cH) \nonumber \\
& \le \sum\limits_{i \in V_1^+ } \p \left( E_i \, \middle| \, \boldsymbol{\sigma_{*}^{1}} \right) \mathbf{1}( \cH) + \sum\limits_{i,j \in V_1^+ : i \neq j} \mathrm{Cov} \left( \mathbf{1}( E_i), \mathbf{1}( E_j) \, \middle| \, \boldsymbol{\sigma_{*}^{1}} \right)  \mathbf{1}(\cH) \nonumber \\
& \leq \left( n^{1 - s^2(\alpha + \beta) / 2 + o(1)} + n^{1 - s^2 (\alpha + \beta) + o(1)} \right) \mathbf{1}( \cH) \nonumber  \\
\label{eq:T_second_moment_upper_bound}
& = n^{1 - s^2 (\alpha + \beta) / 2 + o(1)} \mathbf{1}( \cH).
\end{align}
Above, we have used \eqref{eq:pei_upper_bound} and \eqref{eq:cov_eiej_upper_bound} to bound the terms of the summations in the second line. 
The bounds in~\eqref{eq:T_expectation_lower_bound} and~\eqref{eq:T_second_moment_upper_bound}, along with the Paley-Zygmund inequality, imply that
\begin{align}
\p  \left( |R_* \cap V_1^+ | \ge n^{1 - s^2(\alpha + \beta) / 2- \delta}  \, \middle| \, \boldsymbol{\sigma_{*}^{1}} \right) \mathbf{1}(\cH) 
& \ge \left (1 - n^{-\delta + o(1)}\right )^2 \frac{ \E \left[ | R_* \cap V_1^+ | \, \middle| \, \boldsymbol{\sigma_{*}^{1}} \right]^2 }{ \E \left[ | R_* \cap V_1^+ |^2 \, \middle| \, \boldsymbol{\sigma_{*}^{1}} \right]} \mathbf{1} ( \cH) \nonumber \\
&\geq \left( 1 - n^{- \delta + o(1)} \right)^2 \left( 1 - \frac{ \mathrm{Var} \left( | R_* \cap V_1^+ | \, \middle| \, \boldsymbol{\sigma_{*}^{1}} \right) }{\E \left[ |R_* \cap V_1^+ | \, \middle| \, \boldsymbol{\sigma_{*}^{1}} \right]^2} \right) \mathbf{1} ( \cH) \nonumber \\
& \ge  \left (1 - n^{-\delta + o(1)} \right)^2 \left(1 - n^{- \left(1 - s^2(\alpha + \beta) / 2 \right) + o(1)} \right) \mathbf{1}( \cH) \nonumber \\ 
\label{eq:T_lower_bound}
& = (1 - o(1))\mathbf{1}(\cH).
\end{align}
Together, \eqref{eq:T_upper_bound} and \eqref{eq:T_lower_bound}, along with $\p(\cH) = 1 - o(1)$, show that
\[
\p \left( n^{ 1 - s^2 (\alpha + \beta)/ 2 - \delta} \le | R_* \cap V_1^+ | \le n^{1 - s^2 (\alpha + \beta) / 2 + \delta}  \right) = 1 - o(1).
\]
Identical arguments show that the same holds when $|R_* \cap V_1^+|$ is replaced with $|R_* \cap V_1^-|$. 

We now study the size of $|\overline{R}_* \cap V_1^+|$. 
Since $\overline{R}_* \supseteq R_*$, 
with probability $1 - o(1)$ we have the lower bound 
$|\overline{R}_* \cap V_1^+ | \ge | R_* \cap V_1^+ | \geq n^{1 - s^2 ( \alpha + \beta ) / 2 - \delta}$. 
To establish an upper bound, we use the bound 
\[
| \overline{R}_* \cap V_1^+ |
\leq | \overline{R}_* |
\leq \sum_{i \in R_{*}} \left( 1 + | \cN_{2}(\pi_{*}(i)) | \right) 
\leq | \overline{R}_* | \left( 1 + \max_{j \in [n]} | \cN_{2}(j)| \right).
\]
In light of Lemma \ref{lemma:vertex_degree_bound}, we have that $\max_{j \in [n]} | \cN_2(j) | \le 100s \max\{\alpha, \beta\} \log n$ with probability at least $1 - o(1)$. Hence, with probability at least $1 - o(1)$, we have that 
\begin{align*}
|\overline{R}_* \cap V_1^+ | 
& \le \left( 1 + 100 s \max \{ \alpha, \beta \} \log n \right) | R_* | \\
&\leq 2 (1 + 100 s \max \{ \alpha, \beta\} \log n ) n^{1 - s^2 ( \alpha + \beta ) / 2 + \delta} 
\le n^{1 - s^2 ( \alpha + \beta) / 2 + 2 \delta}.
\end{align*}
Identical steps show that $|\overline{R}_* \cap V_1^-| \le n^{1 - s^2 ( \alpha + \beta ) / 2 + 2 \delta}$ with probability $1 - o(1)$ as well. 
\end{proof}

\subsection{First and second moment estimates for $W_i$: Proofs of Lemmas \ref{lemma:W_first_moment} and \ref{lemma:W_second_moment}}
\label{subsec:moment_estimates}

We first prove some useful intermediate results. Our first result establishes some useful conditional independence properties given the sigma algebra $\cI$. 

\begin{lemma}
\label{lemma:I_consequences}
The following hold: 
\begin{enumerate}

\item The sets $\overline{R}_*$, $\overline{R}_* \cap V_1^+$, and $\overline{R}_* \cap V_1^-$ are $\cI$-measurable.

\item \label{item:A_I_independence}
Let $i \in R_*$. 
Conditioned on $\cI$, $\{A_{i,j} : \{i, j \} \in \cE_{10} \}$ is a collection of mutually independent random variables where 
\begin{equation}
\label{eq:A_distribution}
A_{i,j} \sim \begin{cases}
\mathrm{Bern} \left( \alpha \frac{\log n}{n} \right) &\text{ if } \sigma_*^1(i) = \sigma_*^1(j), \\
\mathrm{Bern} \left( \beta \frac{\log n}{n} \right) &\text{ if } \sigma_*^1(i) = - \sigma_*^1(j).
\end{cases}
\end{equation}

\item \label{item:S_independence}
The random variables $\mathbf{1}(i \in S_*)$ and $\sum_{j \in [n] \setminus \overline{R}_*} A_{i,j} \sigma_*^1(j)$ are conditionally independent given~$\cI$.
\end{enumerate}
\end{lemma}

\begin{proof}
From the formula $\overline{R}_* = R_* \cup \pi_*^{-1} ( \cN_2(\pi_*(R_*)) )$ provided in Definition \ref{def:S}, it is clear that $\overline{R}_*$ is $\cI$-measurable since it depends only on $R_*$, $\pi_*$, and $B$. The sets $\overline{R}_* \cap V_1^+$ and $\overline{R}_* \cap V_1^-$ can be readily obtained from $\overline{R}_*$ and $\boldsymbol{\sigma_{*}^{1}}$. 

We now prove Item \#\ref{item:A_I_independence}. 
Notice that the sets $R_*$ and $\overline{R}_*$ depend only on $G_1 \land_{\pi_*} G_2$, $G_2$, and~$\pi_{*}$. 
Thus, by Lemma~\ref{lemma:random_partition}, 
$G_1 \setminus_{\pi_*} G_2$ is conditionally independent of $R_*$ and $\overline{R}_*$ 
given $\pi_{*}$, $\boldsymbol{\sigma_{*}}$, and the partition $\{\cE_{00}, \cE_{01}, \cE_{10}, \cE_{11} \}$. 
In particular, the collection $\{A_{i, j} : \{i,j \} \in \cE_{10} \}$ is conditionally independent of $\cI$, and~\eqref{eq:A_distribution} follows.

Finally, we prove Item \#\ref{item:S_independence}. By Definition \ref{def:S}, we can write
$$
\mathbf{1}(i \in S_*) = \mathbf{1}(i \in R_*) \mathbf{1}( A_{i,j} = 0 \text{ for all } j \in \overline{R}_* ).
$$
In particular, $\mathbf{1}(i \in S_*)$ is measurable with respect to the sigma-algebra generated by $\cI$ and the collection $\cC_1 : = \{ A_{i,j} : j \in \overline{R}_* \text{ and } \{i,j \} \in \cE_{10} \}$. 
On the other hand, since $\overline{R}_*$ is $\cI$-measurable, $\sum_{j \in [n] \setminus \overline{R}_*} A_{i,j} \sigma_*^1(j)$ is measurable with respect to the sigma-algebra generated by $\cI$ and the collection $\cC_2 : = \{ A_{i,j}: j \in [n] \setminus \overline{R}_* \text{ and } \{ i,j \} \in \cE_{10} \}$. Since $\cC_1 \cap \cC_2 = \emptyset$, Item~\#\ref{item:A_I_independence} implies that the two random variables are conditionally independent given~$\cI$.
\end{proof}

Our next result shows that, with high probability, vertices in $R_*$ are also in $S_*$. 

\begin{lemma}
\label{lemma:S_size}
For any $\delta > 0$ and $i \in R_*$, it holds for sufficiently large $n$ that 
$$
\p \left( i \in S_* \, \middle| \, \cI \right) \mathbf{1}(\cG_\delta) 
\ge \left(1 - n^{- s^2 \dconn(\alpha, \beta) + 2\delta} \right) \mathbf{1}(\cG_\delta).
$$
\end{lemma}

\begin{proof}
For $i \in R_*$, define the following $\cI$-measurable random sets:
\begin{align*}
\cC^+(i) & : = \left \{ j \in \overline{R}_* : \{i,j \} \in \cE_{10} \cap \cE^+(\boldsymbol{\sigma_{*}^{1}}) \right \}, \\
\cC^-(i) & : = \left \{ j \in \overline{R}_* : \{i,j \} \in \cE_{10} \cap \cE^-(\boldsymbol{\sigma_{*}^{1}}) \right \}.
\end{align*}
By Definition \ref{def:S}, $i \in S_*$ if and only if $i \in R_*$ and $A_{i,j} = 0$ for all $j \in \overline{R}_*$. 
Phrased differently, $i \in S_*$ if and only if 
$i \in R_*$ and 
$A_{i,j} = 0$ for all $j \in \cC^+(i) \cup \cC^-(i)$. Item~\#\ref{item:A_I_independence} of Lemma~\ref{lemma:I_consequences} then implies, for $i \in R_{*}$, that   
\begin{align*}
\p \left( i \in S_* \, \middle| \,  \cI \right) 
& = \p \left( A_{i,j} = 0 \text{ for all } j \in \cC^+(i) \cup \cC^-(i) \, \middle| \, \cI \right) 
= \left( 1 - \alpha \frac{\log n}{n} \right)^{ |\cC^+(i)|} \left( 1 - \beta \frac{\log n}{n} \right)^{| \cC^-(i) |} \\
& \ge \left( 1 - \alpha | \cC^+(i) | \frac{\log n}{n} \right) \left( 1 - \beta | \cC^-(i) | \frac{\log n}{n} \right) 
 \ge 1 - (\alpha | \cC^+(i) | + \beta | \cC^-(i) | ) \frac{\log n}{n}.
\end{align*}
Above, the first inequality is due to Bernoulli's inequality. 
To further simplify the lower bound, note that $| \cC^+(i)|, | \cC^-(i)| \le | \overline{R}_*$|, which is at most $2 n^{1 - s^2 \dconn (\alpha, \beta) + \delta}$ on the event $\cG_\delta$. Thus
\[
\p \left( i \in S_* \, \middle| \, \cI \right) \mathbf{1}(\cG_\delta) \ge \left( 1 - 2(\alpha + \beta) (\log n) n^{- s^2 \dconn(\alpha, \beta) + \delta} \right) \mathbf{1}(\cG_\delta).
\]
Since $2(\alpha + \beta)\log n \le n^{\delta}$ for sufficiently large $n$, the desired result follows. 
\end{proof}

Next, for $i \in R_*$, we define the random variable
$$
X_i : = \begin{cases}
\p \left( \sum\limits_{j \in [n] \setminus \overline{R}_*} A_{i,j} \sigma_*^1(j) < 0  \, \middle| \, \cI \right) & \text{ if } i \in R_* \cap V_1^+, \\
\p \left( \sum\limits_{j \in [n] \setminus \overline{R}_*} A_{i,j} \sigma_*^1(j) > 0 \, \middle| \, \cI \right) & \text{ if } i \in R_* \cap V_1^-.
\end{cases}
$$
As we shall see in the proofs of Lemmas \ref{lemma:W_first_moment} and \ref{lemma:W_second_moment}, we can bound the first and second moments of $\sum_{i \in R_* \cap V_1^+} W_i$ and $\sum_{i \in R_* \cap V_1^-} W_i$ by functions of the $X_i$'s. Our next result characterizes the behavior of the $X_i$'s on the high-probability $\cI$-measurable event $\cF \cap \cG_\delta$. 

\begin{lemma}
\label{lemma:Xi}
For $i \in R_*$, we have that 
\[
X_i \mathbf{1}(\cF \cap \cG_\delta) = n^{ - s(1 - s) \dch(\alpha, \beta) + o(1)} \mathbf{1}( \cF \cap \cG_\delta).
\]
\end{lemma}

\begin{proof}
Let $i \in R_* \cap V_1^+$. Since $\overline{R}_*$ and  $\boldsymbol{\sigma_{*}^{1}}$ are $\cI$-measurable, Item~\#\ref{item:A_I_independence} of Lemma~\ref{lemma:I_consequences} implies that
\begin{equation}
\label{eq:maj_distributional_representation}
\sum\limits_{j \in [n] \setminus \overline{R}_*} A_{i,j} \sigma_*^1(j) \stackrel{d}{=} Y - Z,
\end{equation}
where $Y \sim \mathrm{Bin}(m^+, p)$ and $Z \sim \mathrm{Bin}(m^-, q)$ are independent, with $p = \alpha \log(n) / n$, $q = \beta \log(n) / n$, and 
\begin{align*}
m^+ & : = \left | \left \{ j \in [n] \setminus \overline{R}_*: \{i,j \} \in \cE_{10} \cap \cE^+(\boldsymbol{\sigma_{*}^{1}}) \right \} \right|, \\
m^- & : = \left| \left \{ j \in [n] \setminus \overline{R}_* : \{ i,j \} \in \cE_{10} \cap \cE^-(\boldsymbol{\sigma_{*}^{1}}) \right \} \right|.
\end{align*}
We use the shorthand $s_{10} : = s(1 - s)$. On the event $\cF \cap \cG_\delta$, we have the upper bound
$$
m^+ \le \left | \left \{ j : \{ i,j \} \in \cE_{10} \cap \cE^+(\boldsymbol{\sigma_{*}^{1}}) \right \} \right| \le s_{10} \left( \frac{n}{2} + 2n^{3/4} \right) =  (1 + o(1)) s_{10} \frac{n}{2}
$$
and the matching lower bound
$$
m^+ \ge \left | \left \{ j : \{i,j \} \in \cE_{10} \cap \cE^+(\boldsymbol{\sigma_{*}^{1}}) \right \} \right |  - | \overline{R}_* | \ge s_{10} \left( \frac{n}{2} - 2n^{3/4} \right) - n^{1 - s^2 \dconn(\alpha, \beta) + \delta} = (1 - o(1)) s_{10}  \frac{n}{2}.
$$
Similarly, we can show that on the event $\cF\cap \cG_\delta$, we have that 
$$
(1 - o(1)) s_{10} \frac{n}{2} \le m^- \le (1 + o(1)) s_{10} \frac{n}{2}.
$$
Noting that $p = (1 + o(1)) s_{10} \alpha \log (s_{10} n) / (s_{10} n)$ and $q = (1 + o(1)) s_{10} \beta \log( s_{10} n) / (s_{10} n)$, Lemma~\ref{lemma:majority_probability} and \eqref{eq:maj_distributional_representation} now imply that for $i \in R_* \cap V_1^+$, we have that 
\begin{equation}
\label{eq:maj_prob_bound}
\p \left( \sum\limits_{j \in [n] \setminus \overline{T}_*} A_{i,j} \sigma_*^1(j) < 0 \, \middle| \, \cI \right) \mathbf{1}(\cF \cap \cG_\delta) 
= \p \left( Y < Z \, \middle| \, \cI \right) \mathbf{1}(\cF \cap \cG_\delta) 
= n^{- s_{10} \dch(\alpha, \beta) + o(1)} \mathbf{1}( \cF \cap \cG_\delta).
\end{equation}
An identical analysis can be done for $i \in R_* \cap V_1^-$.
\end{proof}

We now turn to the proofs of the moment estimates. 

\begin{proof}[Proof of Lemma \ref{lemma:W_first_moment}]
As a shorthand, denote $\gamma : = s^2 \dconn(\alpha, \beta) - 2 \delta$. Let $i \in R_* \cap V_1^+$. Then,  
\begin{align*}
\E \left[ W_i \, \middle| \, \cI \right] \mathbf{1}(\cF \cap \cG_\delta) 
&= \p \left ( i \in S_* \text{ and } \sum\limits_{j \in [n] \setminus \overline{R}_*} A_{i,j} \sigma_*^1(j) < 0 \, \middle| \, \cI \right) \mathbf{1}(\cF \cap \cG_\delta) \\
&= \p \left( i \in S_* \, \middle| \, \cI \right) \cdot \p \left( \sum\limits_{j \in [n] \setminus \overline{R}_*} A_{i,j} \sigma_*^1(j) < 0 \, \middle| \, \cI \right) \mathbf{1}( \cF \cap \cG_\delta),
\end{align*}
where the first equality follows since 
$\maj(i) = \sum_{j \in [n]} A_{i,j} \sigma_*^1(j)$ 
and $A_{i,j} = 0$ when $i \in S_*$ and $j \in \overline{R}_*$; the second equality is due to Item~\#\ref{item:S_independence} of Lemma~\ref{lemma:I_consequences}. 
Now using Lemma~\ref{lemma:S_size}, 
the definition of $X_{i}$, 
and Lemma~\ref{lemma:Xi}, 
we have that 
\begin{align}
\E \left[ W_i \, \middle| \, \cI \right] \mathbf{1}(\cF \cap \cG_\delta) 
\label{eq:W_X_lower_bound}
& \ge (1 - n^{- \gamma}) X_i \mathbf{1}( \cF \cap \cG_\delta) \\
\label{eq:impossibility_prob_bound_1}
& \ge (1 - n^{- \gamma} ) n^{ - s(1 - s) \dch(\alpha, \beta) + o(1)} \mathbf{1}( \cF \cap \cG_\delta).
\end{align}
Summing over $i \in R_* \cap V_1^+$, we obtain that
\begin{align*}
\E \left[ \sum\limits_{i \in R_* \cap V_1^+} W_i \, \middle| \, \cI \right] \mathbf{1}( \cF \cap \cG_\delta ) 
& \ge (1 - n^{- \gamma} ) \left( | R_* \cap V_1^+ | n^{ - s(1 - s) \dch(\alpha, \beta) + o(1)} \right) \mathbf{1}(\cF \cap \cG_\delta) \\
& \ge (1 - n^{- \gamma}) n^{1- s^2 \dconn(\alpha, \beta) - s(1 - s)\dch(\alpha, \beta) - \delta + o(1)} \mathbf{1} ( \cF \cap \cG_\delta),
\end{align*}
where the inequality on the second line uses the fact that $| R_* \cap V_1^+ | \ge n^{1 - s^2 \dconn(\alpha, \beta) - \delta}$ on $\cG_\delta$. Identical arguments yield the same lower bound for $\E \left[ \sum_{i \in R_* \cap V_1^-} W_i \, \middle| \, \cI \right] \mathbf{1}(\cF \cap \cG_\delta)$. 
\end{proof}

\begin{proof}[Proof of Lemma \ref{lemma:W_second_moment}]
Fix $i, j \in R_* \cap V_1^+$, where $i \neq j$. 
We have that 
\begin{align*}
\E \left[ W_i W_j \, \middle| \, \cI \right] 
&= \p \left( i, j \in S_* \text{ and } \sum\limits_{k \in [n] \setminus  \overline{R}_*} A_{i,k} \sigma_*^1(k), \sum\limits_{k \in [n] \setminus \overline{R}_*} A_{j,k} \sigma_*^1(k) <0 \, \middle| \, \cI \right) \\
& \le \p \left( \sum\limits_{k \in [n] \setminus  \overline{R}_*} A_{i,k} \sigma_*^1(k), \sum\limits_{k \in [n] \setminus \overline{R}_*} A_{j,k} \sigma_*^1(k) <  0 \, \middle| \, \cI \right)  \\
&=  \p \left( \sum\limits_{k \in [n] \setminus  \overline{R}_*} A_{i,k} \sigma_*^1(k) <  0 \, \middle| \, \cI \right) \p \left( \sum\limits_{k \in [n] \setminus  \overline{R}_*} A_{j,k} \sigma_*^1(k) <  0 \, \middle| \, \cI \right) \\
& = X_i X_j.
\end{align*}
Above, the equality on the first line follows since $A_{i,k} = A_{j,k} = 0$ for $i,j \in S_*$ and $k \in \overline{R}_*$. 
The equality on the third line follows from noting that 
$\sum_{k \in [n] \setminus \overline{R}_*} A_{i,k} \sigma_*^1(k)$ is a measurable function of $\cI$ and $\{ A_{i,k} \}_{k \in [n] \setminus \overline{R}_*}$, 
$\sum_{k \in [n] \setminus \overline{R}_*} A_{j,k} \sigma_*^1(k)$ is a measurable function of $\cI$ and $\{A_{j,k} \}_{k \in [n] \setminus \overline{R}_*}$, 
and the collections 
$\{A_{i,k }\}_{k \in [n] \setminus \overline{R}_*}$ 
and 
$\{ A_{j,k} \}_{k \in [n] \setminus \overline{R}_*}$ 
are conditionally independent given $\cI$ (by Item~\#\ref{item:A_I_independence} of Lemma~\ref{lemma:I_consequences}). 
On the other hand, for the case $i = j$ we have through similar arguments that 
\[
\E \left[ W_i^2 \, \middle| \, \cI \right] 
= \E \left[ W_i \, \middle| \, \cI \right] 
= \p \left( i \in S_* \text{ and } \sum\limits_{k \in [n] \setminus \overline{R}_*} A_{i,k} \sigma_*^1(k) <  0 \, \middle| \, \cI \right) 
\le X_i.
\]
We can then write
\begin{align}
\E \left[ \left( \sum\limits_{i \in R_* \cap V_1^+} W_i \right)^2 \, \middle| \, \cI \right] 
& = \sum\limits_{i \in R_* \cap V_1^+} \E \left[ W_i \, \middle| \, \cI \right] + \sum\limits_{i,j \in R_* \cap V_1^+ : i \neq j} \E \left[ W_i W_j \, \middle| \, \cI \right] \nonumber \\
& \le \sum\limits_{i \in R_* \cap V_1^+} X_i + \sum\limits_{i,j \in R_* \cap V_1^+ : i \neq j} X_i X_j 
\label{eq:W_second_moment}
\le \sum\limits_{i \in R_* \cap V_1^+} X_i + \left( \sum\limits_{i \in R_* \cap V_1^+} X_i \right)^2.
\end{align}
In light of the lower bound in \eqref{eq:W_X_lower_bound}, we also have that 
\begin{equation}
\label{eq:EW_squared_lower_bound}
\E \left[ \sum\limits_{i \in R_* \cap V_1^+} W_i \, \middle| \, \cI \right]^2 \mathbf{1}(\cF \cap \cG_\delta) \ge \left( 1 - n^{-\gamma} \right)^2 \left( \sum\limits_{i \in R_* \cap V_1^+ } X_i \right)^2 \mathbf{1}( \cF \cap \cG_\delta),
\end{equation}
where we recall that $\gamma : = s^2 \dconn(\alpha, \beta) - 2 \delta$. 
On the event $\cF \cap \cG_\delta$, we also have the upper bound 
\begin{equation}
\left( \sum\limits_{i \in R_* \cap V_1^+} X_i \right) \mathbf{1} (\cF \cap \cG_\delta) 
= | R_* \cap V_1^+ | n^{ - s(1 - s) \dch(\alpha, \beta) + o(1)} \mathbf{1} ( \cF \cap \cG_\delta) 
\le n^{ \theta + \delta + o(1)} \mathbf{1} (\cF \cap \cG_\delta),
\end{equation}
where the first equality is due to Lemma~\ref{lemma:Xi} and we recall that $\theta : = 1 - s^2 \dconn(\alpha, \beta) - s(1 - s) \dch(\alpha, \beta)$. Combining~\eqref{eq:W_second_moment} and~\eqref{eq:EW_squared_lower_bound} and using $(1 - n^{- \gamma})^2 \ge 1 - 2 n^{- \gamma}$ shows that
\begin{align*}
\mathrm{Var} \left( \sum\limits_{i \in R_* \cap V_1^+} W_i \, \middle| \, \cI \right) \mathbf{1}( \cF \cap \cG_\delta) 
& \le \left( \sum\limits_{i \in R_* \cap V_1^+ } X_i + 2n^{-\gamma} \left( \sum\limits_{i \in R_* \cap V_1^+ } X_i \right)^2 \right) \mathbf{1}(\cF \cap \cG_\delta) \\
& \le \left( n^{\theta + \delta + o(1)} + 2n^{ 2 \theta + 2 \delta - \gamma + o(1)} \right) \mathbf{1} ( \cF \cap \cG_\delta).
\end{align*}
To simplify the bound on the right hand side further, 
notice that if $\delta < \theta / 4$ and $\delta \le \gamma / 6$ (i.e., $\delta \leq s^{2} \dconnab /8$), 
then we have 
$\theta + \delta + o(1) < 2 \theta - 3 \delta$ 
and $2 \theta + 2 \delta - \gamma + o(1) < 2 \theta - 3 \delta$ for $n$ large enough. Hence
\[
\mathrm{Var} \left( \sum\limits_{i \in R_* \cap V_1^+ } W_i \, \middle| \, \cI \right) \mathbf{1} ( \cF \cap \cG_\delta) \le n^{2 \theta - 3 \delta} \mathbf{1}(\cF \cap \cG_\delta).
\]
Through identical steps, the same result holds for $\sum_{i \in R_* \cap V_1^-} W_i$.
\end{proof}




{\small
\bibliographystyle{abbrv}
\bibliography{references}

\begin{thebibliography}{10}

\bibitem{Abbe_survey}
E.~Abbe.
\newblock Community detection and stochastic block models: recent developments.
\newblock {\em Journal of Machine Learning Research}, 18(1):6446--6531, 2017.

\bibitem{abbe2016exact}
E.~Abbe, A.~S. Bandeira, and G.~Hall.
\newblock Exact recovery in the stochastic block model.
\newblock {\em IEEE Transactions on Information Theory}, 62(1):471--487, 2016.

\bibitem{abbe2020ell_p}
E.~Abbe, J.~Fan, and K.~Wang.
\newblock An $\ell_p$ theory of {PCA} and spectral clustering.
\newblock Preprint available at \url{https://arxiv.org/abs/2006.14062}, 2020.

\bibitem{Abbe2020}
E.~Abbe, J.~Fan, K.~Wang, and Y.~Zhong.
\newblock Entrywise eigenvector analysis of random matrices with low expected
  rank.
\newblock {\em Annals of Statistics}, 48(3):1452--1474, 2020.

\bibitem{abbe2015community}
E.~Abbe and C.~Sandon.
\newblock Community detection in general stochastic block models: Fundamental
  limits and efficient algorithms for recovery.
\newblock In {\em 2015 IEEE 56th Annual Symposium on Foundations of Computer
  Science (FOCS)}, pages 670--688, 2015.

\bibitem{ali2019latent}
H.~T. Ali, S.~Liu, Y.~Yilmaz, R.~Couillet, I.~Rajapakse, and A.~Hero.
\newblock Latent heterogeneous multilayer community detection.
\newblock In {\em 2019 IEEE International Conference on Acoustics, Speech and
  Signal Processing (ICASSP)}, pages 8142--8146, 2019.

\bibitem{arroyo2020inference}
J.~Arroyo, A.~Athreya, J.~Cape, G.~Chen, C.~E. Priebe, and J.~T. Vogelstein.
\newblock {Inference for multiple heterogeneous networks with a common
  invariant subspace}.
\newblock {\em Journal of Machine Learning Research}, 22(142):1--49, 2021.

\bibitem{barak2019}
B.~Barak, C.-N. Chou, Z.~Lei, T.~Schramm, and Y.~Sheng.
\newblock {(Nearly) Efficient Algorithms for the Graph Matching Problem on
  Correlated Random Graphs}.
\newblock In {\em Advances in Neural Information Processing Systems (NeurIPS)},
  pages 9190--9198, 2019.

\bibitem{bhattacharya2020consistent}
S.~Bhattacharyya and S.~Chatterjee.
\newblock {Consistent Recovery of Communities from Sparse Multi-relational
  Networks: A Scalable Algorithm with Optimal Recovery Conditions}.
\newblock In {\em Complex Networks XI}, pages 92--103, 2020.

\bibitem{binkiewicz2017covariate}
N.~Binkiewicz, J.~T. Vogelstein, and K.~Rohe.
\newblock Covariate-assisted spectral clustering.
\newblock {\em Biometrika}, 104(2):361--377, 2017.

\bibitem{bopanna1987eigenvalues}
R.~B. Boppana.
\newblock Eigenvalues and graph bisection: An average-case analysis.
\newblock In {\em 28th Annual Symposium on Foundations of Computer Science
  (FOCS)}, pages 280--285, 1987.

\bibitem{bordenave2015nonbacktracking}
C.~Bordenave, M.~Lelarge, and L.~Massoulié.
\newblock Non-backtracking spectrum of random graphs: Community detection and
  non-regular ramanujan graphs.
\newblock In {\em 2015 IEEE 56th Annual Symposium on Foundations of Computer
  Science (FOCS)}, pages 1347--1357, 2015.

\bibitem{bothorel2015clustering}
C.~Bothorel, J.~D. Cruz, M.~Magnani, and B.~Micenkov{\'a}.
\newblock Clustering attributed graphs: models, measures and methods.
\newblock {\em Network Science}, 3(3):408--444, 2015.

\bibitem{bui1984graph}
T.~Bui, S.~Chaudhuri, T.~Leighton, and M.~Sipser.
\newblock {Graph Bisection Algorithms With Good Average Case Behavior}.
\newblock In {\em 25th Annual Symposium on Foundations of Computer Science
  (FOCS)}, pages 181--192, 1984.

\bibitem{chen2020global}
S.~Chen, S.~Liu, and Z.~Ma.
\newblock {Global and Individualized Community Detection in Inhomogeneous
  Multilayer Networks}.
\newblock Preprint available at \url{https://arxiv.org/abs/2012.00933}, 2020.

\bibitem{cullina2016improved}
D.~Cullina and N.~Kiyavash.
\newblock {Improved Achievability and Converse Bounds for Erd\H{o}s-R\'enyi
  Graph Matching}.
\newblock In {\em ACM SIGMETRICS}, volume~44, pages 63--72, 2016.

\bibitem{cullina2018exact}
D.~Cullina and N.~Kiyavash.
\newblock {Exact alignment recovery for correlated Erd\H{o}s-R\'enyi graphs}.
\newblock Preprint available at \url{https://arxiv.org/abs/1711.06783}, 2018.

\bibitem{cullina2020partial}
D.~Cullina, N.~Kiyavash, P.~Mittal, and H.~V. Poor.
\newblock {Partial Recovery of Erd\H{o}s-R\'{e}nyi Graph Alignment via k-Core
  Alignment}.
\newblock {\em SIGMETRICS Perform. Eval. Rev.}, 48(1):99–100, July 2020.

\bibitem{cullina2016simultaneous}
D.~Cullina, K.~Singhal, N.~Kiyavash, and P.~Mittal.
\newblock On the simultaneous preservation of privacy and community structure
  in anonymized networks.
\newblock Preprint available at \url{https://arxiv.org/abs/1603.08028}, 2016.

\bibitem{DKMZ11}
A.~Decelle, F.~Krzakala, C.~Moore, and L.~Zdeborov{\'a}.
\newblock Asymptotic analysis of the stochastic block model for modular
  networks and its algorithmic applications.
\newblock {\em Physical Review E}, 84(6):066106, 2011.

\bibitem{deshpande2018contextual}
Y.~Deshpande, S.~Sen, A.~Montanari, and E.~Mossel.
\newblock {Contextual Stochastic Block Models}.
\newblock In {\em Advances in Neural Information Processing Systems (NeurIPS)},
  pages 8581--8593, 2018.

\bibitem{ding2021efficient}
J.~Ding, Z.~Ma, Y.~Wu, and J.~Xu.
\newblock Efficient random graph matching via degree profiles.
\newblock {\em Probability Theory and Related Fields}, 179(1):29--115, 2021.

\bibitem{dyer1989solution}
M.~Dyer and A.~Frieze.
\newblock The solution of some random {NP}-hard problems in polynomial expected
  time.
\newblock {\em Journal of Algorithms}, 10(4):451--489, 1989.

\bibitem{fan2020spectral}
Z.~Fan, C.~Mao, Y.~Wu, and J.~Xu.
\newblock Spectral graph matching and regularized quadratic relaxations:
  Algorithm and theory.
\newblock In {\em Proceedings of the 37th International Conference on Machine
  Learning (ICML)}, volume 119 of {\em Proceedings of Machine Learning
  Research}, pages 2985--2995. PMLR, 13--18 Jul 2020.

\bibitem{ganassali2020tree}
L.~Ganassali and L.~Massouli\'e.
\newblock From tree matching to sparse graph alignment.
\newblock In {\em Proceedings of the Thirty Third Conference on Learning Theory
  (COLT)}, volume 125 of {\em Proceedings of Machine Learning Research}, pages
  1633--1665. PMLR, 09--12 Jul 2020.

\bibitem{ganassali2021impossibility}
L.~Ganassali, L.~Massouli{\'e}, and M.~Lelarge.
\newblock {Impossibility of Partial Recovery in the Graph Alignment Problem}.
\newblock In {\em Conference on Learning Theory}, pages 2080--2102. PMLR, 2021.

\bibitem{hall2020partial}
G.~Hall and L.~Massouli{\'e}.
\newblock {Partial Recovery in the Graph Alignment Problem}.
\newblock Preprint available at \url{https://arxiv.org/abs/2007.00533}, 2020.

\bibitem{han2015consistent}
Q.~Han, K.~Xu, and E.~Airoldi.
\newblock Consistent estimation of dynamic and multi-layer block models.
\newblock In {\em International Conference on Machine Learning (ICML)}, pages
  1511--1520. PMLR, 2015.

\bibitem{HLL83}
P.~W. Holland, K.~B. Laskey, and S.~Leinhardt.
\newblock Stochastic blockmodels: First steps.
\newblock {\em Social {N}etworks}, 5(2):109--137, 1983.

\bibitem{kanade2016global}
V.~Kanade, E.~Mossel, and T.~Schramm.
\newblock {Global and Local Information in Clustering Labeled Block Models}.
\newblock {\em IEEE Transactions on Information Theory}, 62(10):5906--5917,
  2016.

\bibitem{korula2014efficient}
N.~Korula and S.~Lattanzi.
\newblock An efficient reconciliation algorithm for social networks.
\newblock {\em Proceedings of the VLDB Endowment}, 7(5):377--388, 2014.

\bibitem{lei2019consistent}
J.~Lei, K.~Chen, and B.~Lynch.
\newblock {Consistent community detection in multi-layer network data}.
\newblock {\em Biometrika}, 107(1):61--73, 12 2019.

\bibitem{lu2020contextual}
C.~Lu and S.~Sen.
\newblock Contextual stochastic block model: Sharp thresholds and contiguity.
\newblock Preprint available at \url{https://arxiv.org/abs/2011.09841}, 2020.

\bibitem{Luczak1991}
T.~\L{}uczak.
\newblock Size and connectivity of the k-core of a random graph.
\newblock {\em Discrete Mathematics}, 91(1):61--68, 1991.

\bibitem{lyzinski2018information}
V.~Lyzinski.
\newblock {Information Recovery in Shuffled Graphs via Graph Matching}.
\newblock {\em IEEE Transactions on Information Theory}, 64(5):3254--3273,
  2018.

\bibitem{ma2021community}
Z.~Ma and S.~Nandy.
\newblock {Community Detection with Contextual Multilayer Networks}.
\newblock Preprint available at \url{https://arxiv.org/abs/2104.02960}, 2021.

\bibitem{mao2021exact}
C.~Mao, M.~Rudelson, and K.~Tikhomirov.
\newblock Exact matching of random graphs with constant correlation.
\newblock Preprint available at \url{https://arxiv.org/abs/2110.05000}, 2021.

\bibitem{mao2021random}
C.~Mao, M.~Rudelson, and K.~Tikhomirov.
\newblock {Random Graph Matching with Improved Noise Robustness}.
\newblock In {\em Proceedings of the 34th Conference on Learning Theory
  (COLT)}, pages 3296--3329. PMLR, 2021.

\bibitem{massoulie2014community}
L.~Massouli{\'e}.
\newblock Community detection thresholds and the weak {R}amanujan property.
\newblock In {\em Proceedings of the 46th Annual ACM Symposium on Theory of
  Computing (STOC)}, pages 694--703. ACM, 2014.

\bibitem{mayya2019mutual}
V.~Mayya and G.~Reeves.
\newblock Mutual information in community detection with covariate information
  and correlated networks.
\newblock In {\em 2019 57th Annual Allerton Conference on Communication,
  Control, and Computing (Allerton)}, pages 602--607, 2019.

\bibitem{mossel2014reconstruction}
E.~Mossel, J.~Neeman, and A.~Sly.
\newblock Reconstruction and estimation in the planted partition model.
\newblock {\em Probability Theory and Related Fields}, 162, 07 2014.

\bibitem{mossel2016consistency}
E.~Mossel, J.~Neeman, and A.~Sly.
\newblock {Consistency thresholds for the planted bisection model}.
\newblock {\em Electronic Journal of Probability}, 21(none):1 -- 24, 2016.

\bibitem{mossel2018proof}
E.~Mossel, J.~Neeman, and A.~Sly.
\newblock A proof of the block model threshold conjecture.
\newblock {\em Combinatorica}, 38(3):665--708, 2018.

\bibitem{mossel2016local}
E.~Mossel and J.~Xu.
\newblock {Local Algorithms for Block Models with Side Information}.
\newblock In {\em Proceedings of the 2016 ACM Conference on Innovations in
  Theoretical Computer Science (ITCS)}, pages 71--80, 2016.

\bibitem{mossel2019seeded}
E.~Mossel and J.~Xu.
\newblock Seeded graph matching via large neighborhood statistics.
\newblock In {\em Proceedings of the Thirtieth Annual ACM-SIAM Symposium on
  Discrete Algorithms (SODA)}, pages 1005--1014, 2019.

\bibitem{onaran2016optimal}
E.~Onaran, S.~Garg, and E.~Erkip.
\newblock Optimal de-anonymization in random graphs with community structure.
\newblock In {\em 2016 50th Asilomar Conference on Signals, Systems and
  Computers}, pages 709--713. IEEE, 2016.

\bibitem{paul2020spectral}
S.~Paul and Y.~Chen.
\newblock {Spectral and matrix factorization methods for consistent community
  detection in multi-layer networks}.
\newblock {\em The Annals of Statistics}, 48(1):230 -- 250, 2020.

\bibitem{paul2021null}
S.~Paul and Y.~Chen.
\newblock {Null Models and Community Detection in Multi-Layer Networks}.
\newblock {\em Sankhya A}, pages 1--55, 2021.

\bibitem{pedarsani2011privacy}
P.~Pedarsani and M.~Grossglauser.
\newblock On the privacy of anonymized networks.
\newblock In {\em Proceedings of the 17th ACM SIGKDD International Conference
  on Knowledge Discovery and Data Mining (KDD)}, pages 1235--1243, 2011.

\bibitem{poor_book}
H.~V. Poor.
\newblock {\em An Introduction to Signal Detection and Estimation (2nd Ed.)}.
\newblock Springer-Verlag, Berlin, Heidelberg, 1994.

\bibitem{RS21}
M.~Z. R\'{a}cz and A.~Sridhar.
\newblock {Correlated Stochastic Block Models: Exact Graph Matching with
  Applications to Recovering Communities}.
\newblock In {\em Advances in Neural Information Processing Systems (NeurIPS)},
  2021.

\bibitem{RS22}
M.~Z. R{\'a}cz and A.~Sridhar.
\newblock Correlated randomly growing graphs.
\newblock {\em Annals of Applied Probability}, to appear, 2022.

\bibitem{saad2020sideinfo}
H.~Saad and A.~Nosratinia.
\newblock Recovering a single community with side information.
\newblock {\em IEEE Transactions on Information Theory}, 66(12):7939--7966,
  2020.

\bibitem{shirani2021concentration}
F.~Shirani, S.~Garg, and E.~Erkip.
\newblock A concentration of measure approach to correlated graph matching.
\newblock {\em IEEE Journal on Selected Areas in Information Theory},
  2(1):338--351, 2021.

\bibitem{wu2021settling}
Y.~Wu, J.~Xu, and S.~H. Yu.
\newblock {Settling the Sharp Reconstruction Thresholds of Random Graph
  Matching}.
\newblock Preprint available at \url{https://arxiv.org/abs/2102.00082}, 2021.

\bibitem{yan2021covariate}
B.~Yan and P.~Sarkar.
\newblock Covariate regularized community detection in sparse graphs.
\newblock {\em Journal of the American Statistical Association},
  116(534):734--745, 2021.

\bibitem{yu2021power}
L.~Yu, J.~Xu, and X.~Lin.
\newblock {The Power of $D$-hops in Matching Power-Law Graphs}.
\newblock {\em Proceedings of the ACM on Measurement and Analysis of Computing
  Systems}, 5(2):1--43, 2021.

\bibitem{zhang2016community}
Y.~Zhang, E.~Levina, and J.~Zhu.
\newblock Community detection in networks with node features.
\newblock {\em Electronic Journal of Statistics}, 10(2):3153--3178, 2016.

\end{thebibliography}
}




\end{document}